\documentclass[a4paper]{amsart}
\usepackage{amsmath}
\usepackage{amssymb}
\usepackage{amsfonts}
\usepackage{pxfonts}
\usepackage[OT2,T1]{fontenc}

\usepackage{pict2e}
\usepackage{hyperref}
\usepackage[
textwidth=14.5cm,
textheight=21cm,
hmarginratio=1:1,
vmarginratio=1:1]{geometry}
\usepackage{pifont}
\usepackage{graphicx}

\newtheorem{thm}{Theorem}[subsection]
\newtheorem{cor}[thm]{Corollary}
\newtheorem{lem}[thm]{Lemma}
\newtheorem{prop}[thm]{Proposition}
\theoremstyle{definition}
\newtheorem{defn}[thm]{Definition}

\newtheorem{prob}[thm]{Problem}
\theoremstyle{remark}
\newtheorem{rem}[thm]{Remark}

\numberwithin{equation}{subsection}
\numberwithin{figure}{section}

\newcommand{\diff}{\mathrm{d}}
\newcommand{\C}{{\mathbb C}}
\newcommand{\R}{{\mathbb R}}
\newcommand{\D}{{\mathbb D}}
\newcommand{\Te}{{\mathbb T}}

\newcommand{\imag}{\mathrm{i}}
\newcommand{\e}{\mathrm{e}}

\newcommand{\Z}{{\mathbb Z}}
\newcommand{\EP}{M^2}

\newcommand{\supp}{\operatorname{supp}}
\newcommand{\Ordo}{\mathrm{O}}

\newcommand{\Mop}{{\mathbf M}}
\newcommand{\Lop}{{\mathbf L}}
\newcommand{\Tope}{{\mathbf T}}
\newcommand{\Topep}{\pmb{\mathcal{T}}}
\newcommand{\Hop}{{\mathbf H}}
\newcommand{\Jop}{{\mathbf J}}
\newcommand{\Sop}{{\mathbf S}}
\newcommand{\Sopp}{\pmb{\mathcal{S}}}
\newcommand{\Perop}{\boldsymbol\Pi}
\newcommand{\Superop}{\boldsymbol\Sigma}

\newcommand{\Vop}{\mathbf V}

\newcommand{\Gmop}{\mathbf K}
\newcommand{\Lspaceo}{{\mathfrak{L}}(\R)}
\newcommand{\LspaceI}{{\mathfrak{L}}(I)}

\newcommand{\LspaceIone}{{\mathfrak{L}}(I_{1})}

\newcommand{\LspaceJ}{{\mathfrak{L}}(J)}
\newcommand{\ZspaceI}{{\mathfrak{Z}}(\R;I)}
\newcommand{\Lspaces}{{\mathfrak{L}}_0(\R)}
\newcommand{\Lspaceoper}{{\mathfrak{L}}(\R/2\Z)}

\newcommand{\stars}{\circledast}

\newcommand{\proj}{{\mathbf P}}
\newcommand{\pev}{\mathrm{vap}}

\newcommand{\pv}{\operatorname{pv}}
\newcommand{\sign}{\operatorname{sgn}}

\newcommand{\id}{\operatorname{\text{\bf I}}}
\newcommand{\ci}{\operatorname{\mathrm{ci}}}
\newcommand{\si}{\operatorname{\mathrm{si}}}

\newcommand{\calM}{{\mathcal M}}

\newcommand{\calE}{{\mathcal E}}
\newcommand{\calF}{{\mathcal F}}

\newcommand{\bpi}{\boldsymbol\pi}

\DeclareMathOperator{\re}{Re}
\DeclareMathOperator{\im}{Im}

\begin{document}

\title{The Klein-Gordon equation, the Hilbert transform, and dynamics of
Gauss-type maps}

\author{Haakan Hedenmalm}
\address{Hedenmalm: Department of Mathematics\\
KTH Royal Institute of Technology\\
SE--10044 Stockholm\\
Sweden}

\email{haakanh@kth.se}

\author{Alfonso Montes-Rodr\'\i{}guez}
\address{Montes-Rodr\'\i{}guez: Department of Mathematical Analysis\\
University of Sevilla\\
ES--41012 Sevilla\\
Spain}

\email{amontes@us.es}


\thanks{The research of Hedenmalm was supported by Vetenskapsr\aa{}det (VR).
The research of both authors was supported
by the KAW foundation as well as by Plan Nacional ref. MTM2012-35107,  by 
Junta Andaluc\'\i{}a P12-FQM-633}

%\maketitle
$\quad$

\begin{abstract}
We study the uncertainty principle associated with the Klein-Gordon equation. 
As in the previous work published 2011 in the Annals, we consider vanishing 
along a lattice-cross. The following variants appear naturally: (1) 
vanishing only along "half" of the lattice-cross, where the "half" is 
defined as being on the boundary of a quarter-plane, and (2) that 
the function vanishes on the whole lattice-cross, but we ask the function 
to have Fourier transform supported by one of the two branches of the 
hyperbola. In case (1) the critical phenomenon is whether the given 
condition forces the function to vanish on the quarter-plane in question. 
Here it turns out to be crucial whether the quarter-plane is space-like or
time-like, and in short the answer is yes for space-like and no for 
time-like. The analysis brings us quite far, involving the orbit of the 
Hilbert kernel under the iterates of the transfer operator, and to do so uses 
methods from the theory of totally positive matrices as well as Hurwitz zeta 
functions, and is partially postponed to a separate publication. 
In case (2), the critical phenomenon occurs at another density, and the 
dynamics then comes from the standard Gauss transformation $t\mapsto 1/t$ mod
$\Z$ on the interval $[0,1]$. In the 
intermediate range of the density of the lattice-cross, we obtain 
unique extendability of the Fourier transform from one branch of the 
hyperbola to the other. 
\end{abstract}

%%% ----------------------------------------------------------------------
\maketitle
%%% ----------------------------------------------------------------------

%\addtolength{\oddsidemargin}{-1.3cm}
%\addtolength{\evensidemargin}{-1.3cm}
%\addtolength{\textwidth}{1.5cm}
\addtolength{\textheight}{2.2cm}
%\addtolength{\topmargin}{-1.0cm}
%\addtolength{\footskip}{1.5cm}

%\begin{document}

%%% Research support

%\subjclass{}

%\keywords{}

%%% Section numbering
%\setcounter{section}{-1}

%\setcounter{equation}{0}
%\setcounter{thm}{0}
%\setcounter{prop}{0}
%\setcounter{lemma}{0}
%\setcounter{cor}{0}
%\setcounter{remark}{0}

\section{Introduction}

\subsection{Heisenberg uniqueness pairs}
Let $\mu$ be a finite complex-valued Borel measure in the plane $\R^2$,
and associate with it the Fourier transform
\[
\hat\mu(\xi):=\int_{\R^2}\e^{\imag\pi\langle x,\xi\rangle}\diff\mu(x),
\]
where $x=(x_1,x_2)$ and $\xi=(\xi_1,\xi_2)$, with inner product
\[
\langle x,\xi\rangle=x_1\xi_1+x_2\xi_2.
\]
The Fourier transform $\hat\mu$ is a continuous and bounded function on $\R^2$.
In \cite{HM}, the concept of a Heisenberg uniqueness pair (HUP) was introduced.
It is similar to the notion of weakly mutually annihilating pairs of Borel
measurable sets having positive area measure, which appears, e.g.,
in the book by Havin and J\"oricke \cite{HJ}.
For $\Gamma\subset\R^2$ which is a finite disjoint union of smooth curves in
$\R^2$, let $\mathrm{M}(\Gamma)$ denote the Banach space of
complex-valued finite Borel measures in $\R^2$, supported on $\Gamma$.
Moreover, let $\mathrm{AC}(\Gamma)$ denote the closed subspace of
$\mathrm{M}(\Gamma)$ consisting of the measures that are absolutely
continuous with respect to arc length measure on $\Gamma$.

\begin{defn}
Let $\Gamma$ be a finite disjoint union of smooth curves in $\R^2$.
For a set $\Lambda\subset\R^2$, we say that {\em $(\Gamma,\Lambda)$ is a
Heisenberg uniqueness pair} provided that
\[
\forall\mu\in\mathrm{AC}(\Gamma):\quad
\hat\mu|_{\Lambda}=0 \,\,\,\implies\,\,\,\mu=0.
\]
\end{defn}

Heisenberg uniquess pairs in which $\Gamma$ is a straight line or the
union of two parallel lines were described in \cite{HM}. Later, Blasi
\cite{Bl} solved particular cases of the union of three parallel lines.
The ellipse case was considered independently by Lev and Sj\"{o}lin in
\cite{Lev} and \cite{Sj}; Sj\"{o}lin also considered the parabola in
\cite{Sj2}. More recently, Jaming and Kellay in \cite{JK} developed new tools
to study Heisenberg uniqueness pairs for a variety of curves $\Gamma$, while
Giri and Srivastava studied four parallel lines among other things
\cite{GirSriv}.
As for higher dimensional analogues, in \cite{GroJam} Gr\"ochenig and 
Jaming connected the topic with the Cram\'er-Wold theorem on quadratic 
surfaces, while in \cite{Sriv}, Srivastava studied pairs composed of
spheres and cones.

\subsection{The Zariski closure}

We turn to the notion of the Zariski closure. Note that the Zariski
topology (or hull-kernel topology) is a standard concept in e.g. Algebraic
Geometry, in the setting of spaces of polynomials.
As for notation, we let
\emph{$\mathrm{AC}(\Gamma;\Lambda)$ be the subspace of
$\mathrm{AC}(\Gamma)$ consisting of those measures $\mu$ whose Fourier
transform vanishes on $\Lambda$}.

\begin{defn}
Let $\Gamma$ be a finite disjoint union of smooth curves in $\R^2$, and let
let $\Lambda\subset\R^2$ be arbitrary. With respect to $\mathrm{AC}(\Gamma)$,
the \emph{Zariski closure of} $\Lambda$ is the set
\[
\mathrm{zclos}_{\Gamma}(\Lambda):=
\{\xi\in\R^2:\,[\forall\mu\in\mathrm{AC}(\Gamma;\Lambda):\hat\mu(\xi)=0]\}.
\]
\end{defn}

Less formally, the Zariski closure (or hull) is the set where the Fourier
transform of a measure $\mu\in\mathrm{AC}(\Gamma)$ must vanish given that
it already vanishes on $\Lambda$.
Now, as the Fourier image of $\mathrm{AC}(\Gamma)$ does not form an algebra
with respect to pointwise multiplication of functions, we cannot expect the
Zariski closure to correspond to a topology.
This means that the intersection of two Zariski closures need not be a
closure itself. It is easy to see that the closure operation is idempotent,
however:
$\mathrm{zclos}_{\Gamma}^2=\mathrm{zclos}_{\Gamma}$.
In terms of the Zariski closure, we may express the uniqueness
pair property conveniently: $(\Gamma,\Lambda)$ \emph{is a Heisenberg
uniqueness pair if and only if}
\[
\mathrm{zclos}_{\Gamma}(\Lambda)=\R^2.
\]

\subsection{The Klein-Gordon equation}
In natural units, the Klein-Gordon equation in one spatial dimension reads
\[
\partial_t^2u-\partial_x^2u+\EP u=0.
\]
In terms of the (preferred) coordinates
\[
\xi_1:=t+x,\,\,\,\xi_2:=t-x,
\]
the Klein-Gordon equation becomes
\begin{equation}
\partial_{\xi_1}\partial_{\xi_2}u+\frac{\EP}{4}u=0.
\label{eq-KG100}
\end{equation}

\begin{rem}
Since $t^2-x^2=\xi_1\xi_2$, the \emph{time-like vectors} (those vectors
$(t,x)\in\R^2$ with $t^2-x^2>0$) correspond to the union of the first quadrant
$\xi_1,\xi_2>0$ and the third quadrant $\xi_1,\xi_2<0$ in the
$(\xi,\xi_2)$-plane). Likewise, the \emph{space-like vectors} correspond to the
union of the second quadrant $\xi_1>0,\xi_2<0$ and the fourth quadrant
$\xi_1<0,\xi_2>0$.
\end{rem}

\subsection{Fourier analytic treatment of the Klein-Gordon 
equation}
In the sequel, we will not need to talk about the time and space coordinates
$(t,x)$ as such. So, e.g., we are free to use the notation $x=(x_1,x_2)$
for the Fourier dual coordinate to $\xi=(\xi_1,\xi_2)$.

Let $\calM(\R^2)$ denote the Banach space of all finite complex-valued
Borel measures in $\R^2$. We suppose that $u$ is the Fourier
transform of a $\mu\in\calM(\R^{2})$:
\begin{equation}
u(\xi)=\hat\mu(\xi):=\int_{\R^{2}}
\e^{\imag\pi\langle x,\xi\rangle}
\diff\mu(x),\qquad\xi\in\R^2.
\label{eq-1.1'}
\end{equation}
The assumption that $u$ solves the Klein-Gordon equation \eqref{eq-KG100}
would ask that
\[
\bigg(x_1x_2-\frac{\EP}{4\pi^2}\bigg)\diff\mu(x)=0
\]
as a measure on $\R^{2}$, which we see is the same as a requirement on the
support set of the measure $\mu$:
\begin{equation}
\supp\mu\subset\Gamma_M:=\bigg\{x\in\R^2:
\,x_1x_2=\frac{\EP}{4\pi^2}\bigg\}.
\label{eq-1.2}
\end{equation}
The set $\Gamma_M$ is a hyperbola. We may use the $x_1$-axis to supply
a global coordinate for $\Gamma_M$, and define a complex-valued finite Borel
measure $\bpi_1\mu$ on $\R$ by setting
\begin{equation}
\bpi_1\mu(E)=\int_{E}\diff\bpi\mu(x_1):=\mu(E\times\R)=
\int_{E\times\R}\diff\mu(x).
\label{eq-bpi1}
\end{equation}
We shall at times refer to $\bpi_1\mu$ as the {\em compression} of $\mu$ to the
$x_1$-axis. It is easy to see that $\mu$ may be recovered from $\bpi_1\mu$;
indeed,
\begin{equation}
u(\xi)=\hat\mu(\xi)=\int_{\R^\times}
\e^{\imag\pi[\xi_1t+M^2\xi_2/(4\pi^{2}t)]}
\diff\bpi_1\mu(t),\qquad \xi\in\R^2.
\label{eq-1.3}
\end{equation}
Here, we use the standard notational convention $\R^\times:=\R\setminus\{0\}$.
We note that $\mu$ is absolutely continuous with respect to arc length measure
on $\Gamma_M$ if and only if $\bpi_1\mu$ is absolutely continuous with respect
to Lebesgue length measure on $\R^\times$.

\subsection{The lattice-cross as a uniqueness set for solutions to
the Klein-Gordon equation}
For positive reals $\alpha,\beta$, let $\Lambda_{\alpha,\beta}$
denote the lattice-cross
\begin{equation}
\Lambda_{\alpha,\beta}:=(\alpha\Z\times\{0\})\cup(\{0\}\times\beta\Z),
\label{eq-LC}
\end{equation}
so that the spacing along the $\xi_1$-axis is $\alpha$, and along the
$\xi_2$-axis it is $\beta$. In the   work \cite{HM}, Hedenmalm and
Montes-Rodr\'{i}guez found the following.

\begin{thm} {\rm (Hedenmalm, Montes)} Fix positive reals $M,\alpha,\beta$.
Then $(\Gamma_M,\Lambda_{\alpha,\beta})$ is a Heisenberg uniqueness pair
if and only if $\alpha\beta M^2\le 4\pi^2$.
\label{thm-1}
\end{thm}

In terms of the Zariski closure, the theorem says that
\[
\mathrm{zclos}_{\Gamma_M}(\Lambda_{\alpha,\beta})=\R^2
\]
holds if and only if $\alpha\beta M^2\le 4\pi^2$. By taking the relation
\eqref{eq-1.3} into account, and by reducing the redundancy of the constants
(i.e., we may without loss of generality consider $M=2\pi$ and
$\alpha=1$ only), Theorem \ref{thm-1} is equivalent to the
following statement: \emph{the linear span of the functions}
\[
\e^{\imag \pi m t},\quad \e^{-\imag\pi \beta n/t},\qquad m,n\in\Z,
\]
\emph{is weak-star dense in $L^\infty(\R)$ if and only if} $\beta\le1$.
Here, we supply new and unexpected insight into the theory of Heisenberg
uniqueness pairs, such as a new connection with the standard Gauss map
(motivated by Theorem \ref{thm-2.1}), and, more importantly, we uncover,
in the framework of Fourier Analysis, profound connections between
the Hilbert transform and the dynamics of transfer operators intimately
related to Gauss-type maps leading up to Theorem \ref{thm-2.0}.

\subsection{Dynamic unique continuation from a branch of the
hyperbola}
Just looking at Theorem \ref{thm-1}, one immediately is led to ask what
  happens if we  replace the hyperbola $\Gamma_M$ by one of its
two branches, say
\begin{equation}
\Gamma_M^+:=\Gamma_M\cap(\R_+\times\R_+)=
\bigg\{x\in\R^2:
\,x_1x_2=\frac{\EP}{4\pi^2}\,\,\,\text{and}\,\,\,x_1>0\bigg\}.
\label{eq-Gbranch1.1}
\end{equation}
First,  we will provide a uniqueness theorem for the branch $\Gamma_M^+$ of the
hyperbola $\Gamma_M$, which turns out to be closely related to the famous
Gauss-Kuzmin-Wirsing operator and the Gauss map $x\mapsto1/x$ mod $\Z$.

\begin{thm}
Fix positive reals $\alpha,\beta,M$.
Then  $(\Gamma_M^+,\Lambda_{\alpha,\beta})$ is a Heisenberg uniqueness
pair if and only if $\alpha\beta M^2<16\pi^2$. Moreover, in the critical case
$\alpha\beta M^2=16\pi^2$, the space
$\mathrm{AC}(\Gamma_M^+,\Lambda_{\alpha,\beta})$ is the one-dimensional space
spanned by the measure $\mu_0\in\mathrm{AC}(\Gamma_M^+,\Lambda_{\alpha,\beta})$
whose $x_1$-compression is given by
\[
\diff\bpi_1\mu_0(t):=
\Bigg\{\frac{1_{[0,2/\alpha]}(t)}{2(2+\alpha t)}
-\frac{1_{[2/\alpha,+\infty[}(t)}{\alpha t(2+\alpha t)}
\Bigg\}\diff t.
\]
\label{thm-2.1}
\end{thm}

The proof of Theorem \ref{thm-2.1} is presented in Section
\ref{sec-dynamicuniqcont}. In the same section, it is also shown that in
the critical parameter regime $\alpha\beta=16\pi^2$, the couple
$(\Gamma_M^+,\Lambda_{\alpha,\beta}^\star)$ is indeed a Heisenberg uniqueness pair,
where $\Lambda_{\alpha,\beta}^\star:=\Lambda_{\alpha,\beta}\cup\{\xi^\star\}$, and
$\xi^\star\in(\{0\}\times\R)\cup(\R\times\{0\})$ is any point off
the lattice-cross $\Lambda_{\alpha,\beta}$ (see Theorem \ref{cor-onebranch}).
The analysis of the proof of Theorem \ref{cor-onebranch} involves a
geometric object known as the \emph{Nielsen spiral}.

Again, by taking the relation \eqref{eq-1.3} into account, and by reducing
the redundancy of the constants (i.e., we may without loss of generality
consider $M=2\pi$ and $\alpha=1$ only), it is easy to see that Theorem
\ref{thm-2.1} entails the following assertion:
\emph{the restriction to $\R_+$ of the linear span of the functions}
\[
\e^{\imag \pi m t},\quad \e^{-\imag\pi \beta n/t},\qquad m,n\in\Z,
\]
\emph{is weak-star dense in $L^\infty(\R_+)$ if and only if} $\beta<4$.
\emph{Moreover, if $\beta=4$ the weak-star closure of this linear span has
codimension one in $L^\infty(\R_+)$}.

Theorem \ref{thm-2.1} has the following consequence in terms of unique
continuation from the branch $\Gamma_M^+$, or the complementary branch
$\Gamma_M^-:=\Gamma_M\setminus\Gamma_M^+$, to the entire hyperbola $\Gamma_M$.

\begin{cor} Fix positive reals $\alpha,\beta,M$.
Then $\mu\in\mathrm{AC}(\Gamma_M,\Lambda_{\alpha,\beta})$ is uniquely
determined by its restriction to the hyperbola
branch $\Gamma_M^-$ if and only if $\alpha\beta M^2<16\pi^2$. The same
holds with $\Gamma_M^-$ replaced by $\Gamma_M^+$ as well.
\end{cor}

\subsection{The Zariski closure of the axes and half-axes}
We first consider the Zariski closure of the two axes $\R\times\{0\}$ and
$\{0\}\times\R$ with respect to the space $\mathrm{AC}(\Gamma_M)$ of absolutely
continuous measures,  with respect to arc length,  on the hyperbola $\Gamma_M$.

\begin{prop} Fix a positive real $M$.
If $\mu\in \mathrm{AC}(\Gamma_M)$ is such that $\hat\mu$ vanishes one of
the axes, $\R\times\{0\}$ or $\{0\}\times\R$, then $\mu=0$ identically.
In terms of Zariski closures, this means that
\[
\mathrm{zclos}_{\Gamma_M}(\R\times\{0\})=\mathrm{zclos}_{\Gamma_M}(\{0\}\times\R)
=\R^2.
\]
\label{prop-1.1}
\end{prop}

The proof of Proposition \ref{prop-1.1} is supplied in Section
\ref{sec-zarcalc}.

The next proposition will show the difference between time-like and space-like
quarterplanes. First, we need some notation.
Let $\R_+:=\{t\in\R:\,t>0\}$ and  $\R_-:=\{t\in\R:\,t<0\}$ be the positive
and negative half-lines, respectively. We write $\bar\R_+:=\{t\in\R:\,t\ge0\}$
and  $\bar\R_-:=\{t\in\R:\,t\le0\}$ for the corresponding closed half-lines.

\begin{prop}  Fix a positive real $M$.
Then the Zariski closures of each of the four semi-axes $\R_+\times\{0\}$,
$\R_-\times\{0\}$,  $\{0\}\times\R_+$, and $\{0\}\times\R_-$, are as follows:
\[
\mathrm{zclos}_{\Gamma_M}(\R_+\times\{0\})
=\mathrm{zclos}_{\Gamma_M}(\{0\}\times\R_-)=\bar\R_+\times\bar\R_-
\]
and
\[
\mathrm{zclos}_{\Gamma_M}(\R_-\times\{0\})
=\mathrm{zclos}_{\Gamma_M}(\{0\}\times\R_+)=\bar\R_-\times\bar\R_+.
\]
\label{prop-2.1}
\end{prop}

The proof of Proposition \ref{prop-2.1} is also supplied in
Section \ref{sec-zarcalc}.

\begin{rem}
In each of the instances in Proposition \ref{prop-2.1}, we note that
the Zariski closure of a semi-axis equals the topological closure of
the adjacent quadrant of \emph{space-like vectors}.
\end{rem}

\subsection{The Zariski closure of the lattice-cross restricted 
to a
time-like or space-like quadrant}
\label{subsec-quad1.01}

Let us write
\[
\Z_+:=\{1,2,3,\ldots\},\quad \Z_-:=\{-1,-2,-3,\ldots\},\quad
\Z_{+,0}:=\{0,1,2,\ldots\},\quad\Z_{-,0}:=\{0,-1,-2,\ldots\}
\]
for the sets of positive, negative, nonnegative, and nonpositive integers,
respectively. We consider the following four portions of the lattice-cross
$\Lambda_{\alpha,\beta}$ given by \eqref{eq-LC}:
\[
\Lambda_{\alpha,\beta}^{++}:=(\alpha\Z_{+,0}\times\{0\})\cup(\{0\}\times\beta\Z_+),
\quad
\Lambda_{\alpha,\beta}^{+-}:=(\alpha\Z_{+,0}\times\{0\})\cup(\{0\}\times\beta\Z_-),
\]
and
\[
\Lambda_{\alpha,\beta}^{-+}:=(\alpha\Z_{-,0}\times\{0\})\cup(\{0\}\times\beta\Z_+),
\quad
\Lambda_{\alpha,\beta}^{--}:=(\alpha\Z_{-,0}\times\{0\})\cup(\{0\}\times\beta\Z_-).
\]
We first calculate the Zariski closure of two of these (the first and the 
last), 
corresponding to the first and third quadrants, which are time-like.

\begin{thm} {\rm(time-like)}
Fix positive reals $\alpha,\beta,M$. Then for each point
$\xi^\star\in\R^2\setminus\Lambda_{\alpha,\beta}^{++}$, there exists a measure
$\mu\in\mathrm{AC}(\Gamma_M)$ such that $\hat\mu=0$ on
$\Lambda_{\alpha,\beta}^{++}$, while at the same time $\hat\mu(\xi^\star)\ne0$.
Moreover, the same assertion holds provided that $\Lambda_{\alpha,\beta}^{++}$ is
replaced by $\Lambda_{\alpha,\beta}^{--}$. In terms of Zariski closures,
this means that
\[
\mathrm{zclos}_{\Gamma_M}(\Lambda_{\alpha,\beta}^{++})=\Lambda_{\alpha,\beta}^{++},
\qquad
\mathrm{zclos}_{\Gamma_M}(\Lambda_{\alpha,\beta}^{--})=\Lambda_{\alpha,\beta}^{--}.
\]
\label{prop-zariski1}
\end{thm}

The proof of Theorem \ref{prop-zariski1}, which is presented in Section
\ref{sec-zarcalc2},  requires careful handling of the
$H^1$-BMO duality and the explicit calculation of the Fourier transform
of the unimodular function $t\mapsto\e^{\imag/t}$ as a tempered distribution.

We turn to the Zariski closures of the remaining two portions of the
lattice-cross. We first write down the statement in terms of
weak-star closure of the linear span of a sequence of unimodular functions,
and then explain what it means for the Zariski closure in the form of a
corollary. This is our second main result.

As for notation, let $H^\infty_+(\R)$ denote the weak-star closed subspace
of $L^\infty(\R)$ consisting of those functions whose Poisson extension
to the upper half-plane is holomorphic.

\begin{thm}
Fix positive reals $\alpha,\beta$.
Then the functions
\[
\e^{\imag \pi\alpha m t},\quad \e^{-\imag\pi \beta n/t},\qquad m,n=0,1,2,\ldots,
\]
which are elements of $H^\infty_+(\R)$, span together a weak-star dense
subspace of $H^\infty_+(\R)$ if and only if $\alpha\beta\le1$.
\label{thm-2.0}
\end{thm}

A standard M\"obius mapping brings the upper half-plane to the unit disk 
$\D$, and identifies the space $H^\infty_+(\R)$ with $H^\infty(\D)$, the space 
of all bounded holomorphic functions on $\D$. For this reason, 
Theorem \ref{thm-2.0} is equivalent to the following assertion, which we state
as a corollary.  

\begin{cor}
Fix positive reals $\lambda_1,\lambda_2$. Then the linear span of the inner 
functions
\[
\phi_1(z)^m=\exp\left(m\lambda_1\frac{z+1}{z-1}\right) \quad
\text { and } \quad
\phi_2(z)^n=\exp\left(n\lambda_2\frac{z-1}{z+1}\right),
\qquad m,n=0,1,2,\ldots,
\]
is weak-star dense set in $H^\infty(\D)$ if  and only if 
$\lambda_1\lambda_2\le\pi^2$. 
%In particular, the same is true for the algebra generated by $\phi_1$ 
%and $\psi_2$. 
\label{cor-innerMS}
\end{cor}

We omit the trivial proof of the corollary.

\begin{rem} 
%As pointed out in \cite{MS}, as a consequence of the above Corollary, 
%we have that the lattice of $\mathcal A$-invariant closed subspaces 
%coincide with the invariant subspaces of multiplication by $z$.
Clearly, Corollary \ref{cor-innerMS} supplies a complete and affirmative
answer to Problems 1 and 2 in \cite{MS}. We recall the question from \cite{MS}:
the issue was raised whether the algebra generated by the two inner functions
\[
\phi_1(z)=\exp\bigg(\lambda_1\frac{z+1}{z-1}\bigg) \quad \text { and }
\quad \phi_2(z)=\exp\bigg(\lambda_2\frac{z-1}{z+1}\bigg)
\]
for $0<\lambda_1,\lambda_2<+\infty$, is weak-star dense in $H^\infty(\D)$
if and only if $\lambda_1\lambda_2\le\pi^2$. The ``only if'' was understood
already in \cite{MS}. 
As pointed out in \cite{MS}, it is a consequence of Corollary 
\ref{cor-innerMS} that for $\lambda_1\lambda_2\le\pi^2$, 
the lattice of the closed subspaces 
%of the Hardy space $H^2(\D)$ 
invariant with respect to multiplication by the two inner functions 
$\phi_1,\phi_2$ coincides with the usual shift invariant subspaces in the
Hardy space $H^p(\D)$, where $1<p<+\infty$.
\end{rem}

%Theorem \ref{thm-2.0} can be restated in terms of uniqueness properties 
%of solutions to the Klein-Gordon equation. Note that in the statement below, 
%the pair $(\Lambda_{\alpha,\beta}^{+-},\bar\R_+\times\bar\R_-)$ can be 
%replaced by 
%$(\Lambda_{\alpha,\beta}^{-+},\bar\R_-\times\bar\R_+)$
%without perturbing the validity of the result. 
%
%\begin{cor}
%Fix positive reals $\alpha,\beta,M$ with $\alpha\beta M^2\le 4\pi^2$. 
%Suppose that $u=\hat\mu$ solves the Klein-Gordon equation \eqref{eq-KG100}, 
%where $\mu$ is finite complex Borel measure on $\R^2$, which is assumed
%absolutely continuous with respect to one-dimensional Hausdorff measure. 
%Then the values of $u$ on the space-like quarter-plane 
%$\bar\R_+\times\bar\R_-$ are determined by the values of $u$ on the set 
%$\Lambda_{\alpha,\beta}^{+-}$, which is the portion of the lattice-cross in the
%given quarter-plane. This property does not hold for $\alpha\beta M^2>4\pi^2$.
%it also vanishes on the adjacent quarter-plane
%$\bar\R_+\times\bar\R_-$.
%is determined by the values on $\Lambda^{++}_{\alpha\beta}$ if and 
%only if $\alpha\beta M^2\leq 4\pi^2$,
%\label{KleinG}
%\end{cor}
%
%This formulation is actually a consequence of the Zariski closure result 
%of Corollary \ref{cor-2.0}, so we refer to the explanatory remarks that 
%follow directly after the latter corollary. 

\begin{rem}
It is impossible to derive the assertion of Theorem \ref{thm-2.0} from
Theorem \ref{thm-1}. It is a \emph{much finer statement}. In Section
\ref{sec-spanning},
we explain how the result relies on a hitherto unknown result, presented
in \cite{HMerg}, which extends the standard ergodic theory for
certain Gauss-type transformations on the interval $I_1:=]\!-\!1,1[$, where
the novelty is that we may handle distributions where the standard theory
has only measures. The relevant space of distributions is obtained as the
restriction to $I_1$ of $L^1(\R)$ plus $\Hop L^1(\R)$, where $\Hop$ is the
Hilbert transform (i.e., convolution with the principal value distribution
$\pv\! \frac{1}{\pi t}$ on the line).
%Thinking physically,
The issue has to do with the uniqueness of the absolutely continuous
invariant measure in the larger space.  Thinking physically, in the larger
space, we have two types of particles, localized and delocalized. The localized
particles are represented by $\delta_\xi$, whereas delocalized particles
are represented by $\Hop\delta_\xi$, for some real $\xi$. The state space allows
for scalar multiples of localized and delocalized particles, and linear
combinations of them. Finally, we are looking for such localized and
delocalized particles smeared out in an absolutely continuous way, and call
it an \emph{invariant state} if it is preserved under the corresponding
Gauss-type map. This generalizes the notion of the absolutely continuous
invariant measure which is standard in ergodic theory, and since uniqueness
issues for the invariant measure translate to ergodic properties, we
are left with a far-reaching generalization of ergodic theory.
We have not been able to find any appropriate references for similar
considerations in the literature.
%, but suggest that there may be some
%relevance of the works \cite{ADS1} and \cite{ADS2} for the discrete setting,
%and \cite{Buf} for flows.
%which forms the starting point The proof unwraps
%a previously unknown deep connection between the dynamics of Gauss-type
%maps and the Hilbert transform. The details of the proof are presented in
%Sections \ref{sec-asymptotic.decay<1}  (in the subcritical regime
%$\alpha\beta<1$)  and \ref{sec-asymptotic.decay=1} (in the critical regime
%$\alpha\beta=1$). The underlying methods are developed in Sections
%\ref{sec-background.dynamics}--\ref{sec-Hurwitz} and
%, \ref{sec-iterates1}, \ref{sec-Hilbkernel}, \ref{sec-Hurwitz}, and
%\ref{sec-HilbL1}--\ref{sec-subtransferop}.
%a delicate argument
%We should mention that it is immediate from the work

All the effort in \cite{HMerg} is developed to deal with the ``if'' part
of the assertion of Theorem \ref{thm-2.0}.
On the other hand, the ``only if'' part is much simpler, as for instance
the work in \cite{CHM} shows that in case $\alpha\beta>1$, the
weak-star closure of the linear span in question has infinite codimension in
$H^\infty_+(\R)$.
\end{rem}

Theorem \ref{thm-2.0} can be restated in terms of uniqueness properties 
of solutions to the Klein-Gordon equation. Note that in the statement below, 
the pair $(\Lambda_{\alpha,\beta}^{+-},\bar\R_+\times\bar\R_-)$ can be replaced by 
$(\Lambda_{\alpha,\beta}^{-+},\bar\R_-\times\bar\R_+)$
without perturbing the validity of the result. 

\begin{cor}
Fix positive reals $\alpha,\beta,M$ with $\alpha\beta M^2\le 4\pi^2$. 
Suppose that $u=\hat\mu$ solves the Klein-Gordon equation \eqref{eq-KG100}, 
where $\mu$ is finite complex Borel measure on $\R^2$, which is assumed
absolutely continuous with respect to one-dimensional Hausdorff measure. 
Then the values of $u$ on the space-like quarter-plane 
$\bar\R_+\times\bar\R_-$ are determined by the values of $u$ on the set 
$\Lambda_{\alpha,\beta}^{+-}$, which is the portion of the lattice-cross in the
given quarter-plane. This property does not hold for $\alpha\beta M^2>4\pi^2$.
%it also vanishes on the adjacent quarter-plane
%$\bar\R_+\times\bar\R_-$.
%is determined by the values on $\Lambda^{++}_{\alpha\beta}$ if and 
%only if $\alpha\beta M^2\leq 4\pi^2$,
\label{KleinG}
\end{cor}

This formulation is actually a consequence of the Zariski closure result 
of Corollary \ref{cor-2.0} below, so we refer to the explanatory remarks that 
follow right after it.
%the latter corollary. 

%The formulation of Theorem \ref{thm-2.0} in terms of the Zariski closure
%reads as follows.

\begin{cor}
{\rm(space-like)}
Fix positive reals $\alpha,\beta,M$. The following  assertions are
equivalent:

\noindent{\rm(i)} $\mathrm{zclos}_{\Gamma_M}(\Lambda_{\alpha,\beta}^{+-})
=\bar\R_+\times\bar\R_-$,

\noindent{\rm(ii)} $\mathrm{zclos}_{\Gamma_M}(\Lambda_{\alpha,\beta}^{-+})
=\bar\R_-\times\bar\R_+$,

\noindent{\rm(iii)} $\alpha\beta M^2\le4\pi^2$.
\label{cor-2.0}
\end{cor}

Here, the main part of the equivalence (i)$\Leftrightarrow$(iii) is the
implication
(iii)$\Rightarrow$(i'), where (i') is as follows:

\noindent (i') $\mathrm{zclos}_{\Gamma_M}(\Lambda_{\alpha,\beta}^{+-})
\supset\bar\R_+\times\bar\R_-$.

\noindent The latter implication can be understood in the following terms.
Under the density condition (iii), any measure $\mu\in \mathrm{AC}(\Gamma_M)$
 whose Fourier transform $\hat\mu$ vanishes on $\Lambda_{\alpha,\beta}^{+-}$,
has the property that $\hat\mu$ actually vanishes on the entire space-like
adjacent quarter-plane $\bar\R_+\times\bar\R_-$. This assertion is seen
to be equipotent with Theorem \ref{thm-2.0}, after a scaling argument which
permits us to assume that $M:=2\pi$. Finally, to obtain the equality (i) from
the inclusion (i') which results from Theorem \ref{thm-2.0}, we may use e.g.
Proposition \ref{prop-2.1}.
The remaining equivalence
(ii)$\Leftrightarrow$(iii) is, by a symmetry argument, the same as the
the equivalence (i)$\Leftrightarrow$(iii).

\begin{rem}
Let us now
explain how Theorem \ref{thm-1} is an immediate consequence of
the much deeper result of Corollary \ref{cor-2.0}.
First, an elementary argument (see \cite{HM}, \cite{CHM}) shows that
$\mathrm{zclos}_{\Gamma_M}(\Lambda_{\alpha,\beta})\ne\R^2$ for
$\alpha\beta M^2>4\pi^2$, so that we just need to obtain the implication
\[
\alpha\beta M^2\le4\pi^2\implies
\mathrm{zclos}_{\Gamma_M}(\Lambda_{\alpha,\beta})=\R^2.
\]
In view of Theorem \ref{thm-2.0},
\[
\alpha\beta M^2\le4\pi^2\implies
\mathrm{zclos}_{\Gamma_M}(\Lambda_{\alpha,\beta})
=\mathrm{zclos}_{\Gamma_M}(\Lambda_{\alpha,\beta}^{+-}\cup
\Lambda_{\alpha,\beta}^{-+})\supset
(\bar\R_+\times\bar\R_-)\cup(\bar\R_-\times\bar\R_+)\supset\R\times\{0\},
\]
and Theorem \ref{thm-1} becomes a consequence of Proposition \ref{prop-1.1}
together with the idempotent property
$\mathrm{zclos}_{\Gamma}^2=\mathrm{zclos}_{\Gamma}$.
\end{rem}

\subsection{Acknowledgements} We thank the referee for pointing 
out to us an elementary proof of the fact that the Nielsen spiral does spiral
around the origin and converges to the origin only in the limit.

\section{The Zariski closures of the axes or semi-axes}
\label{sec-zarcalc}

\subsection{The standard Hardy spaces $H^p_+(\R)$}
\label{subsec-Hp}
The Hardy space $H^\infty_+(\R)$ consists of all functions $f\in L^\infty(\R)$
with Poisson extension to the upper half-plane
\[
\C_+:=\{z\in\C:\,\,\im z>0\}
\]
which is holomorphic. Here, the Poisson extension of $f$ is given by the
expression
\[
f(z):=\frac{\im z}{\pi}\int_\R \frac{f(t)}{|z-t|^2}\diff t, \qquad z\in\C_+.
\]
In a similar fashion, for $1\le p<+\infty$, we say that $f\in H^p_+(\R)$
if $f\in L^p(\R)$ and its Poisson extension is holomorphic in $\C_+$.

\subsection{The Zariski closures of the axes and semi-axes}
We now supply the proofs of Propositions \ref{prop-1.1} and \ref{prop-2.1}.
We should mention here that a more general version of Proposition  
\ref{prop-1.1} can be found in \cite{JK}.

\begin{proof}[Proof of Proposition \ref{prop-1.1}]
By symmetry, it is enough to show that $\mathrm{zclos}_{\Gamma_M}(\R\times\{0\})
=\R^2$. More concretely, we need to show that if $\mu\in\mathrm{AC}(\Gamma_M)$
and
\[
\hat\mu(\xi_1,0)=0,\qquad\xi_1\in\R,
\]
then $\mu=0$ as a measure. In view of \eqref{eq-1.3},
\[
\hat\mu(\xi_1,0)=\int_{\R^\times}\e^{\imag\pi\xi_1t}
\diff\bpi_1\mu(t),
\]
where $\bpi_1\mu$ is the compression of $\mu$ to the real line. The uniqueness
theorem for the Fourier transform gives that $\bpi_1\mu=0$, and hence that
$\mu=0$, since $\mu$ and its compression $\bpi_1\mu$ are in a one-to-one
correspondence.
\end{proof}

\begin{proof}[Proof of Proposition \ref{prop-2.1}]
By symmetry, it is enough to show that
\[
\mathrm{zclos}_{\Gamma_M}(\R_+\times\{0\})=\bar\R_+\times\bar\R_-
\]
To this end, we consider a measure $\mu\in\mathrm{AC}(\Gamma_M)$ with
(use \eqref{eq-1.3})
\[
\hat\mu(\xi_1,0)=\int_{\R}\e^{\imag\pi\xi_1t}
\diff\bpi_1\mu(t)=0,\qquad\xi_1\in\R_+.
\]
This condition is equivalent to asking that $\diff\bpi_1\mu(t)=f(t)\diff t$,
where $f\in H^1_+(\R)$.
If follows from standard arguments that
\[
\int_\R g(t)\diff \bpi_1\mu(t)=
\int_\R f(t)g(t)\diff t=0
\]
for all $g\in H^\infty_+(\R)$.
We observe that for $\xi_1\ge0$ and $\xi_2\le0$, the
function
\[
g(t):=\e^{\imag\pi[\xi_1t+M^2\xi_2/(4\pi^2 t)]}
\]
is in $H^\infty_+(\R)$, and so
\[
\hat\mu(\xi_1,\xi_2)=\int_{\R^\times}\e^{\imag\pi[\xi_1t+M^2\xi_2/(4\pi^2 t)]}
\diff\bpi_1\mu(t)=0,\qquad (\xi_1,\xi_2)\in\bar\R_+\times\bar\R_-.
\]
In conclusion, this argument proves the inclusion
\[
\mathrm{zclos}_{\Gamma_M}(\R_+\times\{0\})\supset\bar\R_+\times\bar\R_-.
\]
To obtain the equality of the two sides, we need to show that if
$(\xi_1,\xi_2)\in\R^2\setminus(\bar\R_+\times\bar\R_-)$, then there exists
a $\mu\in\mathrm{AC}(\Gamma_M)$ with $\diff\bpi_1\mu(t)=f(t)\diff t$,
where $f\in H^1_+(\R)$, such that $\hat\mu(\xi_1,\xi_2)\ne0$. But then the
bounded function
\[
g(t)=\e^{\imag\pi[\xi_1t+M^2\xi_2/(4\pi^2 t)]},\qquad t\in\R,
\]
is not an element of $H^\infty_+(\R)$, and by the standard Hardy space duality
theory,
\[
\sup\bigg\{\bigg|\int_\R f(t)g(t)\diff t\bigg|:\,f\in\mathrm{ball}(H^1_+(\R))
\bigg\}=\inf\big\{\|g-h\|_{L^\infty(\R)}:\,h\in H^\infty_+(\R)\big\}>0.
\]
In particular, there must exist an $f\in H^1_+(\R)$ with
\[
\hat\mu(\xi_1,\xi_2)=\int_\R f(t)g(t)\diff t\ne0.
\]
This completes the proof.
\end{proof}

\section{Basic properties of the dynamics of 
Gauss-type maps on intervals}
\label{sec-background.dynamics}

\subsection{Notation for intervals}
\label{subsec-notation-intervals}
For a positive real $\gamma$, let $I_\gamma:=]\!-\!\gamma,\gamma[$ denote
the corresponding symmetric open interval, and let $I_\gamma^+:=]0,\gamma[$
be the positive side of the interval $I_\gamma$. At times, we will need the
half-open intervals $\tilde I_\gamma:=]\!-\!\gamma,\gamma]$ and
$\tilde I_\gamma^+:=[0,\gamma[$, as well as the closed
intervals $\bar I_\gamma:=[-\gamma,\gamma]$ and $\bar I_\gamma^+:=[0,\gamma]$.

\subsection{Dual action notation}
\label{subsec-dualaction}
For a Lebesgue measurable subset $E$ of the real line $\R$,
we write
\[
\langle f,g\rangle_E:=\int_{E}f(t)g(t)\diff t,
\]
whenever $fg\in L^1(E)$. This will be of interest mainly when $E$ is an
open interval, and in this case, we use the same notation to describe the dual
action of a distribution on a test function.

\subsection{Gauss-type maps on intervals}
\label{subsec-Gausstype}
For background material in Ergodic Theory, we refer to the book \cite{CFS}.

For $x\in\R$, let $\{x\}_1$ denote the \emph{fractional part of $x$}, that is,
the unique number in the half-open interval $\tilde I_1^+=[0,1[$ with
$x-\{x\}_1\in\Z$.
Likewise, we let $\{x\}_2$ denote the \emph{even-fractional part of $x$},
by which we mean the unique number in the half-open interval
$\tilde I_1=]\!-\!1,1]$ with $x-\{x\}_2\in2\Z$. We will be interested in
the Gauss-type maps $\sigma_\gamma:\tilde I_1^+\to \tilde I_1^+$ and
$\tau_\beta:\tilde I_1\to \tilde I_1$
given by
\[
\sigma_\gamma(x):=\bigg\{\frac{\gamma}{x}\bigg\}_1\quad\text{and}\quad
\tau_\beta(x):=\bigg\{-\frac{\beta}{x}\bigg\}_2.
\]
Here, $\beta,\gamma$ are reals with $0<\beta,\gamma\le1$. Then $\sigma_1$
is the classical Gauss map of the unit interval $I_1^+$.

\subsection{Transfer, subtransfer, and compressed Koopman
operators}
\label{subsec-transf-koopman}
Fix two reals $\beta,\gamma$ with $0<\beta,\gamma\le1$.
Let $\Gmop_\gamma:L^\infty(I_1^+)\to L^\infty(I_1^+)$ and
$\Lop_\beta:L^\infty(I_1)\to L^\infty(I_1)$ and
denote the \emph{compressed Koopman operators}
(or \emph{sub-Koopman operators})
\begin{equation}
\Gmop_\gamma f(x):=1_{I_\gamma^+}(x)f\circ\sigma_\gamma(x),\qquad
\Lop_\beta f(x):=1_{I_\beta}(x)f\circ\tau_\beta(x).
\label{eq-Cop.Kop}
\end{equation}
Here, as always, $1_E$ stands for the characteristic function of the set $E$,
which equals $1$ on $E$ and vanishes elsewhere.
The \emph{subtransfer operators} $\Sop_\gamma:L^1(I_\gamma^+)\to L^1(I_\gamma^+)$
and $\Tope_\beta:L^1(I_1)\to L^1(I_1)$ are defined by
\begin{equation}
\Sop_\gamma f(x):=\sum_{j=1}^{+\infty}\frac{\gamma}{(j+x)^2}
f\bigg(\frac{\gamma}{j+x}\bigg),\qquad
\Tope_\beta f(x):=\sum_{j\in \Z^\times}\frac{\beta}{(2j+x)^2}
f\bigg(-\frac{\beta}{2j+x}\bigg).
\label{eq-Uop.Wop}
\end{equation}
Here, we use the notation $\Z^\times:=\Z\setminus\{0\}$.
A standard argument shows that
\begin{equation}
\begin{cases}
\displaystyle
\langle \Sop_\gamma f,g\rangle_{I_1^+}=\langle f,\Gmop_\gamma g\rangle_{I_1^+},
\qquad f\in L^1(I_1^+),\,\,\,g\in L^\infty(I_1^+),
\\
\displaystyle
\langle \Tope_\beta f,g\rangle_{I_1}=\langle f,\Lop_\beta g\rangle_{I_1},
\qquad\,\, f\in L^1(I_1),\,\,\,g\in L^\infty(I_1);
\end{cases}
\label{eq-duality.Uop.Wop.Cop.Kop}
\end{equation}
in other words, \emph{$\Sop_\gamma$ is the preadjoint of $\Gmop_\gamma$, and
$\Tope_\beta$ is the preadjoint of $\Lop_\beta$}.

The \emph{cone of positive functions} consists of all integrable functions
$f$ with $f\ge0$ a.e. on the respective interval. Similarly, we say that $f$ is
\emph{positive} if $f\ge0$ a.e. on the given interval.

\begin{prop} Fix $0<\beta,\gamma\le1$.
The operators $\Tope_\beta:L^1(I_1)\to L^1(I_1)$ and $\Sop_\gamma:L^1(I_\gamma^+)
\to L^1(I_\gamma^+)$ are both norm contractions, which preserve the respective
cones of positive functions. For $\beta=\gamma=1$, $\Tope_1$ and $\Sop_1$ act
isometrically on the positive functions. The associated adjoints
$\Lop_\beta:L^\infty(I_1)\to L^\infty(I_1)$ and
$\Gmop_\gamma:L^\infty(I_1^+)\to L^\infty(I_1^+)$ are norm contractions as well.
\label{prop-contract1}
\end{prop}

This is well-known for $\gamma=\beta=1$ and very easy to obtain for
$0<\beta,\gamma<1$.

\subsection{An elementary observation and an estimate of the
$\Tope_\beta$-orbit of certain functions}
\label{subsec-3.5}
We begin with the following elementary observation.
\medskip

\noindent{\sc Observation.}
The subtransfer operators $\Tope_\beta,\Sop_\gamma$, initially defined on
$L^1$ functions, make sense for wider classes of functions. Indeed,
if $f\ge0$, then the formulae \eqref{eq-Uop.Wop} make sense pointwise, with
values in the extended nonnegative reals $[0,+\infty]$. More generally, if
$f$ is complex-valued, we may use the triangle inequality to dominate the
convergence of $\Tope_\beta f$ by that of $\Tope_\beta|f|$. This entails that
$\Tope_\beta f$ is well-defined a.e. if $\Tope_\beta|f|<+\infty$ holds a.e.
The same goes for $\Sop_\gamma$ of course.
\medskip

In view of the above observation, it is meaningful to try to control
$\Tope_\beta f$ for $f\ge0$.
The following basic size estimate is useful.

\begin{prop}
Fix $0<\beta\le1$.
If $f:I_1\to\R$ is even and its restriction to $I_1^+$ is increasing, and if
$f\ge0$, then
\[
\beta C_0 f(0)\le
\Tope_\beta f(x)-\frac{\beta}{(2-|x|)^2}f\bigg(\frac{\beta}{2-|x|}\bigg)\le
\beta C_1 f(\tfrac12\beta),\qquad x\in I_1,
\]
where
$C_0:=\frac{\pi^2}6-\frac54=0.3949\ldots$ and $C_1:=\frac{\pi^2}{6}-1=
0.6449\ldots$.
\label{prop-convexitypres2}
\end{prop}

\begin{proof}
For convenience of notation, we write
\begin{equation}
s_j(x):=-\frac{\beta}{2j+x},
\label{eq-sjdef}
\end{equation}
which is an increasing function on $I_1$ for $j\in\Z^\times=\Z\setminus\{0\}$.
We first consider the right half of the interval, i.e., $x\in I_1^+=]0,1[$.
As $f$ is even, we see that
\[
f(s_j(x))=f\bigg(-\frac{\beta}{2j+x}\bigg)=f\bigg(\frac{\beta}{2j+x}\bigg),
\]
and since $f$ is increasing on $I_1^+$, we obtain that for integers $j\ge1$,
\[
f(0)\le f\bigg(\frac{\beta}{2j+1}\bigg)\le
f(s_j(x))=f\bigg(\frac{\beta}{2j+x}\bigg)
\le f\bigg(\frac{\beta}{2j}\bigg)\le f(\tfrac12\beta),\qquad x\in I_1^+,
\]
while for integers $j\le-2$ we have a similar estimate:
\[
f(0)\le f\bigg(\frac{\beta}{2|j|}\bigg)\le
f(s_j(x))=f\bigg(\frac{\beta}{2|j|-x}\bigg)
\le f\bigg(\frac{\beta}{2|j|-1}\bigg)\le f(\tfrac13\beta)\le
f(\tfrac12\beta),\qquad x\in I_1^+.
\]
Since
\[
\Tope_\beta f(x)-\frac{\beta}{(2-x)^2}f\bigg(\frac{\beta}{2-x}\bigg)=
\frac{1}{\beta}\sum_{j\in\Z\setminus\{0,-1\}}[s_j(x)]^2\,f(s_j(x)),
\]
the claimed estimate follows from
\[
\frac{\pi^2}{6}-\frac54\le
\frac{1}{\beta^2}\sum_{j\in\Z\setminus\{0,-1\}}[s_j(x)]^2\le
\frac{\pi^2}{6}-\frac54,\qquad x\in I_1^+.
\]
The remaining case when $x\in I_1^-:=]\!-\!1,0[$ is analogous.
\end{proof}

\subsection{Symmetry preservation of the subtransfer operator
$\Tope_\beta$}

The fact that the action of $\Tope_\beta$ commutes with the reflection in the
origin will be needed. The precise formulation reads as follows. Let
$\check\id$ be the antipodal operator $ \check\id f (x):=f(-x)$, which is its
own inverse: $\check\id^2=\id$.

\begin{prop}
Fix $0<\beta\le1$.
Suppose $f:I_1\to\R$ is a function satisfying $\Tope_\beta|f|(x)<+\infty$
for a point $x\in I_1$. Then
\[
\Tope_\beta f(x)=\check\id\Tope_\beta\check\id f(x).
\]
\label{lem-symmetry1}
\end{prop}

\begin{proof}
We keep the notation $s_j(x)=-\beta/(2j+x)$ from Proposition
\ref{prop-convexitypres2}, and note that
\[
s_{-j}(-x)=-s_j(x),
\]
which gives that
\[
\check\id\Tope_\beta \check\id f(x)
=\frac{1}{\beta}\sum_{j\in\Z^\times}[s_{-j}(-x)]^2f(-s_{-j}(-x))
=\frac{1}{\beta}\sum_{j\in\Z^\times}[s_{j}(x)]^2f(s_{j}(x))=
\Tope_\beta f(x).
\]
The assumption $\Tope_\beta|f|(x)<+\infty$ guarantees the absolute convergence
of the above series.
\end{proof}

\subsection{Symmetry, monotonicity, convexity, and the
operator $\Tope_\beta$}

We may now derive the property that $\Tope_\beta$ preserves the class of
functions that are odd and increasing.

\begin{prop} Fix $0<\beta\le1$.
If $f:I_1\to\R$ is odd and (strictly) increasing, then so is
$\Tope_\beta f:I_1\to\R$.
\label{prop-increaspres1}
\end{prop}

\begin{proof}
If $f$ is odd and increasing,  then $|f|$ is even and its restriction to
$I_1^+$ is increasing. From Proposition \ref{prop-convexitypres2}, we get
that $\Tope_\beta|f|(x)<+\infty$ for every $x\in I_1$, so that by Proposition
\ref{lem-symmetry1}, $\Tope_\beta f(x)=-\Tope_\beta f(-x)$, which means that
$\Tope_\beta f$ is \emph{odd}.
Since
\[
\Tope_\beta f(x)=\frac{1}{\beta}\sum_{j\in\Z^\times}[s_j(x)]^2f(s_j(x))
=\frac{1}{\beta}\sum_{j\in\Z^\times}t^2f(t)\big|_{t:=s_j(x)},
\]
where $s_j(x)=-\beta/(2j+x)$ is known to be increasing on $I_1$ for each
$j\in\Z^\times$, it is enough to check that $t^2f(t)$ is increasing in
$t\in I_1$, which in its turn is an immediate consequence of the assumption
that $f$ is odd and increasing. The strict case is analogous.
\end{proof}

We can now derive the property that $\Tope_\beta$ preserves the class of
functions that are positive, even, and convex.

\begin{prop} Fix $0<\beta\le1$.
If $f:I_1\to\R$ is even and convex, and  if $f\ge0$,
then so is $\Tope_\beta f$.
% has the same properties: it is even and convex, and $\Tope_\beta f\ge0$.
\label{prop-convexitypres1}
\end{prop}

\begin{proof}
From Proposition \ref{prop-convexitypres2} we see that
$0\le\Tope_\beta f(x)<+\infty$ holds for each $x\in I_1$.
We keep the notation $s_j(x)=-\beta/(2j+x)$ from Proposition
\ref{prop-convexitypres2}.
Since $f$ is even, we know from Proposition \ref{lem-symmetry1}
that $\Tope_\beta f$ is even as well.
A direct calculation, based on $s_j'(x)=\beta^{-1}[s_j(x)]^2$, shows that
\[
\frac{\diff}{\diff x}\big\{[s_j(x)]^2\,f(s_j(x))\big\}=
\frac{1}{\beta}\big(2t^3f(t)+t^4f'(t)\big)\Big|_{t:=s_j(x)}
\]
where both sides are understood not in the pointwise but in the sense of
distribution theory.
Convexity means that the derivative is increasing, so we need to check
that the left-hand side is increasing as a function of
$x$.
Now, since the function $x\mapsto s_j(x)$ is increasing on $I_1$ for each
$j\in\Z^\times$, the above calculation gives that it is enough to check
that the function $t\mapsto 2t^3f(t)+t^4f'(t)$ is increasing on $I_1$.
By assumption, $f'(t)$ is odd and increasing, and hence
$t^4f'(t)$ is odd and increasing too. Moreover, as $f(t)$ is even and
convex, $f$ is increasing on $I_1$. Thus $t\mapsto t^3f(t)$ is odd and
increasing on $I_1$.
The statement now follows from the fact that the sum of convex functions
is convex as well.
\end{proof}

\subsection{Preservation of continuous functions under 
$\Tope_\beta$}

For $\gamma$ with $0<\gamma<+\infty$, let $C(\bar I_\gamma)$ denote the space
of continuous functions on the compact symmetric interval
$\bar I_\gamma=[-\gamma,\gamma]$. The following is a rather 
immediate observation (the proof is omitted).

\begin{prop}
Fix $0<\beta\le1$.
If $f\in C(\bar I_\beta)$, then $\Tope_\beta f\in C(\bar I_1)$.
\label{prop-endpointpres1.1}
\end{prop}

\begin{prop}
Fix $0<\beta\le1$.
If $f\in C(\bar I_\beta)$ is odd, then $\Tope_\beta f(1)=\beta f(\beta)$.
\label{prop-3.8.2}
\end{prop}

\begin{proof}
By \eqref{eq-Uop.Wop} and the assumption that $f$ is odd, cancellation of all
terms except for the one corresponding to index $j=-1$ gives that
\[
\Tope_\beta f(1)=\sum_{j\in\Z^\times}\frac{\beta}{(2j+1)^2}
f\bigg(-\frac{\beta}{2j+1}\bigg)=\beta f(\beta).
\]
The proof is complete.
\end{proof}

\subsection{Subinvariance of certain key functions}
\label{subsec-subinvariance}
It is well-known that the Gauss map $\sigma_1(x)=\{1/x\}_1$ has the absolutely
continuous invariant measure
\[
\frac{\diff t}{(1+t)\log 2},\qquad t\in I_1^+,
\]
normalized to be a probability measure. This suggests that we should
analyze the behavior of the subtransfer operator $\Sop_\gamma$ on the function
\[
\lambda_1(x):=\frac{1}{1+x},\qquad x\in I_1^+.
\]

\begin{prop}
Fix $0<\gamma\le1$.
With $\lambda_1(x)=1/(1+x)$ on $I_1$, we have that for $n=1,2,3,\ldots$,
\[
\Sop_\gamma^n\lambda_1(x)\le \bigg(\frac{2\gamma}{1+\gamma}\bigg)^n\,
\lambda_1(x),\qquad x\in I_1^+.
\]
\label{prop-Wop.iter}
\end{prop}

\begin{proof}
We first establish the assertion for $n=1$.
It is elementary to establish that for $j=1,2,3,\ldots$,
\[
\frac{\gamma}{(j+x)(j+x+\gamma)}\le\frac{2\gamma}{1+\gamma}\,
\frac{1}{(j+x)(j+x+1)}\,,\qquad x\in I_1^+,
\]
and since
\[
\Sop_\gamma\lambda_1(x)=\sum_{j=1}^{+\infty}\frac{\gamma}{(j+x)^2}\,
\frac{1}{1+\frac{\gamma}{j+x}}=\sum_{j=1}^{+\infty}
\frac{\gamma}{(j+x)(j+x+\gamma)}\,,
\]
the assertion of the proposition for $n=1$ now follows from the telescoping
sum identity
\[
\sum_{j=1}^{+\infty}
\frac{1}{(j+x)(j+x+1)}=\sum_{j=1}^{+\infty}\bigg\{\frac{1}{j+x}-\frac{1}{j+x+1}
\bigg\}=\frac{1}{1+x}\,,\qquad x\in I_1^+.
\]
Finally, the assertion for $n>1$ follows by repeated application of the
$n=1$ case, using that $\Sop_\gamma$ is positive, i.e., it preserves the
positive cone.
\end{proof}

Next, we consider the $\Tope_\beta$-iterates of the function
(for $0<\alpha\le1$)
\begin{equation}
\kappa_\alpha(x):=\frac{\alpha}{\alpha^2-x^2}\,,\qquad x\in I_1.
\label {eq-kappadef1}
\end{equation}
This function is not in $L^1(I_1)$, although it is in $L^{1,\infty}(I_1)$, the
weak $L^1$-space; however, by the observation made in Subsection
\ref{subsec-3.5}, we may still calculate the expression
$\Tope_\beta \kappa_\alpha$ pointwise wherever
$\Tope_\beta|\kappa_\alpha|(x)<+\infty$.
Note that $\kappa_1(x)\diff x$ is the invariant measure for the transformation
$\tau_1(x)=\{-1/x\}_2$, which in terms of the transfer operator $\Tope_1$
means that $\Tope_1\kappa_1=\kappa_1$. It is of fundamental importance in
most of our considerations that this invariant measure has
\emph{infinite mass}, i.e.,  that $\kappa_1\not\in L^1(I_1)$. The reason
for this is that $\tau_1$ has indifferent fixed points. The Gauss map
$\sigma_1$, on the other hand, has only repelling fixed points, and an
invariant measure $\lambda_1(x)\diff x$ with finite mass. This is the main
reason why the transfer operators $\Sop_1$ and $\Tope_1$ behave differently.
We should add that the control of the orbits is much more difficult and not
so well understood in the case of indifferent fixed points, in contrast
with the case of repelling fixed points when the theory is well developed.

\begin{lem} Fix $0<\beta\le1$.
For the function $\kappa_\beta(x)=\beta/(\beta^2-x^2)$, we have that
\[
\Tope_\beta\kappa_\beta(x)=\Tope_\beta|\kappa_\beta|(x)=
\kappa_1(x)=\frac{1}{1-x^2}\,,\qquad \text{a.e.}\,\,\, x\in I_1,
\]
As for the function $\kappa_1(x)=(1-x^2)^{-1}$, we have the estimate
\[
0\le\Tope_\beta\kappa_1(x)\le\beta\,\kappa_1(x)=\frac{\beta}{1-x^2}\,,
\qquad  x\in I_1.
\]
\label{prop-kappa1}
\end{lem}

\begin{proof}
In view of \eqref{eq-Uop.Wop}, we have that
\begin{multline}
\Tope_\beta \kappa_\alpha(x)=\sum_{j\in\Z^\times}\frac{\beta}{(x+2j)^2}
\frac{\alpha}{\alpha^2-[s_j(x)]^2}
\\
=\sum_{j\in\Z^\times}\frac{\beta}{(x+2j)^2}
\frac{\alpha}{\alpha^2-\frac{\beta^2}{(x+2j)^2}}
=\sum_{j\in\Z^\times}\frac{\alpha\beta}{\alpha^2(x+2j)^2-\beta^2}\,,
\label{eq-5.0001}
\end{multline}
where the series converges absolutely unless it happens that a term is
undefined (as the result of division by $0$).
Since $s_j(x)\in I_\beta$ for $x\in I_1$, we see that
each term is positive for $\alpha=\beta$, and hence
\[
\Tope_\beta\kappa_\beta(x)=\Tope_\beta|\kappa_\beta|(x)=\sum_{j\in\Z^\times}
\frac{1}{(x+2j)^2-1}=\frac{1}{2}\sum_{j\in\Z^\times}
\bigg\{\frac{1}{x+2j-1}-\frac{1}{x+2j+1}\bigg\}=\frac{1}{1-x^2}\,,
\]
by telescoping sums, as claimed. Next, since for $0<\beta\le1$ and
$j\in\Z^\times$,
\[
0\le\frac{\beta}{(x+2j)^2-\beta^2}\le\frac{\beta}{(x+2j)^2-1}\,,\qquad
x\in I_1,
\]
it follows that, by the same calculation,
\[
0\le\Tope_\beta\kappa_1(x)\le\sum_{j\in\Z^\times}
\frac{\beta}{(x+2j)^2-1}=\frac{\beta}{1-x^2}\,,
\qquad x\in I_1,
\]
as claimed. The proof is complete.
\end{proof}

\begin{rem}
In particular, for $\beta=1$, we have equality:
$\Tope_1\kappa_1=\kappa_1$.
\end{rem}

We also obtain a uniform estimate of $\Tope_\beta^n\kappa_1$ for $0<\beta<1$
and $n=1,2,3,\ldots$.

\begin{prop}
 Fix $0<\beta<1$.
For $n=1,2,3,\ldots$, we have
that
\[
\Tope_\beta^n\kappa_1(x)\le\frac{2\beta^n}{1-\beta}\,,\qquad x\in I_1.
\]
\label{prop-Uop.iter}
\end{prop}

\begin{proof}
We first establish the asserted estimate for $n=1$.
As the function $\kappa_1(x)=(1-x^2)^{-1}$ is positive, even, and convex,
Proposition \ref{prop-convexitypres2} tells us that

\begin{equation}
\Tope_\beta\kappa_1(x)\le
\beta C_1\kappa_1(\tfrac12\beta)+\frac{\beta}{(2-|x|)^2}
\kappa_1\bigg(\frac{\beta}{2-|x|}\bigg)\le \beta C_1\kappa_1(\tfrac12)+
\beta\kappa_1(\beta)\le\frac{2\beta}{1-\beta}.
\label{eq-case:n=1}
\end{equation}
Here, we used that $\kappa_1$ is increasing on $I_1^+=]0,1[$, and that
$C_1\kappa_1(\frac12)=\frac{4}{3}(\frac{\pi^2}{6}-1)\le1$.

Next, by iteration of Lemma \ref{prop-kappa1}, using that $\Tope_\beta$
is positive, we obtain that $\Tope_\beta^{n-1}\kappa_1\le\beta^{n-1}\kappa_1$,
so that a single application of the estimate \eqref{eq-case:n=1} gives that
\[
\Tope_\beta^{n}\kappa_1(x)=\Tope_\beta\Tope_\beta^{n-1}\kappa_1(x)\le
\beta^{n-1}\Tope_\beta\kappa_1(x)\le\frac{2\beta^n}{1-\beta}\,,\qquad x\in I_1,
\]
as claimed.
\end{proof}

\subsection{The associated transfer operators}
For $0<\beta\le1$ and a function $f\in L^1(I_1)$, extended
to vanish on $\R\setminus I_1$, we let $\Topep_\beta f$ denote the function
defined by
\begin{equation}
\Topep_\beta f(x):=
\begin{cases}
\displaystyle\sum_{j\in\Z}\frac{\beta}{(x+2j)^2}\,f\bigg(-
\frac{\beta}{x+2j}\bigg),\qquad x\in I_1,
\\
0,\qquad\qquad\qquad\qquad\qquad\qquad\, x\in\R\setminus I_1,
\end{cases}
\label{eq-Sop1.002'}
\end{equation}
whenever the sum converges  absolutely.
Analogously, for $0<\gamma\le1$ and a function $f\in L^1(I_1^+)$, extended
to vanish on $\R\setminus I_1^+$, we let $\Sopp_\gamma f$ denote the function
defined by
\begin{equation}
\Sopp_\gamma f(x):=
\begin{cases}
\displaystyle\sum_{j=0}^{+\infty}\frac{\gamma}{(x+j)^2}\,f\bigg(
\frac{\gamma}{x+j}\bigg),\qquad x\in I_1^+,\\
0,\qquad\qquad\qquad\qquad\qquad x\in\R\setminus I_1^+,
\end{cases}
\label{eq-Sop1.002-W}
\end{equation}
whenever the sum converges absolutely. If we compare the definition of
$\Topep_\beta f$ with that of $\Tope_\beta f$, and the definition of
$\Sopp_\gamma f$ with that of $\Sop_\gamma f$, we note that the index $j=0$
is included in the summation this time. The operators
$\Topep_\beta,\,\Sopp_\gamma$ are \emph{transfer operators}.

\begin{prop}
 Fix $0<\beta\le1$.  The operator $\Topep_\beta$ is norm contractive
$L^1(I_1)\to L^1(I_1)$. Indeed, we have that
\[
\int_{-1}^{1}|\Topep_\beta f(x)|\diff x\le \int_{-1}^1|f(x)|\diff x,\qquad
f\in L^1(I_1),
\]
with equality if $f\ge0$.
\label{prop-5.7.1}
\end{prop}

\begin{proof}
As a matter of definition, the function $\Topep_\beta f$ vanishes off  $I_1$.
Next, by the triangle inequality and the change-of-variables formula,
we have that
\begin{multline*}
\int_{-1}^1|\Topep_\beta f(x)|\diff x\le\sum_{j\in\Z}\int_{-1}^{1}
\bigg|f\bigg(-\frac{\beta}{x+2j}\bigg)\bigg|\frac{\beta\,\diff x}{(x+2j)^2}
\\
=\int_{I_1\setminus I_\beta}|f(t)|\diff t+
\sum_{j\in\Z^\times}\int_{-\beta/(2j-1)}^{-\beta/(2j+1)}|f(t)|\diff t=
\int_{-1}^1|f(t)|\diff t,
\end{multline*}
for $f\in L^1(I_1)$, understood to vanish off $I_1$. For $f\ge0$, there is
no loss in the triangle inequality, and we obtain equality.
\end{proof}

\begin{prop}
 Fix $0<\gamma\le1$. The operator $\Sopp_\gamma$ is norm contractive
$L^1(I_1^+)\to L^1(I_1^+)$. Indeed, we have that
\[
\int_{0}^{1}|\Sopp_\gamma f(x)|\diff x\le \int_{0}^1|f(x)|\diff x,\qquad
f\in L^1(I_1^+),
\]
with equality if $f\ge0$.
\label{prop-5.7.1-2.0}
\end{prop}

The proof is analogous to that of Proposition \ref{prop-5.7.1} and therefore
omitted.

\subsection{Aspects from the dynamics of the Gauss-type maps}
We recall the interval notation from Subsection
\ref{subsec-notation-intervals}.
For $0<\beta,\gamma<1$, the transformations $\tau_\beta(x)=\{-\beta/x\}_2$
and $\sigma_\gamma(x)=\{\gamma/x\}_1$ are rather degenerate on the sets
$I_1\setminus\bar I_\beta$ and $I_1^+\setminus \bar I_\gamma^+$.
Indeed, the set $I_1\setminus\bar I_\beta$ is \emph{invariant} for
$\tau_\beta$, as $\tau_\beta(I_1\setminus\bar I_\beta)=I_1\setminus\bar I_\beta$,
and the points in $I_1\setminus\bar I_\beta$ are $2$-periodic, since
\[
\tau_\beta\circ\tau_\beta(x)=\tau_\beta(\tau_\beta(x))=x, \qquad x\in
I_1\setminus\bar I_\beta.
\]
In the same vein, the set $I_1^+\setminus \bar I_\gamma^+$ is invariant for
$\sigma_\gamma$, and all points are $2$-periodic, since
\[
\sigma_\gamma\circ\sigma_\gamma(x)=\sigma_\gamma(\sigma_\gamma(x))=x, \qquad x\in
I_1^+\setminus\bar I_\gamma^+.
\]
Clearly, the set $I_1\setminus\bar I_\beta$ acts as an attractor for the
transformation $\tau_\beta$, and similarly, the set
$I_1^+\setminus \bar I_\gamma^+$ acts as an attractor for the transformation
$\sigma_\gamma$. We would like to analyze the sets of points
which remain outside the attractor in a given number of steps. To this end, we
put, for $N=2,3,4,\ldots$,
\begin{equation}
\begin{array}{r@{}l}
\calE_{\beta,N}&:=\big\{x\in\bar I_\beta:\,\,\tau_\beta^{n}(x)\in \bar I_\beta
\,\,\,\text{for}\,\,\,n=1,\ldots,N-1\big\},\\
\calF_{\gamma,N}&:=\big\{x\in\bar I_\gamma^+:\,\,\sigma_\gamma^{n}(x)\in
\bar I_\gamma^+\,\,\,\text{for}\,\,\,n=1,\ldots,N-1\big\}.
\end{array}
\label{eq-EsetN}
\end{equation}
where $\tau_\beta^{n}:=\tau_\beta\circ\cdots\circ\tau_\beta$ and
$\sigma_\gamma^{n}:=\sigma_\gamma\circ\cdots\circ\sigma_\gamma$
($n$-fold composition). We also agree that $\calE_{\beta,1}:=\bar I_\beta$
and that $\calF_{\gamma,1}:=\bar I_\gamma^+$.
The sets $\calE_{\beta,N}$ and $\calF_{\gamma,N}$ get smaller as $N$ increases,
and we form their intersections
\begin{equation}
\calE_{\beta,\infty}:=\bigcap_{N=1}^{+\infty}\calE_{\beta,N},\qquad
\calF_{\gamma,\infty}:=\bigcap_{N=1}^{+\infty}\calF_{\gamma,N},
\label{eq-Esetinfty}
\end{equation}
which are known as \emph{wandering sets}, and consist of points whose
orbits stay away from the attractor.

\begin{prop} $(0<\beta,\gamma<1)$
For $N=1,2,3,\ldots$, we have the estimates
\[
\int_{\calF_{\gamma,N}}\frac{\diff t}{1+t}\le
\bigg(\frac{2\gamma}{1+\gamma}\bigg)^N\log2\quad\text{and}\quad
\int_{\calE_{\beta,N}}\frac{\diff t}{1-t^2}\le
\frac{4\beta^N}{1-\beta}.
\]
As a consequence, the one-dimensional Lebesgue measures of
the sets $\calE_{\beta,\infty}$ and $\calF_{\gamma,\infty}$ both vanish.
\label{lem-5.8.1}
\end{prop}

\begin{proof}
By inspection of the definition of the Koopman operators \eqref{eq-Cop.Kop},
we see that a.e. on the respective interval,
\[
\Lop_\beta^N1=1_{\calE_{\beta,N}},\qquad \Gmop_\gamma^N1=1_{\calF_{\beta,N}},
\]
where $1$ stands for the corresponding constant function.
In view of the duality \eqref{eq-duality.Uop.Wop.Cop.Kop}, it follows that
\[
\int_{\calF_{\beta,N}}\frac{\diff t}{1+t}=\langle\lambda_1,\Gmop_\gamma^N1
\rangle_{I_1^+}=\langle\Sop_\gamma^N\lambda_1,1\rangle_{I_1^+}\le
\bigg(\frac{2\gamma}{1+\gamma}\bigg)^N\langle\lambda_1,1\rangle_{I_1^+}
=\bigg(\frac{2\gamma}{1+\gamma}\bigg)^N\log2
\]
where $\lambda_1(x)=(1+x)^{-1}$ and the estimate comes from Proposition
\ref{prop-Wop.iter}. It remains to obtain the corresponding estimate for
the set $\calE_{\beta,N}$.
Let $\psi:=1_{I_\eta}\kappa_1$ for some $\eta$, $0<\eta<1$, where
$\kappa_1(x)=(1-x^2)^{-1}$. Then $\psi\in L^1(I_1)$, and we obtain from the
duality \eqref{eq-duality.Uop.Wop.Cop.Kop} together with Proposition
\ref{prop-Uop.iter} that
\[
\int_{I_\eta\cap\calE_{\beta,N}}\frac{\diff t}{1-t^2}=
\langle\psi,\Lop_\beta^N1\rangle_{I_1}=\langle\Tope_\beta^N\psi,1\rangle_{I_1}\le
\langle\Tope_\beta^N\kappa_1,1\rangle_{I_1}\le
\frac{2\beta^N}{1-\beta}\langle 1,1\rangle_{I_1}
=\frac{4\beta^N}{1-\beta}.
\]
Letting $\eta\to1$, the remaining assertion follows by e.g. monotone
convergence.

As for the sets $\calE_{\beta,\infty}$ and $\calF_{\gamma,\infty}$, we just need
to observe that right-hand sides converge to $0$ geometrically, since
$2\gamma/(1+\gamma)<1$.
\end{proof}

The $2$-periodicity of the points in the attractor of $\tau_\beta$  gets
reflected in the fact that the functions supported on the attractor are
two-periodic for the transfer operator $\Topep_\beta$. Naturally, the same
is true in the context of $\sigma_\gamma$  and $\Sopp_\gamma$. We omit the
easy proof.

\begin{prop}
Fix $0<\beta,\gamma\le1$.
The operator $\Topep_\beta$ maps $L^1(I_1\setminus I_\beta)$
contractively into itself. Likewise, $\Sopp_\gamma$ maps
$L^1(I_1^+\setminus I_\gamma^+)$ contractively into itself. Moreover,
$\Topep_\beta^2f=f$ for $f\in L^1(I_1\setminus I_\beta)$, and, analogously,
$\Sopp_\gamma^2f=f$ for $f\in L^1(I_1\setminus I_\gamma)$.
\label{prop-5.8.2.0}
\end{prop}

We shall need the following result, which describes the interlacing of the
iterates of $\Topep_\beta$ with the multiplication by characteristic functions.

\begin{prop}
 Fix $0<\beta\le1$.
For $N=1,2,3,\ldots$ and $f\in L^1(I_1)$, we have the  identities
a.e. on $I_1$:
\[
1_{I_\beta}\Topep_\beta^{N-1} f=\Topep_\beta^{N-1}(1_{\calE_{\beta,N}}f),\qquad
\Topep_\beta^{N}(1_{\calE_{\beta,N}}f)=\Tope_\beta^{N} f .
\]
\label{prop-5.8.2}
\end{prop}

\begin{proof}
To simplify the presentation,  we replace the $L^1(I_1)$ function by a Dirac
point mass $f=\delta_\xi$ at an arbitrary point $\xi\in I_1$. If we can show
that the claimed equalities holds for $f=\delta_\xi$, i.e.,
\[
1_{I_\beta}\Topep_\beta^{N-1}\delta_\xi=\Topep_\beta^{N-1}(1_{\calE_{\beta,N}}\delta_\xi),
\qquad\Topep_\beta^{N}(1_{\calE_{\beta,N}}\delta_\xi)=\Tope_\beta^{N}\delta_\xi,
\]
for Lebesgue almost every point $\xi\in I_1$, then the claimed equalities
hold for every $f\in L^1(I_1)$ by ``averaging''. Indeed, a general
$f\in L^1(I_1)$ may be written as
\begin{equation}
f(x)=\int_{I_1}\delta_x(t)f(t)\diff t=\int_{I_1}\delta_t(x)f(t)\diff t,
\qquad x\in I_1,
\label{eq-ptmassinterpret1}
\end{equation}
where the integral is to be understood in the sense distribution theory, so
that, e.g.,
\begin{equation*}
\Topep_\beta f(x)=\int_{I_1}\Topep_\beta \delta_t(x)f(t)\diff t,
\qquad x\in I_1.
\end{equation*}
We first focus on the claimed identity
\begin{equation}
1_{I_\beta}\Topep_\beta^{N-1}\delta_\xi=\Topep_\beta^{N-1}(1_{\calE_{\beta,N}}\delta_\xi).
\label{eq-9.9.3}
\end{equation}
Here, we should remark that the multiplication of a point mass and a
characteristic function need only make sense for almost every $\xi\in I_1$.
For $N=1$, \eqref{eq-9.9.3} holds trivially. In the following, we
consider integers $N>1$.
The canonical extension of the transfer operator $\Topep_\beta$ to such point
masses reads
\begin{equation}
\Topep_\beta\delta_\xi=\delta_{\tau_\beta(\xi)}=\delta_{\{-\beta/\xi\}_2}.
\label{eq-uopp-on-pointmass}
\end{equation}
Note that by iteration of \eqref{eq-uopp-on-pointmass}, we have
\begin{equation}
\Topep_\beta^{N-1} \delta_\xi=\delta_{\tau_\beta^{N-1}(\xi)}\quad\text{for}\,\,\,\,
\xi\in I_1.
\label{eq-9.9.4}
\end{equation}
As a matter of definition, we know that
$\tau_\beta^{N-1}(\xi)\in\bar I_\beta$ for $\xi\in \calE_{\beta,N}$,  while
for a.e. $\xi\in I_1\setminus\calE_{\beta,N}$, there exists an $n=1,\ldots,N-1$
such that $\tau_\beta^{n}(\xi)\in I_1\setminus\bar I_\beta$. As
$J_\beta=I_1\setminus\bar I_\beta$ is an attractor for $\tau_\beta$,
we conclude that for a.e. $\xi\in I_1\setminus\calE_{\beta,N}$, we have that
$\tau_\beta^{N-1}(\xi)\in I_1\setminus\bar I_\beta$. The asserted identity
\eqref{eq-9.9.3} now follows from these observations.

We turn to the remaining assertion, which claims that
\begin{equation}
\Topep_\beta^{N}(1_{E_{\beta,N}}f)=\Tope_\beta^{N}f,\qquad N=1,2,3,\ldots.
\label{eq-9.9.5}
\end{equation}
By inspection of the definition \eqref{eq-Uop.Wop} of the subtransfer operator,
the action of $\Tope_\beta$ lifts to a point mass at $\xi\in I_1$
for a.e. $\xi$ in the following fashion:
\[
\Tope_\beta\delta_\xi=\begin{cases}
\delta_{\tau_\beta(\xi)}\qquad \text{if} \quad\xi\in \bar I_\beta,
\\
0\qquad\quad\,\,\, \text{if} \quad\xi\in I_1\setminus\bar I_\beta,
\end{cases}
\]
so that by iteration, again for a.e. $\xi\in I_1$, we obtain that
\[
\Tope_\beta^{N}\delta_\xi=\begin{cases}
\delta_{\tau_\beta^{N}(\xi)}\qquad \text{if} \quad\xi\in\calE_{\beta,N},
\\
0\qquad\quad\,\,\, \text{if} \quad\xi\in I_1\setminus\calE_{\beta,N}.
\end{cases}
\]
A comparison with the corresponding formula \eqref{eq-9.9.4} shows that
the identity \eqref{eq-9.9.5} holds. The proof is now complete.
\end{proof}

The corresponding relations for $\Sop_\gamma$ and $\Sopp_\gamma$ read as follows.

\begin{prop}
Fix $ 0<\gamma<1$.
For $N=1,2,3,\ldots$ and $f\in L^1(I_1^+)$, we have the following identities
a.e. on $I_1^+$:
\[
1_{I_\gamma^+}\Sopp_\gamma^{N-1} f=\Sopp_\beta^{N-1}(1_{\calF_{\gamma,N}}f),\qquad
\Sopp_\beta^{N}(1_{\calF_{\gamma,N}}f)=\Sop_\gamma^{N}f.
\]
\label{prop-Wopp1}
\end{prop}

The proof is similar to that of Proposition \ref{prop-5.8.2} and therefore
omitted.

\subsection{Exact endomorphisms}
We need the concept of exactness. Here, we follow the abstract approach
to this notion (see e.g. \cite{Lin}) and say that $\tau_1$ (and the
transfer operator $\Tope_1$ as well) is \emph{exact} if, in the a.e. sense,
\[
\bigcap_{n=1}^{+\infty}\Lop_1^n L^\infty(I_1)=\{\text{constants}\}.
\]
For $0<\beta<1$, however, $\tau_\beta$ has a nontrivial attractor, and the
notion needs to be modified. So, for $0<\beta<1$, we say that $\tau_\beta$
(and the transfer operator $\Tope_\beta$ as well) is \emph{subexact} if, in
the a.e. sense,
\[
\bigcap_{n=1}^{+\infty}\Lop_\beta^n L^\infty(I_1)=\{0\}.
\]
Mutatis mutandis, if we replace the triple $\Tope_\beta,\Lop_\beta,I_1$
by $\Sop_\gamma,\Gmop_\gamma,I_1^+$, we also obtain the definition of exactness
and subexactness for $\Sop_\gamma$ (and the transformation $\sigma_\gamma$ as
well).

\begin{prop}
Fix $0<\beta,\gamma<1$.
The operators $\Tope_\beta:L^1(I_1)\to L^1(I_1)$ and
$\Sop_\gamma:L^1(I_1^+)\to L^1(I_1^+)$ are subexact in the sense that
\[
\bigcap_{n=1}^{+\infty}\Lop_\beta^n L^\infty(I_1)=\{0\},\quad
\bigcap_{n=1}^{+\infty}\Gmop_\gamma^n L^\infty(I_1^+)=\{0\}.
\]
\label{prop-exact1}
\end{prop}

\begin{proof}
By inspection of the compressed Koopman operator $\Lop_\beta^n$, an element
of the intersection
\[
\bigcap_{n=1}^{+\infty}\Lop_\beta^n L^\infty(I_1)
\]
is a function in $L^\infty(I_1)$ which vanishes off the wandering set
$\calE_{\beta,\infty}$, but by Proposition \ref{lem-5.8.1}, this is a null set,
so the function vanishes a.e. The analogous argument applies in the case
of $\Gmop_\gamma$.
\end{proof}

Exactness in the case $\beta=\gamma=1$ is  known and can be derived
from the work of Thaler \cite{Thaler}, see also Aaronson's book \cite{Aar}:

\begin{prop}
Fix $\beta=\gamma=1$.
The operators $\Tope_1:L^1(I_1)\to L^1(I_1)$ and
$\Sop_1:L^1(I_1^+)\to L^1(I_1^+)$ are exact in the sense that
\[
\bigcap_{n=1}^{+\infty}\Lop_1^n L^\infty(I_1)=\{\mathrm{constants}\},\quad
\bigcap_{n=1}^{+\infty}\Gmop_1^n L^\infty(I_1^+)=\{\mathrm{constants}\}.
\]
\label{prop-exact2}
\end{prop}

\begin{proof}
The map $\tau_1$ meets the conditions (1)--(4) of Thaler's paper \cite{Thaler},
p. 69, so by the Theorem 1 \cite{Thaler}, p. 73, $\Tope_1$ is exact (or, in
more standard terminology, $\tau_1$ is exact). Let us check the conditions
one by one, mutatis mutandis, as he uses the interval $[0,1]$ and not
$\bar I_1=[-1,1]$ as we do.

\noindent{\sc Condition (1).} The fundamental intervals are given by
$B(j):=]\frac{1}{2j+1},\frac{1}{2j-1}[$ for $j\in\Z^\times=\Z\setminus\{0\}$
except when $j=\pm1$, when we adjoin an end point: $B(-1)=[-1,-\frac13[$
and $B(1)=]\frac13,1]$. The transformation $\tau_1$ is of $C^2$-class
on each fundamental interval $B(j)$, with $j\in\Z^\times$, and has complete
branches (it is ``filling'' in the terminology of \cite{CHM}). Moreover,
each fundamental interval $B(j)$ contains exactly one fixed point $x_j$,
and $\tau_1'(x_j)>1$ except on two fundamental intervals, $B(-1)$ and
$B(1)$, where the fixed points are the boundary points $1$ and $-1$. On each
fundamental interval $B(j)$ we replace $\tau_1(x)=\{-1/x\}_2$ by the
appropriate branch $\tau_{1,j}(x)=2j-1/x$ (this makes a difference only at the
end points). The derivative at the remaining fixed points is then
$\tau_{1,-1}'(-1)=\tau_{1,1}'(1)=1$.

\noindent{\sc Condition (2).} This condition is satisfied since
$\tau_1'(x)=x^{-2}\ge (1-\epsilon)^{-2}>1$ holds on the interval $I_{1-\epsilon}$
within each fundamental interval $B(j)$.

\noindent{\sc Condition (3).} The derivative $\tau_1'(x)=x^{-2}$ is decreasing
on $]\frac13,1[$ and increasing on $]\!-\!1,-\frac13[$.
The remaining requirements are void.

\noindent{\sc Condition (4).} In each fundamental interval
$B(j)$, the expression $|\tau_1''(x)|/\tau_1'(x)^2=2|x|$ is uniformly bounded.

We conclude from the definition of exactness in \cite{Thaler} that
up to null sets, $\{\emptyset,I_1\}$ are the only measurable subsets of $I_1$
which for each $n=1,2,3,\ldots$ may be written in the form $\tau_1^{-n}(E_n)$
for some measurable set $E_n\subset I_1$. This is equivalent to having
\[
\bigcap_{n=1}^{+\infty}\Lop_1^n L^\infty(I_1)=\{\mathrm{constants}\}.
\]

We turn to the Gauss map $\sigma_1(x)=\{1/x\}_2$, whose exactness is
well-known. But we may derive it from Theorem 1 in \cite{Thaler} as well.
However, the condition (2) is not fulfilled, as $\sigma_1'(x)=-x^{-2}\le -1$
in the interior of the fundamental intervals. But the iterate
$\sigma_1\circ\sigma_1$ is uniformly expanding with
$\inf(\sigma_1\circ\sigma_1)'>1$, and the conditions (1)--(4) may be
verified for it. So the exactness of $\sigma_1\circ\sigma_1$ follows in the
same fashion; this leads to
\[
\bigcap_{n=1}^{+\infty}\Gmop_1^{2n} L^\infty(I_1^+)=\{\mathrm{constants}\},
\]
as required.
\end{proof}

\begin{rem}
Some aspects of the work of Thaler \cite{Thaler} have been further developed
by Melbourne and Terhesiu \cite{MelTer}.
\end{rem}

\subsection{Asymptotical behavior of the orbits of $\Tope_\beta$
and $\Sop_\gamma$}

We now apply the obtained exactness to show how the iterates of $\Tope_\beta$
and $\Sop_\gamma$ behave.

\begin{prop}
Fix $0<\beta,\gamma<1$.

\noindent{\rm(a)} For $f\in L^1(I_1^+)$,
we have that $\|\Sop_\gamma^n f\|_{L^1(I_1^+)}\to0$ as $n\to+\infty$.

\noindent{\rm(b)} For $f\in L^1(I_1)$,
we have that $\|\Tope_\beta^n f\|_{L^1(I_1)}\to0$ as $n\to+\infty$.
\label{prop-exactappl1}
\end{prop}

\begin{proof}
This follows from Proposition \ref{prop-exact1} combined with Theorem 4.3
in \cite{Lin}.
\end{proof}

\begin{prop}
Fix $\beta=\gamma=1$.

\noindent{\rm(a)} For $f\in L^1(I_1^+)$ with $\langle f,1\rangle_{I_1^+}=0$,
we have that $\|\Sop_1^n f\|_{L^1(I_1^+)}\to0$ as $n\to+\infty$.

\noindent{\rm(b)} For $f\in L^1(I_1)$ with $\langle f,1\rangle_{I_1}=0$,
we have that $\|\Tope_1^n f\|_{L^1(I_1)}\to0$ as $n\to+\infty$.
\label{prop-exactappl2}
\end{prop}

\begin{proof}
This follows from Proposition \ref{prop-exact2} combined with Theorem 4.3
in \cite{Lin}.
\end{proof}

There is a weak analogue of Proposition \ref{prop-exactappl1}(b) which applies
for $\beta=1$. The proof is based on the fact that the absolutely continuous
invariant measure has infinite mass.

\begin{prop}
Fix $\beta=1$.
For $f\in L^1(I_1)$, we have that for fixed $\eta$, $0<\eta<1$,
\[
\lim_{n\to+\infty}\int_{-\eta}^{\eta}|\Tope_1^nf(x)|\diff x=0.
\]
\label{prop-weak.convergence1}
\end{prop}

\begin{proof}
Since pointwise $|\Tope_1^nf|\le\Tope_1^n|f|$, we may assume without loss
of generality that $f\ge0$.
We recall the notation $\kappa_1(x)=(1-x^2)^{-1}$, and pick a number $\xi$ with
$0<\xi<1$. Let $g$ be the function
\[
g(x):=\frac{\langle f,1\rangle_{I_1}}{\langle 1_{I_\xi}\kappa_1,1\rangle_{I_1}}
\,1_{I_\xi}(x)\kappa_1(x),\qquad x\in I_1.
\]
Then $g\in L^1(I_1)$, and
\[
\langle f-g,1\rangle_{I_1}=0.
\]
By Proposition \ref{prop-exactappl2}(b), we conclude that
$\|\Tope_1^n (f-g)\|_{L^1(I_1)}\to0$ as $n\to+\infty$. Moreover, by
the triangle inequality, we have that
\[
\|\Tope_1^nf\|_{L^1(I_\eta)}\le\|\Tope_1^n (f-g)\|_{L^1(I_1)}
+\|\Tope_1^ng\|_{L^1(I_\eta)}.
\]
Since the function $g$ is positive, and
\[
g(x)\le\frac{\langle f,1\rangle_{I_1}}{\langle 1_{I_\xi}\kappa_1,1\rangle_{I_1}}
\,\kappa_1(x),
\]
we see that
\begin{equation}
\|\Tope_1^ng\|_{L^1(I_\eta)}=\langle\Tope_1^ng,1_{I_\eta}\rangle_{I_1}\le
\frac{\langle f,1\rangle_{I_1}}{\langle 1_{I_\xi}\kappa_1,1\rangle_{I_1}}
\langle\Tope_1^n\kappa_1,1_{I_\eta}\rangle_{I_1}=
\frac{\langle f,1\rangle_{I_1}}{\langle 1_{I_\xi}\kappa_1,1\rangle_{I_1}}
\langle\kappa_1,1_{I_\eta}\rangle_{I_1},
\label{eq-triangle.ineq1}
\end{equation}
since $\Tope_1\kappa_1=\kappa_1$ (see Lemma \ref{prop-kappa1}).
Moreover, since
\[
\langle 1_{I_\xi}\kappa_1,1\rangle_{I_1}\to+\infty\quad\text{as}\,\,\,\xi\to1,
\]
we may get the norm $\|\Tope_1^ng\|_{L^1(I_\eta)}$ as small as we like for fixed
$\eta$ by letting $\xi$ be appropriately close to $1$. This means that
the right hand side of \eqref{eq-triangle.ineq1} may be as close to $0$
as we want, the first term by letting $n$ be big, and the second, by letting
$\xi$ be close to $1$. The proof is complete.
\end{proof}

\section{Background material: the Hardy and BMO spaces on
the line}

\subsection{The Hardy $H^1$-space; analytic and real}
\label{subsec-H1anreal}
For a reference on the basic facts of Hardy spaces and BMO (bounded mean
oscillation), we refer to, e.g., the monographs of Duren and Garnett
\cite{Dur}, \cite{Gar}, as well as those of Stein \cite{Steinbook1},
\cite{Steinbook2}, and Stein and Weiss \cite{SteinWeissbook}.

\emph{Let $H^1_+(\R)$ and $H^1_-(\R)$ be the subspaces of $L^1(\R)$ consisting
of those functions whose Poisson extensions to the upper half plane
\[
\C_+:=\{z\in\C:\,\,\im z>0\}
\]
are holomorphic and conjugate-holomorphic, respectively}. Here, we use
the term conjugate-holo\-mor\-phic (or anti-holomorphic) to mean that the
complex conjugate of the function in question is holomor\-phic.

It is well-known that any function $f\in H^1_+(\R)$ has vanishing integral,
\begin{equation}
\langle f,1\rangle_\R=\int_\R f(t)\diff t=0,\qquad f\in H^1_+(\R).
\label{eq-int=0}
\end{equation}
In other words, $H^1_+(\C)\subset L^1_0(\R)$, where
\begin{equation}
L^1_0(\R):=\big\{f\in L^1(\R):\,\,\langle f,1\rangle_\R=0 \big\}.
\label{eq-defL0.11}
\end{equation}
In fact, there is a related Fourier analysis characterization of the
Hardy space $H^1_+(\R)$ and $H^1_-(\R)$: for $f\in L^1(\R)$,
\begin{equation}
f\in H^1_+(\R)\,\,\,\Longleftrightarrow \,\,\
\forall y\ge0:\,\,\,\int_\R \e^{\imag yt}f(t)\diff t=0
\label{eq-intcharH1}
\end{equation}
and
\begin{equation}
f\in H^1_-(\R)\,\,\,\Longleftrightarrow \,\,\
\forall y\le0:\,\,\,\int_\R \e^{\imag yt}f(t)\diff t=0.
\label{eq-intcharH1-2}
\end{equation}
We will refer to the space
\[
H^1_\stars(\R):=H^1_+(\R)\oplus H^1_-(\R)
\]
as the \emph{real $H^1$-space of the line $\R$}. Here, the symbol $\oplus$
 means the direct sum, i.e,  the elements of $f\in H^1_\stars(\R)$ are functions
$f\in L^1_0(\R)$ which may be written in the form
\begin{equation}
f=f_1+f_2,\quad\text{ where}\,\,\,f_1\in H^1_+(\R),\,\,f_2\in H^1_-(\R),
\label{eq-H1decomp}
\end{equation}
plus the fact that $H^1_+(\R)\cap H^1_-(\R)=\{0\}$, which is a Fourier-analytic
consequence of \eqref{eq-intcharH1} and \eqref{eq-intcharH1-2}.
Obviously, we have the inclusion $H^1_\stars(\R)\subset L^1_0(\R)$; it is
perhaps slightly less obvious that $H^1_\stars(\R)$ is dense in $L^1_0(\R)$ in
the norm of $L^1(\R)$. It is clear that
the  decomposition \eqref{eq-H1decomp} is unique.
As for notation, \emph{we let $\proj_+$ and $\proj_-$ denote the
projections $\proj_+f:=f_1$ and $\proj_-f:=f_2$ in the decomposition}
\eqref{eq-H1decomp}. These \emph{Szeg\H{o} projections} $\proj_+,\proj_-$ 
can of course be extended beyond this $H^1_\stars(\R)$ setting; more 
about this in the following subsection.

\subsection{The BMO space and the modified Hilbert transform}
With respect to the dual action
\[
\langle f,g\rangle_\R=\int_\R f(t)g(t)\diff t,
\]
we may identify the dual space of $H^1_\stars(\R)$ with $\mathrm{BMO}(\R)/\C$.
Here, $\mathrm{BMO}(\R)$ is the space of functions of
\emph{bounded mean oscillation}; this is the celebrated
\emph{Fefferman duality theorem} \cite{Fef}, \cite{FeSt}. As for notation,
we write ``$\cdot/\C$'' to express that we mod out with respect to the
constant functions.
One of the main results in the theory is the theorem of Fefferman and Stein
\cite{FeSt} which tells us that
\begin{equation}
\mathrm{BMO}(\R)=L^\infty(\R)+\tilde\Hop L^\infty(\R).
\label{eq-split1.3}
\end{equation}
or, in words, a function $g$ is in $\mathrm{BMO}(\R)$ if and only if
it may be written in the form $g=g_1+\tilde\Hop g_2$, where
$g_1,g_2\in L^\infty(\R)$. Here, $\tilde\Hop$ denotes the
\emph{modified Hilbert transform}, defined for $f\in L^\infty(\R)$ by the formula
\begin{equation}
\tilde\Hop f(x):=\frac{1}{\pi}\mathrm{pv}\int_\R f(t)\Bigg\{\frac{1}{x-t}+
\frac{t}{1+t^2}\Bigg\}\diff t=
\lim_{\epsilon\to0^+}\int_{\R\setminus[x-\epsilon,x+\epsilon]} f(t)\Bigg\{\frac{1}{x-t}+
\frac{t}{1+t^2}\Bigg\}\diff t.
\label{eq-tildeHilbert01}
\end{equation}
The decomposition \eqref{eq-split1.3} is clearly not unique. The non-uniqueness
of the decomposition is equal to the intersection space
\begin{equation}
H^\infty_\stars(\R):=L^\infty(\R)\cap\tilde\Hop L^\infty(\R),
\label{eq-split1.3.1}
\end{equation}
the \emph{real $H^\infty$-space}.

We should compare the modified Hilbert transform $\tilde\Hop$ with the
standard \emph{Hilbert transform}
$\Hop$, which acts boundedly on $L^p(\R)$ for $1<p<+\infty$, and maps
$L^1(\R)$ into $L^{1,\infty}(\R)$ for $p=1$. Here, $L^{1,\infty}(\R)$ denotes the
\emph{weak $L^1$-space}, see Subsection \ref{subsec-HilbtransL1} below.
The Hilbert transform of a function $f$, assumed integrable on the line $\R$
with respect to the measure $(1+t^2)^{-1/2}\diff t$, is defined as
the principal value integral
\begin{equation}
\Hop f(x):=\frac{1}{\pi}\mathrm{pv}\int_\R f(t)\frac{\diff t}{x-t}=
\lim_{\epsilon\to0^+}\frac{1}{\pi}
\int_{\R\setminus[x-\epsilon,x+\epsilon]} f(t)\frac{\diff t}{x-t}.
\label{eq-Hilbert02}
\end{equation}
If $f\in L^p(\R)$, where $1\le p<+\infty$, then both $\Hop f$ and
$\tilde\Hop f$ are well-defined a.e., and it is easy to see that the difference
$\tilde\Hop f-\Hop f$ \emph{equals to a constant}.
It is often useful to think of the natural harmonic extensions of the Hilbert
transforms $\Hop f$ and $\tilde\Hop f$ to the upper half-plane $\C_+$ given by
\begin{equation}
\Hop f(z):=\frac{1}{\pi}\int_\R \frac{\re z-t}{|z-t|^2}f(t)\diff t,\quad
\tilde\Hop f(z):=\frac{1}{\pi}\int_\R \Bigg\{
\frac{\re z-t}{|z-t|^2}+\frac{t}{t^2+1}\Bigg\}f(t)\diff t.
\label{eq-hilbtrupper1.1}
\end{equation}
So, as a matter of normalization, we have that $\tilde\Hop f(\imag)=0$.
This tells us the value of the constant mentioned above:
$\tilde\Hop f-\Hop f=-\Hop f(\imag)$.

Returning to the real $H^1$-space, we note the following characterization
of the space in terms of the Hilbert transform:
for $f\in L^1(\R)$,
\[
f\in H^1_\stars(\R) \iff f\in L^1_0(\R)\,\,\,\,\text{and}\,\,\,
\Hop f\in L^1_0(\R);
\]
see Proposition \ref{prop-Katz} later on.

The Szeg\H{o}projections $\proj_+$ and $\proj_-$ which were mentioned 
in Subsection \ref{subsec-H1anreal} are more generally defined in terms 
of the Hilbert transform:
\begin{equation}
\proj_+f:=\frac12(f+\imag \Hop f),\quad \proj_-f:=\frac12(f-\imag \Hop f).
\label{eq-projform1}
\end{equation}
In a similar manner, for $f\in L^\infty(\R)$, based on the modified Hilbert
transform $\tilde\Hop$ we may define the corresponding modified Szeg\H{o} 
projections (which are actually projections modulo the constant functions)
\begin{equation}
\tilde\proj_+f:=\frac12(f+\imag\tilde\Hop f),\quad \tilde\proj_-f
:=\frac12(f-\imag \tilde\Hop f),
\label{eq-projform2}
\end{equation}
so that, by definition, $f=\tilde\proj_+f+\tilde\proj_-f$.
If we are given two functions $f\in  H^1_\stars(\R)$ and $g\in L^\infty(\R)$,
the dual action $\langle\cdot,\cdot\rangle_\R$ naturally splits into
holomorphic and conjugate-holomorphic parts:
\begin{equation}
\langle f,g\rangle_\R=\langle \proj_+f,\tilde\proj_-g\rangle_\R+
\langle \proj_-f,\tilde\proj_+g\rangle_\R.
\label{eq-split1.1}
\end{equation}

Modulo the constants, the space $\mathrm{BMO}(\R)$ naturally splits into
holomorphic and conjugate-holomorphic components:
\begin{equation}
\mathrm{BMO}(\R)/\C=[\mathrm{BMOA}^+(\R)/\C]\oplus[\mathrm{BMOA}^-(\R)/\C].
\label{eq-split1.2}
\end{equation}
The spaces appearing on the right-hand side, $\mathrm{BMOA}^+(\R)$ and
$\mathrm{BMOA}^-(\R)$, denote the subspaces of $\mathrm{BMO}(\R)$ consisting
of functions with Poisson extensions to the upper half-plane $\C_+$ that are
holomorphic and conjugate-holomorphic, respectively.

The operator $\tilde\Hop$ makes sense also on functions from
$\mathrm{BMO}(\R)$. It is then natural to ask what is $\tilde \Hop ^2$:

\begin{lem}
For $f\in L^p(\R)$, $1<p<+\infty$, we have that $\Hop^2f=-f$. Moreover, for
$f\in L^\infty(\R)$, we have that $\tilde\Hop^2f=-f+c(f)$, where $c(f)$ is
the constant
\[
c(f):=\frac{1}{\pi}\int_\R \frac{f(t)}{t^2+1}\diff t.
\]
\end{lem}

\begin{proof}
The assertion for $1<p<+\infty$ is completely standard (see any textbook in
Harmonic Analysis). We turn to the assertion for $p=+\infty$. First, we
observe that without loss of generality, we may assume $f$ is real-valued.
Then the function $2\tilde\proj_+f$ is the holomorphic function in the upper
half-plane whose real part is the Poisson extension of $f$, and the choice
of the imaginary part is fixed by the requirement
$2\im \tilde\proj_+ f(\imag)=\tilde\Hop f(\imag)=0$.
The function
\[
-2\imag\tilde\proj_+ f=\tilde\Hop f-\imag f
\]
extends to a holomorphic function in the upper half-plane $\C_+$, with
real part $\tilde\Hop f$. So we may identify $-f$ with $\tilde \Hop^2f$ up
to an  additive constant. The additive constant is determined by the
requirement that $\tilde\Hop^2f(\imag)=0$, and so
$\tilde\Hop^2f(\imag)=-f+f(\imag)=-f+c(f)$. Here, $f(\imag)$ is understood
in terms of Poisson extension.
\end{proof}

\subsection{BMO and the Fourier transform}
The Fourier transform of a function $f\in L^1(\R)$ is given by
\begin{equation}
%\four
\hat f(x):=\int_\R\e^{\imag\pi xt}f(t)\diff t,
\label{eq-Fouriertrans1D}
\end{equation}
and it is well understood how to extend the Fourier transform
to the setting of tempered distributions (see, e.g., \cite{Hormbook}).
It is well-known how to characterize in terms of the Fourier transform
the spaces $\mathrm{BMOA}^+(\R)$ and $\mathrm{BMOA}^-(\R)$ as subspaces of
$\mathrm{BMO}(\R)$. We state these known facts as a lemma (without supplying
a proof). We recall the notation $\bar\R_+=[0,+\infty[$ and
$\bar\R_-=]-\infty,0]$.

\begin{lem}
Suppose $f\in\mathrm{BMO}(\R)$. Then $f\in\mathrm{BMOA}^+(\R)$ if and only if
$\hat f$ is supported on the interval $\bar\R_-$. Likewise,
$f\in\mathrm{BMOA}^-(\R)$ if and only if $\hat f$ is supported on the
interval $\bar\R_+$.
\label{lem-3.3}
\end{lem}

\subsection{The BMO space of $2$-periodic functions}
\label{subsec-BMO2per}
We shall need the space
\[
\mathrm{BMO}(\R/2\Z):=\{f\in\mathrm{BMO}(\R):\,\, f(t+2)\equiv f(t)\},
\]
that is, the BMO space of $2$-periodic functions. Via the complex exponential
mapping $t\mapsto\e^{\imag\pi t}$ ($\R\mapsto\Te$), we identify the unit circle
$\Te$ with $\R/2\Z$: $\R/2\Z\cong\Te$, and the space $\mathrm{BMO}(\R/2\Z)$
is then just the standard BMO space on $\Te$.
Let us write
\[
\mathrm{BMOA}^+(\R/2\Z):=\mathrm{BMOA}^+(\R)\cap\mathrm{BMO}(\R/2\Z)
\]
and
\[
\mathrm{BMOA}^-(\R/2\Z):=\mathrm{BMOA}^-(\R)\cap\mathrm{BMO}(\R/2\Z),
\]
for the subspaces of $\mathrm{BMO}(\R/2\Z)$ that consist of functions whose
Poisson extensions to the upper half-plane $\C_+$ are holomorphic and
conjugate-holomorphic, respectively.

As $L^2$-integrable functions on the ``circle'' $\R/2\Z$, the elements of
the space $\mathrm{BMO}(\R/2\Z)$ have (a.e. convergent) Fourier series
expansions.
This means that the Fourier transform $\hat f$ of a function
$f\in\mathrm{BMO}(\R/2\Z)$, defined by \eqref{eq-Fouriertrans1D} and
interpreted in the sense of distribution theory,
is a sum of Dirac point masses along the integers $\Z$. We formalize this
observation as a lemma.

\begin{lem}
Suppose $f\in\mathrm{BMO}(\R)$. Then $f\in\mathrm{BMO}(\R/2\Z)$ if and only if
the distribution $\hat f$ is supported on the integers $\Z$, and at each
point of $\Z$, it is a Dirac point mass.
\label{lem-3.4}
\end{lem}

This result is well-known.

\section{The Zariski closures of two portions of the 
lattice-cross}
\label{sec-zarcalc2}

\subsection{An involution and the modified Hilbert transform
on BMO}
\label{subsec-invol}
For a positive real parameter $\beta$, let $\Jop_\beta^*$ be the
involutive operator defined by
\begin{equation}
\Jop_\beta^*f(x):=f(-\beta/x),\qquad x\in\R^\times.
\label{eq-Jop1.1}
\end{equation}

We recall the definition \eqref{eq-tildeHilbert01}
of the modified Hilbert transform $\tilde\Hop$.

\begin{lem}
For $f\in\mathrm{BMO}(\R)$ and a positive real $\beta$, we have that
\[
(\Jop_\beta^*\tilde\Hop f)(x)=(\tilde\Hop\Jop_\beta^* f)(x)+c_\beta(f),
\]
where $c_\beta(f)$ is the constant
\[
c_\beta(f):=\tilde\Hop f(\imag\beta)
=(\beta^2-1)\int_\R\frac{tf(t)\,\diff t}{(1+t^2)(\beta^2+t^2)}.
\]
\label{lem-Jbetacomm1.1}
\end{lem}

\begin{proof}
Without loss of generality, we may assume that $f$ is real-valued. The
mapping $x\mapsto -\beta/x$ extends to a conformal automorphism of the upper
half-plane given by $z\mapsto-\beta/z$, and the function
$2\tilde\proj_+ f$ is a holomorphic function in the upper half plane with
real part equal to the Poisson extension of $f$. We realize that the
functions $\Jop_\beta^*\tilde\proj_+f$ and  $\Jop_\beta^*\tilde\proj_+f$ differ by
an imaginary constant. Taking imaginary parts, the result follows by
plugging at the point $z=\imag$.
\end{proof}

\subsection{The Zariski closure of the portions of
the lattice-cross on the space-like cone boundary}

Recall that $1_E$ stands the characteristic function of the set $E$, which
equals $1$ on $E$ and $0$ off of $E$. The  Fourier transform
of the function $\e^{\imag/t}$ in the sense of Schwartzian distributions
may be known, but we have no specific reference.

\begin{prop}
 In the sense of distribution theory on the real line $\R$ we have that,
\[
\lim_{\epsilon\to0^+}\int_\R \e^{\imag/t+\imag tx-\epsilon|t|}
\frac{\diff t}{2\pi}
=\delta_0(x)-1_{\R_+}(x) x^{-1/2}J_1(2x^{1/2}),
\]
where $\delta_0$ is the unit
Dirac point mass at $0$, and $J_1$ denotes the standard Bessel function,
so that
\[
x^{-1/2}J_1(2x^{1/2})=\sum_{j=0}^{+\infty}\frac{(-1)^j}{j!(j+1)!}x^j.
\]
\label{prop-3.5}
\end{prop}

\begin{proof}
A direct calculation can be obtained on the basis of formula 3.324 in
\cite{GrRy}. A less cumbersome approach is to compute  the Fourier
transform of the function $H_1(x):=1_{\R_+}(x)x^{-1/2}J_1(2x^{1/2})$:
\[
\hat H_1(y)=\int_\R \e^{\imag\pi xy}H_1(x)\diff x=\int_0^{+\infty}\e^{\imag\pi xy}
x^{-1/2}J_1(2x^{1/2})\diff x=2\int_0^{+\infty}\e^{\imag\pi yt^2}
J_1(2t)\diff t,
\]
where the integral is absolutely convergent for $\im y>0$ and has a
well-defined interpretation on the real line $\R$, e.g., in terms of
nontangential limits. From the standard Bessel function asymptotics, we
know that
\[
|H_1(x)|=\mathrm{O}(x^{-3/4})\quad\text{as}\,\,\,\,x\to+\infty,
\]
so that, in particular, $H_1\in L^2(\R)$. By basic Hardy space theory,
the nontangential limit interpretation from the upper half-plane agrees with
the standard $L^2$ Fourier transform on the line $\R$.
By application of formula 6.631 in \cite{GrRy}, we have that, for $\im y>0$,
\[
\hat H_1(y)=2\int_0^{+\infty}\e^{\imag\pi yt^2}
J_1(2t)\diff t=\e^{-\imag/(2\pi y)}M_{0,\frac12}\bigg(\frac{\imag}{\pi y}\bigg),
\]
where the function on the right-hand side is of \emph{Whittaker type}.
In view of the integral representation of such Whittaker functions (formula
9.221 in \cite{GrRy}) we find that
\[
\hat H_1(y)=1-\e^{-\imag/(\pi y)},\qquad \im y>0,
\]
and, in a second step, that the above identification of the Fourier transform
holds in the $L^2$-sense a.e. on $\R$. Since the Fourier transform of the
Dirac delta $\delta_0$ is the constant function $1$, the assertion of the
proposition now follows from the Fourier inversion formula.
\end{proof}

\begin{proof}[Proof of Theorem \ref{prop-zariski1}]
We obviously have the inclusions
\[
\Lambda_{\alpha,\beta}^{++}\subset
\mathrm{zclos}_{\Gamma_M}(\Lambda_{\alpha,\beta}^{++}),
\qquad
\Lambda_{\alpha,\beta}^{--}\subset\mathrm{zclos}_{\Gamma_M}(\Lambda_{\alpha,\beta}^{--}),
\]
and it remains to show that the Zariski closure contains no extraneous points.
We will focus our attention to the set $\Lambda_{\alpha,\beta}^{++}$; the
treatment of the set $\Lambda_{\alpha,\beta}^{--}$ is analogous.
In view of \eqref{eq-1.3} (which relates $\hat\mu(\xi)$ to the compressed
measure $\bpi_1\mu$) we need to do the following. Given a point
$\xi^\star=(\xi_1^\star,\xi_2^\star)\in\R^2\setminus\Lambda_{\alpha,\beta}^{++}$, we
need to find a finite complex-valued absolutely continuous Borel measure
$\nu$ on $\R^\times$, such that
\[
\int_{\R^\times}
\e^{\imag\pi[\xi_1^\star t+M^2\xi_2^\star/(4\pi^{2}t)]}
\diff\nu(t)\ne0,
\]
while at the same time
\[
\int_{\R^\times}
\e^{\imag\pi\alpha mt}\diff\nu(t)=\int_{\R^\times}
\e^{\imag M^2\beta n/(4\pi t)}
\diff\nu(t)=0,\quad m,n\in\Z_{+,0}.
\]
By  a scaling argument, we may without loss of generality restrict our
attention to the normalized case $\alpha:=1$ and $M:=2\pi$.
As $\nu$ is absolutely continuous, we may write $\diff\nu(t):=g(t)\diff t$,
where $g\in L^1(\R)$. Given the above normalization, we need $g$ to satisfy
\begin{equation}
\int_{\R^\times}
\e^{\imag\pi[\xi_1^\star t+\xi_2^\star/t]}
g(t)\diff t\ne0,
\label{eq-nonannih1}
\end{equation}
where
\[
(\xi_1^\star,\xi_2^\star)\in\R^2\setminus
[(\Z_{+,0}\times\{0\})\cup(\{0\}\times
\beta\Z_{+})],
\]
while at the same time
\begin{equation}
\int_{\R^\times}
\e^{\imag\pi mt}g(t)\diff t=\int_{\R^\times}
\e^{\imag\pi\beta n/t}g(t)\diff t=0,\quad m,n\in\Z_{+,0}.
\label{eq-annih1.1}
\end{equation}
We will try to find such a function $g$ in the slightly smaller space
$H^1_\stars(\R)$.
To analyze the condition \eqref{eq-annih1.1}, we might as well study the
weak-star closures in the dual space $\mathrm{BMO}(\R)/\C$ of the linear spans
of (i) the functions $t\mapsto\e^{\imag\pi mt}$, with $m\in\Z_{+,0}$, and of
(ii) the functions $t\mapsto \e^{\imag \pi\beta n/t}$, with $n\in\Z_{+,0}$.
In the first case, we obtain the subspace
$\mathrm{BMOA}^+(\R/2\Z)/\C$ (see Subsection \ref{subsec-BMO2per} for the
notation).
In the second case, we obtain instead the subspace
$\mathrm{BMOA}^-_{\langle\beta\rangle}(\R)/\C$, where
$\mathrm{BMOA}^-_{\langle\beta\rangle}(\R)=\Jop_\beta^*\mathrm{BMOA}^-(\R/2\Z)$
and the operator $\Jop_\beta^*$ is as in \eqref{eq-Jop1.1}.
Now, for $g\in H^1_\stars(\R)$, \eqref{eq-annih1.1} expresses that $g$
annihilates the sum space
$\mathrm{BMOA}^+(\R/2\Z)+\mathrm{BMOA}^-_{\langle\beta\rangle}(\R)$.

To simplify the notation, we let $F_0\in L^\infty(\R)$ be the function
$F_0(t):=\e^{\imag\pi[\xi_1^\star t+\xi_2^\star/t]}$. Then, in view of
\eqref{eq-split1.1}, we have that
\[
\langle g,F_0\rangle_{\R}=\langle \proj_+g,\tilde\proj_-F_0\rangle_\R+
\langle \proj_-g,\tilde\proj_+F_0\rangle_\R.
\]
It follows that if we can obtain that
\begin{equation}
\tilde\proj_+F_0\notin\mathrm{BMOA}^+(\R/2\Z)\quad\text{or}\quad
\tilde\proj_-F_0\notin\mathrm{BMOA}^-_{\langle\beta\rangle}(\R),
\label{eq-test1.1}
\end{equation}
then we are done, because we are free to choose $g\in H^1_\stars(\R)$ as
we like.
Indeed, if $\tilde\proj_+F_0\notin\mathrm{BMOA}^+(\R/2\Z)$, then we just pick a
$g\in H^1_-(\R)$ which does not annihilate $\mathrm{BMOA}^+(\R/2\Z)$, and if
$\tilde\proj_-F_0\notin\mathrm{BMOA}^-_{\langle\beta\rangle}(\R)$,
then we just pick a $g\in H^1_+(\R)$ which does
not annihilate $\mathrm{BMOA}^-_{\langle\beta\rangle}(\R)$. In each case, we achieve
\eqref{eq-nonannih1}. Using $\Jop_\beta^*$, we see by Lemma
\ref{lem-Jbetacomm1.1}) that \eqref{eq-test1.1} is equivalent to
having
\begin{equation}
\tilde\proj_+F_0\notin\mathrm{BMOA}^+(\R/2\Z)\quad\text{or}\quad
\tilde\proj_-\Jop_\beta^* F_0\notin\mathrm{BMOA}^-(\R/2\Z).
\label{eq-test1.2}
\end{equation}
Moreover, the function $F_1:=\Jop_\beta^* F_0$ is of the same general type as
$F_0$: $F_1(t)=\e^{-\imag\pi[\eta_1^\star t+\eta_2^\star/t]}$, where
$\eta_1^\star:=\xi_2^\star /\beta$ and
$\eta_2^\star:=\beta\xi_1^\star$.  We can bring this one step further,
and consider $F_2(t):=\e^{\imag\pi[\eta_1^\star t+\eta_2^\star/t]}$
(this is just the complex conjugate of $F_1(t)$), and express
the requirement \eqref{eq-test1.2} in the form
\begin{equation}
\tilde\proj_+F_0\notin\mathrm{BMOA}^+(\R/2\Z)\quad\text{or}\quad
\tilde\proj_+F_2\notin\mathrm{BMOA}^+(\R/2\Z).
\label{eq-test1.3}
\end{equation}
By combining Lemmas \ref{lem-3.3} and \ref{lem-3.4} with Proposition
\ref{prop-3.5} in the appropriate manner, using that the Bessel function
$J_1$ is real-analytic (so that its zero set is a discrete set of points),
we find that
\[
\tilde\proj_+F_0\in \mathrm{BMOA}^+(\R/2\Z)\quad\Longleftrightarrow\quad
\xi^\star=(\xi_1^\star,\xi_2^\star)
\in (\bar\R_-\times\bar\R_+)\cup(\Z_+\times\{0\}).
\]
The analogous case with $F_2$ in place of $F_0$ reads
\[
\tilde\proj_+F_2\in \mathrm{BMOA}^+(\R/2\Z)\quad\Longleftrightarrow\quad
\xi^\star(\xi_1^\star,\xi_2^\star)\in
(\bar\R_+\times\bar\R_-)\cup(\{0\}\times\beta\Z_+).
\]
As we put these assertion together, it becomes clear that
\[
\tilde\proj_+F_0,\tilde\proj_+F_2\in\mathrm{BMOA}^+(\R/2\Z)
\quad\Longleftrightarrow\quad
(\xi_1^\star ,\xi_2^\star )\in(\Z_{+,0}\times\{0\})\cup
(\{0\}\times\beta\Z_+).
\]
The set of $\xi^\star$ in the right-hand side expression is precisely the
excluded set of points on the lattice-cross, and we conclude that
\eqref{eq-test1.3} must hold.
This completes the proof of the theorem.
\end{proof}

%\end{document}

\section{Dynamic unique continuation from one branch
of the hyperbola to the other}
\label{sec-dynamicuniqcont}

\subsection{Dynamic unique continuation and the critical density case}
We recall the definition of the hyperbola $\Gamma_M$ and its branch
$\Gamma_M^+$ from the introduction, see \eqref{eq-1.2} and
\eqref{eq-Gbranch1.1}.
Here, we will supply the proof of Theorem \ref{thm-2.1}. As Theorem
\ref{thm-2.1} is somewhat defective at the critical regime
$\alpha\beta M^2=16\pi^2$, we may ask whether adding an additional point to
the lattice-cross $\Lambda_{\alpha,\beta}$ might improve the situation. Indeed,
this turns out to be the case, provided that the point we add is on the
cross (but not on the lattice-cross itself, of course):

\begin{thm}
Fix $0<\alpha,\beta,M<+\infty$.
Suppose $\alpha\beta M^2=16\pi^2$, and pick a point
$\xi^\star\in(\R\times\{0\})\times(\{0\}\times\R)$ on the cross, which is not
in $\Lambda_{\alpha,\beta}$. If we write
$\Lambda^\star_{\alpha,\beta}:=\Lambda_{\alpha,\beta}\cup\{\xi^\star\}$,
then $(\Gamma_M^+,\Lambda^\star_{\alpha,\beta})$ is a Heisenberg uniqueness
pair.
\label{cor-onebranch}
\end{thm}

Theorem \ref{cor-onebranch} has a reformulation in terms of unique continuation
from $\Gamma_M^+$ to $\Gamma_M$, which we think of as an example of
\emph{dynamic unique continuation}.

\begin{cor}
Fix $0<\alpha,\beta,M<+\infty$.
Suppose $\alpha\beta M^2=16\pi^2$, and pick a point
$\xi^\star\in(\R\times\{0\})\times(\{0\}\times\R)$ on the cross, which is not
in $\Lambda_{\alpha,\beta}$. If we write
$\Lambda^\star_{\alpha,\beta}:=\Lambda_{\alpha,\beta}\cup\{\xi^\star\}$,
then any measure $\mu\in\mathrm{AC}(\Gamma_M,\Lambda_{\alpha,\beta}^\star)$
is uniquely determined by its restriction to the hyperbola branch
$\Gamma_M^+$.
\label{cor-onebranch2}
\end{cor}

We first supply the proof of Theorem \ref{thm-2.1}, and then proceed with
the proof of Theorem \ref{cor-onebranch}.

\begin{proof}[Proof of Theorem \ref{thm-2.1}]
We pick an arbitrary measure
$\mu\in\mathrm{AC}(\Gamma_{M},\Lambda_{\alpha,\beta})$ and form its
$x_1$-compres\-sion $\nu:=\bpi_1\mu$ which is a finite  absolutely
continuous complex measure on the positive half-axis $\R_+$.
 By a scaling argument, we may assume that
that
\[
\alpha=2,\quad M=2\pi.
\]
Since $\nu$ is absolutely continuous, we may write $\diff\nu(t)=f(t)\diff t$,
where $f\in L^1(\R_+)$.
We observe that the vanishing condition $\hat\mu=0$ on $\Lambda_{\alpha,\beta}$
with $\alpha=2$ and $M=2\pi$ amounts to having
\begin{equation}
\int_{\R_+}\e^{\imag 2\pi mt}f(t)\diff t=\int_{\R_+}
\e^{\imag2\pi\gamma n/t}f(t)\diff t=0,\qquad m,n\in\Z,
\label{eq-fusb1}
\end{equation}
where $\gamma:=\beta/2$. It was shown in \cite{CHM} that for $2<\beta<+\infty$,
there is an infinite-dimensional space of solutions $f$. So, in the sequel,
\emph{we will restrict the parameter $\beta$ to $0<\beta\le2$, and hence
$\gamma$ to $0<\gamma\le1$}. To complete the proof of the theorem, we
need to show that

\smallskip

(i) for $0<\gamma<1$, the condition \eqref{eq-fusb1} entails
that $f=0$ holds a.e. on $\R_+$, whereas

(ii) for $\gamma=1$, \eqref{eq-fusb1}
implies that $f=C_0f_0$ holds a.e. on $\R_+$ for some constant $C_0$, where
$f_0$ is the function
\begin{equation}
f_0(t):=\frac{1_{[0,1]}(t)}{1+t}-\frac{1_{[1,+\infty[}(t)}{t(1+t)}.
\label{eq-f0.101}
\end{equation}

\smallskip

As a first step, we rewrite \eqref{eq-fusb1} in the form
\begin{equation}
\int_{\R_+}\e^{\imag2\pi mt}f(t)\diff t=\int_{\R_+}
\e^{\imag2\pi nt}f\bigg(\frac{\gamma}{t}\bigg)\frac{\diff t}{t^2}=0,
\qquad m,n\in\Z.
\label{eq-fusb2}
\end{equation}
Next,  for a function $g\in L^1(\R_+)$ and
an integer $m\in\Z$ we have that
that
\begin{multline}
\int_{\R_+}\e^{\imag2\pi mt}g(t)\diff t=\sum_{j=0}^{+\infty}\int_{[j,j+1]}\e^{\imag2\pi mt}
g(t)\diff t
\\
=\sum_{j=0}^{+\infty}\int_{[0,1]}\e^{\imag2\pi mt}
g(t+j)\diff t=\int_{[0,1]}\e^{\imag2\pi mt}\sum_{j=0}^{+\infty}g(t+j)\diff t.
\label{eq-Fourier-calc1.1}
\end{multline}
Together with the uniqueness theorem for Fourier series,
\eqref{eq-Fourier-calc1.1} now shows that
\begin{equation}
\int_{\R_+}\e^{\imag2\pi mt}g(t)\diff t=0 \quad \forall m\in\Z \quad\Longleftrightarrow
\quad \sum_{j=0}^{+\infty}g(t+j)=0\,\,\,\mathrm{a.e.}\,\,\,\mathrm{on}\,\,\,
\R_+.
\label{eq-Fourier-calc1.2}
\end{equation}
As we apply \eqref{eq-Fourier-calc1.2} to the two cases $g(t)=f(t)$ and
$g(t)=t^{-2}f(\gamma/t)$, the conditions of \eqref{eq-fusb2} find an
equivalent formulation:
\begin{equation}
\sum_{j=0}^{+\infty}f(t+j)=\sum_{j=0}^{+\infty}\frac{1}{(t+j)^2}
f\bigg(\frac{\gamma}{t+j}\bigg)=0\,\,\,\,\mathrm{a.e.}\,\,\mathrm{on}\,\,\,\,
\R_+.
\label{eq-fusb3}
\end{equation}
We single out the first term in each sum, and rewrite \eqref{eq-fusb3} further:
\begin{equation}
f(t)=-\sum_{j=1}^{+\infty}f(t+j),\quad \frac{1}{t^2}\,
f\bigg(\frac{\gamma}{t}\bigg)=-\sum_{j=1}^{+\infty}\frac{1}{(t+j)^2}
f\bigg(\frac{\gamma}{t+j}\bigg),
\label{eq-fusb4}
\end{equation}
in both cases a.e. on $\R_+$. After the change-of-variables
$t\mapsto \gamma/t$ in the second condition, \eqref{eq-fusb4} becomes
\begin{equation}
f(t)=-\sum_{j=1}^{+\infty}f(t+j),\quad f(t)=-
\sum_{j=1}^{+\infty}\frac{\gamma^2}{(\gamma+jt)^2}
f\bigg(\frac{\gamma t}{\gamma+jt}\bigg),
\label{eq-fusb5}
\end{equation}
again a.e. on $\R_+$.
By combining the conditions of equality in \eqref{eq-fusb5}, we find that
\begin{equation}
f(t)=\sum_{j,l=1}^{+\infty}\frac{\gamma^2}{[\gamma+l(j+t)]^2}
\,f\bigg(\frac{\gamma(t+j)}{\gamma+l(t+j)}\bigg),\quad\mathrm{a.e.}\,\,\,
\mathrm{on}\,\,\,\R_+.
\label{eq-fusb7.3}
\end{equation}
Now, it is easy to check that after restriction to the interval $I_1^+=]0,1[$,
condition \eqref{eq-fusb7.3} amounts to having
\begin{equation}
f=\Sop_\gamma^2f\quad\mathrm{a.e.}\,\,\,\mathrm{on}\,\,\,I_1^+,
\label{eq-fusb7.4}
\end{equation}
where $\Sop_\gamma$ is the subtransfer operator as given by \eqref{eq-Uop.Wop}.
If $0<\gamma<1$, Proposition \ref{prop-exactappl1}(a) tells us that
$\Sop_\gamma^{2n}f\to0$ in $L^1(I_1^+)$ as $n\to+\infty$, so the only way
the equality \eqref{eq-fusb7.4} is possible is if $f=0$ a.e. on $I_1^+$.
But then the second equality in \eqref{eq-fusb5} gives that $f=0$ a.e. on
$\R\setminus I_1^+$, and hence $f=0$ a.e. on $\R_+$, as desired.
This settles (i).

\emph{We turn to the remaining case} $\gamma=1$.
It is well-known that the function $\lambda_1(t)=(1+t)^{-1}$ is an
invariant density on $I_1^+$ for the Gauss map $\theta_1(t)=\{1/t\}_1$
(cf. Subsection \ref{subsec-subinvariance}). In terms of the transfer
operator $\Sop_1$, this means that $\Sop_1\lambda_1=\lambda_1$, so that
$\Sop_1^2\lambda_1=\lambda_1$ as well. Next, we consider the function
\[
h:=f-\frac{\langle 1,f\rangle_{I_1^+}}{\log2}\lambda_1\in L^1(I_1^+),
\]
which by construction has $\langle h,1\rangle_{I_1^+}=0$ and
$h=\Sop_1^2h$. By iteration, the latter property entails that $h$ has
$h=\Sop_1^{2n}h$ for $n=1,2,3,\ldots$, so that in view of Proposition
\ref{prop-exactappl2}(a),we have that
\[
h=\Sop_1^{2n}h\to0\quad\text{as}\,\,\,\,n\to+\infty,
\]
where the convergence is in the norm of $L^1(I_1^+)$, which implies that
$h=0$ a.e. on $I_1^+$. It is now immediate that
\[
f=C_0\lambda_1\,\,\,\,\text{a.e. on}\,\,\,\,I_1^+,\quad\text{where}\,\,\,\,
C_0:=\frac{\langle 1,f\rangle_{I_1^+}}{\log2}\in\C.
\]
Next, the second identity in \eqref{eq-fusb5}  with $\gamma=1$ tells us what
$f_0$ equals on the remaining set $\R_+\setminus I_1^+$:
\begin{multline*}
f(t)=-C_0
\sum_{j=1}^{+\infty}\frac{1}{(1+jt)^2}
\lambda_1\bigg(\frac{t}{1+jt}\bigg)=-C_0
\sum_{j=1}^{+\infty}\frac{1}{(1+jt)^2}
\frac{1}{1+\frac{t}{1+jt}}
\\
=-C_0\sum_{j=1}^{+\infty}\frac{1}{(1+jt)(1+(j+1)t)}=-\frac{C_0}{t}
\sum_{j=1}^{+\infty}\bigg\{\frac{1}{1+jt}-\frac{1}{1+(j+1)t}\bigg\}
=-\frac{C_0}{t(1+t)}.
\end{multline*}
The conclusion that $f=C_0f_0$ a.e. on $\R_+$ is now immediate,
where $f_0$ is given by \eqref{eq-f0.101} and $C_0\in\C$ is a constant.
Finally, it is an exercise to verify that the function $f_0$ indeed
satisfies \eqref{eq-fusb3}, so that $f_0$ (and its
complex constant multiples) meets the vanishing condition for the Fourier
transform, as expressed in \eqref{eq-fusb1}. This settles (ii), and the
the proof is complete.
\end{proof}

\begin{proof}[Proof of Theorem \ref{cor-onebranch}]
As before, rescaling allows us to fix the parameter values:
\[
\alpha=\beta=2,\,\,\, M=2\pi,
\]
which corresponds to $\gamma=\beta/2=1$ in the preceding proof.
We need to show that if $\mu\in\mathrm{AC}(\Gamma_M,\Lambda^\star_{\alpha,\beta})$,
then $\mu=0$ as a measure. Since
$\Lambda^\star_{\alpha,\beta}\supset\Lambda_{\alpha,\beta}$, and we are in the critical
parameter regime in terms of Theorem \ref{thm-2.1}, we have that
necessarily $\diff\bpi_1\mu(t)=C_0f_0(t)$, where $f_0$ is given by
\eqref{eq-f0.101}, and $C_0$ is a complex constant.
We recall that $\Lambda^\star_{\alpha,\beta}=\Lambda^\star_{\alpha,\beta}\cup
\{\xi^\star\}$, for some point $\xi^\star=(\xi^\star_1,\xi^\star_2)$ with
either $\xi^\star_1=0$ or $\xi^\star_2=0$, which is not the lattice-cross
$\Lambda_{\alpha,\beta}$. By symmetry, both cases are equivalent, and we choose
to consider $\xi^\star_2=0$, so that $\xi^\star=(\xi^\star_1,0)$, where
$\xi_1^\star\in\R\setminus\alpha\Z=\R\setminus2\Z$. The Fourier transform of
$\mu$ restricted to the axis $\R\times\{0\}$ equals (cf. \eqref{eq-1.3})
\begin{multline}
\hat\mu(\xi_1,0)=\int_{\R^\times}\e^{\imag\pi\xi_1t}\diff\bpi_1\mu(t)
=C_0\int_{\R^\times}\e^{\imag\pi\xi_1t}f_0(t)\diff t
\\
=C_0\Bigg\{\int_{[0,1]}
\e^{\imag\pi\xi_1 t}\frac{\diff t}{1+t}-\int_{[1,+\infty[}
\e^{\imag\pi\xi_1 t}\frac{\diff t}{t(1+t)}\Bigg\}
\\
=C_0\Bigg\{\int_{[0,1]}
\e^{\imag\pi\xi_1 t}\frac{\diff t}{1+t}-\int_{[1,+\infty[}
\e^{\imag\pi\xi_1 t}\bigg(\frac{1}{t}-\frac{1}{1+t}\bigg)\diff t\Bigg\}
\\
=C_0\Bigg\{\int_{[0,+\infty[}
\e^{\imag\pi\xi_1 t}\frac{\diff t}{1+t}-\int_{[1,+\infty[}
\e^{\imag\pi\xi_1 t}\frac{\diff t}{t}\Bigg\}=C_0(\e^{-\imag\pi\xi_1}-1)\int_{[1,+\infty[}
\e^{\imag\pi\xi_1 t}\frac{\diff t}{t}.
\label{eq-9.1.12.11}
\end{multline}
Here, in the rightmost expression, the integral should be understood as a
generalized Riemann integral.
Since our additional vanishing condition is $\hat\mu(\xi_1^\star,0)=0$,
above calculation \eqref{eq-9.1.12.11} tells us that this is the same as
\[
C_0(\e^{-\imag\pi\xi_1^\star}-1)\int_{[1,+\infty[}
\e^{\imag\pi\xi_1^\star t}\frac{\diff t}{t}=0.
\]
Moreover, since $\xi_1^\star$ is real but not an even integer, we know that
$\e^{\imag\pi\xi_1^\star}\neq1$, and the above equation simplifies to
\begin{equation}
C_0\int_{[1,+\infty[}
\e^{\imag\pi\xi_1^\star t}\frac{\diff t}{t}=0.
\label{eq-9.1.13.11}
\end{equation}
Splitting the above generalized Riemann integral into real and imaginary parts,
we see that
\[
\int_{1}^{+\infty}\e^{\imag \pi\xi_1^\star t}\frac{\diff t}{t}
=\int_{1}^{+\infty}\cos(\pi\xi_1^\star t)\frac{\diff t}{t}+
\imag\int_{1}^{+\infty}\sin(\pi\xi_1^\star t)\frac{\diff t}{t}.
\]
The real and imaginary parts may be expressed in terms the rather standard
functions ``si'' and ``ci'':
\[
\int_{1}^{+\infty}\cos(\pi\xi^\star_1 t)\frac{\diff t}{t}
=\int_{\pi|\xi_1^\star|}^{+\infty}\frac{\cos y}{y}\,\diff y=-\ci(\pi|\xi_1^\star|),
\]
and
\[
\int_{1}^{+\infty}\sin(\pi\xi_1^\star t)\frac{\diff t}{t}=
\sign(\xi_1^\star)\int_{\pi|\xi_1^\star|}^{+\infty}\frac{\sin y}{y}\,\diff y
=-\sign(\xi_1^\star)\si(\pi|\xi_1^\star|),
\]
so that
\[
\int_{1}^{+\infty}\e^{\imag \pi\xi_1^\star t}\frac{\diff t}{t}
=-\ci(\pi|\xi_1^\star|)-\imag\sign(\xi_1^\star)\si(\pi|\xi_1^\star|).
\]
Here, we write $\sign(x)=x/|x|$ for the standard sign function.
We now observe that is rather well-known that the parametrization
\[
\ci(\pi x)+\imag\si(\pi x),\qquad 0<x<+\infty,
\]
forms the {\em Nielsen} (or sici) spiral which converges to the origin as
$x\to+\infty$, and whose curvature is proportional to $x$
(see, e.g. \cite{PlM}). We will only need the fact that the spiral 
never intersects the origin:
\begin{equation}
\int_{[1,+\infty[}\e^{\imag\pi\xi_1^\star t}\frac{\diff t}{t}\ne0.
\label{eq-Nielsen1.1}
\end{equation}
Given that \eqref{eq-Nielsen1.1} holds, \eqref{eq-9.1.13.11} gives us that 
$C_0=0$ and consequently that $\mu=0$ as a measure, and the assertion of the
theorem follows.

It remains to derive the property \eqref{eq-Nielsen1.1}. 
For positive $x$, we put
\[
\rho(x):=\bigg|\int_1^{+\infty}\e^{\imag x t}\frac{\diff t}{t}\bigg|^2
=\bigg|\int_x^{+\infty}\e^{\imag t}\frac{\diff t}{t}\bigg|^2=(\ci(x))^2+(\si(x))^2,
\] 
and we need only show that $\rho(x)>0$, as the condition 
\eqref{eq-Nielsen1.1} is invariant under complex conjugation. 
By the fundamental theorem of Calculus and a standard properties 
of the cosine, the derivative of the function $\rho$ equals 
\begin{multline*}
\rho'(x)=-\frac{2}{x}\re\int_x^{+\infty}\e^{\imag (t-x)}\frac{\diff t}{t}
=-\frac{2}{x}\re\int_0^{+\infty}\e^{\imag t}\frac{\diff t}{t+x}
\\
=-\frac{2}{x}\int_0^{+\infty}\cos t\,\frac{\diff t}{t+x}
=-\frac{2}{x}\sum_{k=0}^{+\infty}\Bigg\{\int_{2k\pi}^{(2k+1)\pi}\cos t 
\frac{\diff t}{t+x}
+\int_{(2k+1)\pi}^{(2k+2)\pi}\cos t \frac{\diff t}{t+x}\Bigg\}
\\
=-\frac{2}{x}\sum_{k=0}^{+\infty}\int_0^{\pi}
\left(\frac{1}{t+2k\pi+x}-\frac{1}{t+(2k+1)\pi+x}\right)\cos t\,\diff t
=-\frac{2}{x}\sum_{k=0}^{+\infty}\int_0^{\pi}\psi_k(t)\cos t\,\diff t,
\end{multline*}
where the function
$$
\psi_k(t):=\frac{\pi}{(t+2k\pi+x)(t+(2k+1)\pi+x)}
$$
is strictly decreasing. Again by standard properties of the cosine and the 
strict monotonicity of $\psi_k$, we find that
\begin{equation*}
\int_0^{\pi}\psi_k(t)\cos t\,\diff t
=\int_0^{\pi/2}\bigl(\psi_k(t)-\psi_k(\pi-t)\bigr)\cos t\,\diff t
>0
\end{equation*}
and hence $\rho'(x)<0$, as the cosine factor is positive. It now follows 
from the mean value theorem of Calculus that the function $\rho(x)$ is 
strictly decreasing, and since clearly $\rho(x)\ge0$, we must have 
$\rho(x)>0$ for all $x\ge0$. In geometric terms, the modulus of the running 
point of the spiral is strictly decreasing and reaches the origin only 
as $x\to+\infty$.
This completes the proof.
\end{proof}

%\end{document}

\section{The Hilbert transform on $L^1$ and the predual of real
$H^\infty$ on the line}
\label{sec-HilbL1}

\subsection{The Hilbert transform on $L^1$}
\label{subsec-HilbtransL1}
For background material on the Hilbert transform and related topics, see,
e.g. the monographs \cite{Dur}, \cite{Gar}, \cite{Steinbook1},
\cite{Steinbook2}, and \cite{SteinWeissbook}.

Let $L^{1,\infty}(\R)$ denote the \emph{weak $L^1$-space}, i.e., the space
of Lebesgue measurable functions $f:\R\to\C$ such that the set
\[
E_f(\lambda):=\{x\in\R:\,\,|f(x)|>\lambda\}, \qquad \lambda\in\bar\R_+,
\]
enjoys the estimate (the absolute value of a measurable subset of $\R$
stands for its Lebesgue measure)
\[
|E_f(\lambda)|\le \frac{C_f}{\lambda},\qquad \lambda\in\R_+;
\]
the optimal constant $C_f$ is written $\|f\|_{L^{1,\infty}(\R)}$; it is the
\emph{$L^{1,\infty}(\R)$-quasinorm} of $f$.
By identifying functions that coincide
almost everywhere, the space $L^{1,\infty}(\R)$ is a
\emph{quasi-Banach space}. It is well-known that the Hilbert transform as
given by \eqref{eq-Hilbert02} maps $\Hop: L^1(\R)\to L^{1,\infty}(\R)$.
Note, however, that functions in $L^{1,\infty}(\R)$ are rather wild and, e.g.,
it is not immediately clear how to associate such a function with a
distribution. However, there is another interpretation of the Hilbert
transform as a mapping from $L^1(\R)$ into a space of distributions on $\R$,
and it is good to know that these interpretations of $\Hop f$ for a given
$f\in L^1(\R)$ are in a one-to-one correspondence. The weak $L^1$-space
associated with an interval $I$ (or a set of positive Lebesgue measure),
written $L^{1,\infty}(I)$, is defined analogously.

If for the moment we use the symbol $\mathbf{F}$ to denote the Fourier
transform, then the Hilbert transform is $\Hop=-\imag\mathbf{F}^{-1}\Mop_{\sign}
\mathbf{F}$, where $\Mop_{\sign}$ stands for multiplication by the sign
function $\sign$. Thus, after taking the Fourier transform, the
distributional interpretation of the Hilbert is that of multiplication by the
unimodular function which takes the value $-\imag$ on the positive half-line,
and the value $\imag$ on the negative half-line.
The distributional interpretation can also be implemented more directly:
\begin{equation}
\langle\varphi,\Hop f\rangle_\R:=-\langle\Hop\varphi,f\rangle_\R,
\label{eq-Hilbtrans-distr1}
\end{equation}
where $\varphi$ is a test function with compact support, and $f\in L^1(\R)$.
Note that $\Hop\varphi$, the Hilbert transform of the test function,
may be defined without the need of the principal value integral:
\[
\Hop\varphi(x)=\frac{1}{2\pi}\int_\R\frac{\varphi(x-t)-\varphi(x+t)}{t}\diff t;
\]
it is a $C^\infty$ function on $\R$ with decay $\Hop \varphi(x)=\Ordo(|x|^{-1})$
as $|x|\to+\infty$. As a consequence, it is clear from
\eqref{eq-Hilbtrans-distr1} how to extend the notion $\Hop f$ to functions
$f$ with $(|x|+1)^{-1}f(x)$ in $L^1(\R)$.

Our next proposition characterizes the space $H^1_\stars(\R)$. For the proof, we
need the notation for the \emph{open unit disk}:
\[
\D:=\{z\in\C:\,\,|z|<1\}.
\]

\begin{prop}
Suppose $f\in L^1(\R)$. Then the following are equivalent:

\noindent{\rm(i)}
$f\in H^1_\stars(\R)$.

\noindent{\rm(ii)} $\Hop f\in L^1(\R)$, where $\Hop f$ is
understood as a distribution on the line $\R$.

\noindent{\rm(iii)} $\Hop f\in L^1(\R)$, where $\Hop f$ is
understood as an almost everywhere defined function in $L^{1,\infty}(\R)$.

\label{prop-Katz}
\end{prop}

\begin{proof}
The implications (i)$\Leftrightarrow$(ii)$\Rightarrow$(iii) are trivial,
so we turn to the remaining implication (iii)$\Rightarrow$(i).
This result, however, is the real line analogue of the result for the circle
in \cite{Katzbook}, p. 87. The transfer to the unit disk is handled by an
appropriate Moebius map from $\D$ to $\C_+$.
\end{proof}

A first application of Proposition \ref{prop-Katz} gives us the following
result.

\begin{cor}
Suppose $f\in L^1(\R)$, and that $\Hop f=0$ pointwise almost everywhere on
$\R$. Then $f=0$ almost everywhere.
\end{cor}

\begin{proof}
Without loss of generality, $f$ is real-valued.
In view of Proposition \ref{prop-Katz}, $f\in H^1_\stars(\R)$, and as a
consequence, the function $F:=f+\imag\Hop f$ is in $H^1_+(\R)$. But on the
real line, $F$ is real-valued, so that the Poisson extension of $F$ to
$\C_+$ is real-valued as well. But this Poisson extension is holomorphic in
$\C_+$, so $F$ must be constant, and the constant is seen to be $0$.
\end{proof}

\begin{rem}
We note that there are the closely related theories of reflectionless measures
(see, e.g., \cite{Pol}) and of real outer functions \cite{GaSa}.
\end{rem}

\subsection{The real $H^\infty$ space}
The \emph{real $H^\infty$ space} is denoted by $H^\infty_\stars(\R)$, and it
consists of all functions $f\in L^\infty(\R)$ of the form
\begin{equation}
f=f_1+f_2,\qquad f_1\in H^\infty_+(\R),\,\,\,f_2\in H^\infty_-(\R).
\label{eq-Hinftydecomp1.01}
\end{equation}
Here, $H^\infty_+(\R)$ consists of all functions in $L^\infty(\R)$ whose
Poisson extension to the upper half-plane is holomorphic, while
$H^\infty_-(\R)$ consists of all functions in $L^\infty(\R)$
whose Poisson extension to the upper half-plane is conjugate-holomorphic
(alternatively, the Poisson extension to the lower half-plane is holomorphic).
The decomposition \eqref{eq-Hinftydecomp1.01} is unique up to additive
constants. Equipped with the natural norm, $H^\infty_\stars(\R)$ is a
Banach space.

The content of next proposition is well-known. For the convenience
of the reader, we supply the  simple  proof.

\begin{prop}
We have the equivalence
\[
f\in H^\infty_\stars(\R)\,\,\,\Longleftrightarrow\,\,\,
f,\tilde\Hop f\in L^\infty(\R).
\]
\label{prop-Hinftychar1.1}
\end{prop}

\begin{proof}
If $f\in H^\infty_\stars(\R)$, then $f=f_1+f_2$, where
$f_1\in H^\infty_+(\R)$ and $f_2\in H^\infty_-(\R)$. Since $\tilde \Hop f=
\imag(f_2-f_1)+c$, where $c$ is the constant that makes
$\tilde \Hop f(\imag)=0$, we see that $\tilde \Hop f\in L^\infty(\R)$.

On the other hand, if $f,\tilde\tilde\Hop f\in L^\infty(\R)$, then
$f+\imag\tilde\Hop f\in H^\infty_+(\R)$ and
$f-\imag\tilde\Hop f\in H^\infty_-(\R)$, so that
\[
2f=(f+\imag\tilde\Hop f)+(f-\imag\tilde\Hop f)\in H^\infty_\stars(\R).
\]
The proof is complete.
\end{proof}

\subsection{The predual of real $H^\infty$}
We shall be concerned with the following space of distributions on the line
$\R$:
\[
\Lspaceo:=L^1(\R)+\Hop L^1_0(\R),
\]
which we supply with the appropriate norm
\begin{equation}
\|u\|_{\Lspaceo}:=\inf\big\{\|f\|_{L^1(\R)}+\|g\|_{L^1(\R)}:\,\,
u=f+\Hop g,\,\, f\in L^1(\R),\,\,g\in L^1_0(\R)\big\},
\label{eq-normLspaceo}
\end{equation}
which makes $\Lspaceo$ a Banach space.

We recall that $L^1_0(\R)$ is  the codimension-one subspace of $L^1(\R)$
which consists of the functions whose integral over $\R$ vanishes.
Given $f\in L^1(\R)$ and $g\in L^1_0(\R)$, the action of $u:=f+\Hop g$ on
a test function $\varphi$ is (compare with \eqref{eq-Hilbtrans-distr1})
\begin{equation}
\langle \varphi,f+\Hop g\rangle_\R=\langle \varphi,f\rangle_\R-
\langle \Hop \varphi,g\rangle_\R=\langle \varphi,f\rangle_\R-
\langle \tilde\Hop \varphi,g\rangle_\R;
\label{eq-Hilbtrans-distr2}
\end{equation}
we observe that the last identity uses that $\langle 1,g\rangle_\R=0$ and the
fact that the functions $\tilde\Hop\varphi$ and $\Hop\varphi$ differ
by a constant.
\medskip

\noindent{\sc Observation.} In view of
Proposition \ref{prop-Hinftychar1.1}, the right hand side of
\eqref{eq-Hilbtrans-distr2} makes sense for $\varphi\in H^\infty_\stars(\R)$. To
be more precise, in accordance with \eqref{eq-Hilbtrans-distr2}, every
$\varphi\in H^\infty_\stars(\R)$ defines a continuous linear functional on
$\Lspaceo$.
\medskip

It remains to identify the dual space of $\Lspaceo$ with $H^\infty_\stars(\R)$.

\begin{prop}
Each continuous linear functional $\Lspaceo\to\C$ corresponds to a function
$\varphi\in H^\infty_\stars(\R)$ in accordance with \eqref{eq-Hilbtrans-distr2}.
In short, the dual space of $\Lspaceo$ equals $H^\infty_\stars(\R)$.
\label{prop-predualHinfty-1.1}
\end{prop}

\begin{proof}
A standard approximation argument involving test functions can be used
to establish that $L^1(\R)$ is a dense subspace of $\Lspaceo$. As the inclusion
map $L^1(\R)\to\Lspaceo$ is continuous, it follows that every continuous linear
functional $\Lspaceo\to\C$ restricts to a continuous linear functional
$L^1(\R)$, which by standard functional analysis corresponds to an element
$\varphi\in L^\infty(\R)$. By density and continuity, $\varphi$ determines the
linear functional completely. As $\varphi\in L^\infty(\R)$, we see that
$\tilde\Hop\varphi\in \mathrm{BMO}(\R)$. By \eqref{eq-Hilbtrans-distr2},
$\tilde\Hop\varphi$ must give a continuous  linear functional
$L^1_0(\R)\to\C$. It is easy to see that this is only possible if $\tilde\Hop
\varphi\in L^\infty(\R)$, which completes the proof, by Proposition
\ref{prop-Hinftychar1.1}.
\end{proof}

The space $\Lspaceo$ is a Banach space, and Proposition
\ref{prop-predualHinfty-1.1} \emph{asserts that its dual space is
$H^\infty_\stars(\R)$} (the real $H^\infty$ space).
For this reason, we will refer to $\Lspaceo$ as the (canonical)
\emph{predual of real} $H^\infty$.

\begin{rem}
Since an $L^1$-function $f$ gives rise to an absolutely continuous measure
$f(t)\diff t$, it is natural to think of $\Lspaceo$ as embedded into the
space $\mathfrak{M}(\R):=M(\R)+\Hop M_0(\R)$, where $M(\R)$ denotes
the space of complex-valued finite Borel measures on $\R$, and $M_0(\R)$ is
the subspace of measures $\mu\in M(\R)$ with $\mu(\R)=0$. The Hilbert
transforms of singular measures noticeably differ from those of
absolutely continuous measures (see \cite{PSZ}).
\end{rem}

\subsection{The ``valeur au point'' function associated with an element
of the predual of real $H^\infty$}
We recall that $\Lspaceo$ consists of distributions on the real line.
However, the definition
\[
\Lspaceo=L^1(\R)+\Hop L^1_0(\R)
\]
would allow us to also think of this space as a subspace of $L^{1,\infty}(\R)$,
the weak $L^1$-space. It is a natural question to wonder about the relationship
between the distribution and the $L^{1,\infty}$ function.
We will stick to the distribution theory definition of $\Lspaceo$, and
associate with a given $u\in\Lspaceo$ the ``valeur au point'' function
$\pev[u]$ at almost all points of the line. The precise definition of
$\pev[u]$ is as follows.

\begin{defn} For a fixed $x\in\R$, let  $\chi=\chi_x$ is a compactly supported
$C^\infty$-smooth function on $\R$ with $\chi(t)=1$ for all $t$ in an open
neighborhood of the point $x$. Also, let
\[
P_{x+\imag\epsilon}(t):=\pi^{-1}\frac{\epsilon}{\epsilon^2+(x-t)^2}
\]
be the Poisson kernel. The \emph{valeur au point function} associated with
the distribution $u$ on $\R$ is the function $\pev[u]=\pev[u\chi]$ given by
\begin{equation}
\pev[u](x):=
\lim_{\epsilon\to0^+}\langle \chi P_{x+\imag\epsilon}, u\rangle_\R,\qquad
x\in\R,
\label{eq-pv1001}
\end{equation}
wherever the limit exists.
\end{defn}

In principle, $\pev[u](x)$ might depend on the choice of the cut-off function
$\chi$. The following lemma guarantees that this is not the case in the
relevant situation.

\begin{lem}
For $u=f+\Hop g\in\Lspaceo$, where $f\in L^1(\R)$ and $g\in L^1_0(\R)$,
the valeur au point function $\pev[u](x)$ does not depend on the choice of
the cut-off $\chi$.
Moreover, we have that
\[
\pev[u](x)=f(x)+\Hop g(x),\qquad \text{a.e.}\,\,\,x\in\R,
\]
where on the right hand side, the function $\Hop g(x)$ is defined pointwise as
a principal value.
\label{lem-indep-of-chi01}
\end{lem}

\begin{proof}
For $f\in L^1(\R)$, it is a standard exercise involving Poisson integrals to
show that $\pev[f](x)=f(x)$ holds for almost all $x\in\R$
(for details, see, e.g., \cite{Gar}, Chapter 1), and that the choice of $\chi$
does not matter for the value of $\pev[f](x)$ for a given point $x\in\R$.

We turn to the evaluation of $\pev[\Hop g](x)$. By translation invariance, we
may as well consider only $x=0$. As a matter of definition, we have that
\begin{multline}
\pev[\Hop g](0)
=\lim_{\epsilon\to0^+}\langle \chi P_{\imag\epsilon}, \Hop g\rangle_\R=
-\lim_{\epsilon\to0^+}\langle \Hop[\chi P_{\imag\epsilon}], g\rangle_\R
\\
=\lim_{\epsilon\to0^+}\Big\{\langle \Hop[\tilde\chi P_{\imag\epsilon}],
g\rangle_\R
-\langle \Hop[P_{\imag\epsilon}], g\rangle_\R\Big\},
\label{eq-vp:pv}
\end{multline}
where $\tilde\chi:=1-\chi$ and $\chi$ is a smooth cut-off function with
$\chi(t)=1$ near $t=0$. Here, as above, $P_{\imag\epsilon}$ is the function
\[
P_{\imag\epsilon}(t)=\pi^{-1}\frac{\epsilon}{\epsilon^2+t^2},
\]
and its Hilbert transform is given by
\[
\Hop[P_{\imag\epsilon}](t)=\pi^{-1}\frac{t}{\epsilon^2+t^2}.
\]
A calculation reveals that
\[
\pi^{-1}\frac{t}{\epsilon^2+t^2}=\int_{0}^{+\infty}
\frac{1_{\R\setminus[-\tau,\tau]}}{\pi t}\,
\frac{2\epsilon^2\tau}{(\epsilon^2+\tau^2)^2}\diff\tau,
\]
which can used to show that
\[
-\lim_{\epsilon\to0^+}\langle \Hop[P_{\imag\epsilon}](t),g\rangle_\R=-\lim_{\tau\to0^+}
\int_{\R\setminus[-\tau,\tau]}\frac{g(t)}{\pi t}\diff t=\Hop g(0),
\]
where the rightmost equality sign is a matter of the pointwise definition of
the Hilbert transform. The desired conclusion now follows from \eqref{eq-vp:pv},
once we have established that  for fixed $\tilde\chi$, we have
\[
\|\Hop[\tilde\chi P_{\imag\epsilon}]\|_{L^\infty(\R)}=\Ordo(\epsilon)
\]
as $\epsilon\to0^+$. This is rather elementary and left to the interested
reader; here, we only observe that the function $\tilde\chi$ is smooth and
bounded, which equals $1$ near infinity and vanishes near the origin, so that
$\tilde\chi P_{\imag\epsilon}$ becomes a very small and quite smooth function.
\end{proof}

Additional properties of the mapping $\pev$ are outlined below.

\begin{prop}
{\rm (Kolmogorov)}
The mapping $\pev:\Lspaceo\to L^{1,\infty}(\R)$, $u\mapsto\pev[u]$, is
continuous.
\label{prop-weakL1cont}
\end{prop}

\begin{proof}
This follows from the standard weak-type estimate for the Hilbert transform
(see, e.g., \cite{Gar}).
\end{proof}

The next result allows us to identify $u$ with $\pev[u]$.

\begin{prop}
{\rm (Kolmogorov)}
If $u\in\Lspaceo$ and $\pev[u]=0$ almost everywhere on $\R$, then $u=0$ as
a distribution.
\label{prop-distrptwise1}
\end{prop}

\begin{proof}
We write $u=f+\Hop g$, where $f\in L^1(\R)$ and $g\in L^1_0(\R)$. Since
$g\in L^1_0(\R)$ and, by assumption,  $\pev[g]=-f\in L^1(\R)$,
it follows from Proposition \ref{prop-Katz} that $g\in H^1_\stars(\R)$ and
consequently that $\Hop g\in L^1(\R)$ as a distribution.
Since the Hilbert transform $\Hop$ leaves the space $H^1_\stars(\R)$ invariant,
 we also obtain that $f\in H^1_\stars(\R)$, and then it is immediate from the
assumption that $u=0$ as a distribution.
\end{proof}

The local version of Proposition \ref{prop-distrptwise1} runs as follows.

\begin{prop}
If $u\in\Lspaceo$ and $\pev[u]=0$ almost everywhere on an open interval
$I\subset\R$, then the distribution $u$ is supported on $\R\setminus I$.
\label{prop-distrptwise1:local1}
\end{prop}

\begin{proof}
We split $u=f+\Hop g$, where $f\in L^1(\R)$ and $g\in L^1_0(\R)$. Without
loss of generality, we may assume that $f$ and $g$ are real-valued. Again,
without loss of generality, the open interval $I$ is assumed to be
\emph{bounded}. By the classical theorem of Kolmogorov \cite{Dur}, the
function $G:=g+\imag\Hop g$ is in the $H^p$-space in the upper half plane
(with respect to Poissonian measure
$\pi^{-1} (1+t^2)^{-1}\diff t$ on the real line), for each $p$ with $0<p<1$.
In Kolomogorov's theorem, $\Hop g$ initially has the pointwise interpretation,
but in a second step, it is valid with the distributional interpretation
as well. By assumption, $\pev[\Hop g]=-f$ holds on the bounded open interval
$I$, so that the boundary function for $G$ is in $L^1$ on $I$. Essentially,
this means that $G$ is in $H^1$ near $I$ in the upper half-plane.
This can be made precise in the following manner. We choose a slightly smaller
interval $J\subset I$, whose both endpoints differ from those of $I$.
Next, we choose a bounded simply connected Jordan domain $\Omega$ in the upper
half-plane $\C_+$ whose boundary curve $\partial\Omega$ is $C^\infty$-smooth,
with the property that $\partial\Omega\cap\R=J$. Then it is not difficult to
see that $G$, restricted to $\Omega$, belongs to the $H^1$-space on $\Omega$,
which is most conveniently defined in terms of a fixed conformal
mapping from the unit disk $\D$ onto $\Omega$. The remaining part of the
proof is an exercise in Schwarzian reflection across the interval $J$.
\end{proof}

\subsection{Dual action on intervals}
If $I\subset\R$ is an open interval, and $f,g:I\to\C$ are two Borel measurable
functions with $fg\in L^1(I)$, then we may define \emph{the dual action on}
$I$:
\[
\langle f, g\rangle_I:=\int_{I}f(t)g(t)\diff t;
\]
this is a special case of dual action on a more general measurable set (see
Subsection \ref{subsec-dualaction}).
For instance, if $f$ is a test function with compact support in $I$, and
$g$ is locally integrable on $I$, then the dual action is well-defined. More
generally, we will write $\langle\cdot,\cdot\rangle_I$ to denote the dual
action of distributions on test functions on the given interval $I$.
Naturally, this agrees with the notation we have introduced so far for the
case $I=\R$.

\subsection{The restriction of $\Lspaceo$ to an interval}
\label{subsec-4.6.01}

If $u$ is a distribution on an open interval $J$, then the \emph{restriction
of $u$ to an open subinterval $I$, denoted $u|_I$}, is the distribution defined
by
\[
\langle\varphi,u|_I\rangle_I:=\langle\varphi,u\rangle_J,
\]
where $\varphi$ is a $C^\infty$-smooth test function whose support is compact
and contained in $I$.

\begin{defn}
Let $I$ be an open interval of the real line. Then $u\in\LspaceI$ means by
definition that $u$ is a distribution on $I$ such that there exists a
distribution $v\in\Lspaceo$ such that $u=v|_I$.
\label{defn-4.6.1}
\end{defn}

\begin{rem}
The following observation is pretty trivial, but quite useful.
If we are given two open intervals $I$ and $J$ of the line $\R$,
with $I\subset J$,  then the restriction operation $v\mapsto v|_I$ makes
sense $\LspaceJ\to\LspaceI$.
\end{rem}

Proposition \ref{prop-distrptwise1:local1} has a localized version on
a given interval $J$.

\begin{cor}
Suppose $I,J\subset\R$ are open intervals with $I\subset J$. If
$u\in\LspaceJ$ and $\pev[u]=0$ almost everywhere on $I$, then the support
of the distribution $u$ has empty intersection with $I$.
\label{cor-distrptwise1:local2}
\end{cor}

\begin{proof}
The assertion of the corollary is immediate from Proposition
\ref{prop-distrptwise1:local1}.
\end{proof}

The following result will prove quite useful.

\begin{prop}
Let $I$ be a nonempty bounded open interval of the line $\R$.
Then $L^1(I)$ is a norm dense subspace of $\LspaceI$.
\end{prop}

\begin{proof}
As a matter of definition, we have that
\[
\LspaceI=\Lspaceo/\ZspaceI,
\]
where
\[
\ZspaceI:=\{u\in\Lspaceo:\,\,I\cap\supp u=\emptyset\}.
\]
By elementary Functional Analysis, we know that the dual space $\LspaceI^*$
is given by the annihilator
\[
\LspaceI^*=\ZspaceI^\perp=\{f\in H^\infty_\stars(\R):\,\forall u\in \ZspaceI:
\,\langle f,u\rangle_\R=0\}.
\]

\smallskip

{\sc Observation.} We have that
$\ZspaceI^\perp\subset\{f\in H^\infty_\stars(\R):\,f=0\,\,\mathrm{a.e.}\,\,
\mathrm{on}\,\,\R\setminus I\}$.

\smallskip

{\sc Proof of the observation:} Indeed, if $f\in H^\infty_\stars(\R)$ and
the restriction to $\R\setminus I$ is nonzero on a set of positive Lebesgue
measure, we readily construct a function $u\in L^1(\R)$ which vanishes on
$I$
%[so that $u\in\ZspaceI$]
such that $\langle f,u\rangle_\R\ne0$. Since $u\in\ZspaceI$, this proves
the asserted containment.

We proceed with the proof of the proposition. If $f\in H^\infty_\stars(\R)$
vanishes a.e. on $\R\setminus I$, and as a functional on $\LspaceI$,
$f$ annihilates $L^1(I)$, then we may conclude that $f=0$ a.e. on $I$ as well.
But now $f=0$ a.e. on the line $\R$, so $f=0$ as an element of
$H^\infty_\stars(\R)$. By the Hahn-Banach theorem, we derive that
$L^1(I)$ is norm dense in $\LspaceI$.
\end{proof}

\begin{rem}
A more refined argument shows that in the context of the observation, we
actually have equality:
$\ZspaceI^\perp=\{f\in H^\infty_\stars(\R):\,f=0\,\,\mathrm{a.e.}\,\,
\mathrm{on}\,\,\R\setminus I\}$.
\end{rem}

We may also translate Proposition \ref{prop-weakL1cont} to this local context.

\begin{cor}
Let $I$ be a nonempty open interval of the line $\R$. Then the ``valeur au
point'' mapping is continuous $\pev:\LspaceI\to L^{1,\infty}(I)$.
\end{cor}

\section{Background material: the Hardy and BMO spaces on
the circle}
%\label{sec-background.dynamics}

\subsection{The Hardy $H^1$ space on the circle: analytic and real}
Let $L^1(\R/2\Z)$ denote the space of $2$-periodic  Borel  measurable functions
$f:\R\to\C$ subject to the integrability condition
\[
\|f\|_{L^1(\R/2\Z)}:=\int_{I_1}|f(t)|\diff t<+\infty,
\]
where $I_1=]\!-\!1,1[$ as before. As usual, we identify functions that
agree except on a null set. Via the
exponential mapping $t\mapsto\e^{\imag\pi t}$, which is $2$-periodic and
maps the real line $\R$ onto the unit circle $\Te$, we may identify the
space $L^1(\R/2\Z)$ with the standard Lebesgue space $L^1(\Te)$ of the
unit circle. This will allow us to develop the elements of Hardy space theory
in the setting of $2$-periodic functions. We shall need the subspace
$L^1_0(\R/2\Z)$ consisting of all $f\in L^1(\R/2\Z)$ with
\[
\langle f,1\rangle_{I_1}=\int_{I_1}f(t)\diff t=0;
\]
it has codimension $1$ in $L^1(\R/2\Z)$. The Hardy space $H^1_+(\R/2\Z)$ is
defined as the subspace of $L^1(\R/2\Z)$ consisting of functions
$g\in L^1(\R/2\Z)$ with
\begin{equation}
\int_{-1}^1 \e^{\imag \pi n t}g(t)\diff t=0,\qquad n=0,1,2,\ldots.
\label{eq-FourChar-H1-circle}
\end{equation}
The space $H^1_+(\R/2\Z)$ is the periodic analogue of the Hardy space
$H^1_+(\R)$, and it can be understood in terms of the Hardy $H^1$-space of
the disk. If $H^1_+(\Te)$ denotes the standard Hardy space on the unit disk
(restricted to the boundary unit circle), \emph{then $g\in H^1_+(\R/2\Z)$
means that $g(x)=f(\e^{\imag\pi x})$ for some $f\in H^1_+(\Te)$ with $f(0)=0$}.
In particular, the functions in $H^1_+(\R/2\Z)$ have holomorphic extensions
to the upper half-plane which are $2$-periodic. As a matter of definition,
\emph{$H^1_-(\R/2\Z)$ consists of the functions $g$ in $L^1(\R/2Z)$ whose
complex conjugate $\bar g$ is in $H^1_+(\R/2\Z)$}.
Finally, we put
\[
H^1_{\stars}(\R/2\Z):=H^1_+(\R/2\Z)\oplus H^1_-(\R/2\Z),
\]
where we think of the elements of the sum space as $2$-periodic functions
(as before the symbol $\oplus$ means direct sum, since
$H^1_+(\R/2\Z)\cap H^1_-(\R/2\Z)=\{0\}$).
We note that, for instance,  $H^1_\stars(\R/2\Z)\subset L^1_0(\R/2\Z)$. We
will think of $H^1_{\stars}(\R/2\Z)$ as the \emph{real $H^1$ space of
$2$-periodic functions}.

\subsection{The Hilbert transform on $2$-periodic functions and
distributions}
For $f\in L^1(\R/2\Z)$, we let $\Hop_2$ be the convolution operator
\begin{equation}
\Hop_2 f(x):=
\frac{1}{2}\pv
\int_{I_1} f(t)\cot\frac{\pi(x-t)}{2}\diff t,
\label{eq-Hilbert04}
\end{equation}
where again $\pv$  stands for principal value, which means we take the
limit as $\epsilon\to0^+$ of the integral where the set
\[
\{x\}+2\Z+[-\epsilon,\epsilon]
\]
is removed from the interval $I_1=]\!-\!1,1[$. It is obvious from the
periodicity of the cotangent function that $\Hop_2 f$,   if it exists
as a limit, is $2$-periodic. By a standard trigonometric identity,
\[
\frac{1}{2}\cot\frac{\pi y}{2}=\lim_{N\to+\infty}\frac{1}{\pi}\sum_{n=-N}^{N}
\frac{1}{y+2n},
\]
where the convergence is uniform on compact subsets of the line. By a
change of variables,
\begin{equation}
\Hop_2 f(x)=
\frac{1}{2}\lim_{\epsilon\to0^+}
\int_{I_1\setminus I_\epsilon} f(x-t)\cot\frac{\pi t}{2}\diff t,
\label{eq-Hilbert04.1}
\end{equation}
(here, as usual, $I_\epsilon=]\!-\!\epsilon,\epsilon[$) from which we
conclude, by uniform convergence and periodicity, that
\begin{multline}
\Hop_2 f(x)=
\frac{1}{\pi}\lim_{N\to+\infty}\lim_{\epsilon\to0^+}
\sum_{n=-N}^{N}\int_{I_1\setminus I_\epsilon} f(x-t)\frac{\diff t}{t+2n}
\\
=\frac{1}{\pi}\lim_{\epsilon\to0^+}
\int_{I_1\setminus I_\epsilon} f(x-t)\frac{\diff t}{t}+
\frac{1}{\pi}\lim_{N\to+\infty}\sum_{n:|n|\le N, n\ne0}
\int_{I_1} f(x-t)\frac{\diff t}{t+2n}
\\
=\frac{1}{\pi}\lim_{\epsilon\to0^+}
\int_{I_1\setminus I_\epsilon} f(x-t)\frac{\diff t}{t}+
\frac{1}{\pi}\lim_{N\to+\infty}\sum_{n:|n|\le N, n\ne0}
\int_{[2n-1,2n+1]} f(x-t)\frac{\diff t}{t}
\\
=\lim_{N\to+\infty}\lim_{\epsilon\to0^+}\frac{1}{\pi}
\int_{I_{2N+1}\setminus I_\epsilon} f(x-t)\frac{\diff t}{t}.
\label{eq-Hilbert04.2}
\end{multline}
In other words, the operator $\Hop_2$ is just the natural extension of the
Hilbert transform to the $2$-periodic functions. We observe that $\Hop_21=0$,
which contrasts with the non-periodic case (where no nontrivial function is
mapped to the zero function).
It is well-known that the  periodic Hilbert transform $\Hop_2$ maps
$L^1(\R/2\Z)$ into the weak $L^1$-space $L^{1,\infty}(\R/2\Z)$. However, we prefer
to work within the framework of distribution theory, so we proceed as follows.

Let $C^\infty(\R/2\Z)$ denote the space of $C^\infty$-smooth $2$-periodic
functions on $\R$. It is easy to see that
\[
\varphi\in C^\infty(\R/2\Z) \implies \Hop_2\varphi\in C^\infty(\R/2\Z).
\]
To emphasize the importance of the circle $\Te\cong\R/2\Z$, we write
\begin{equation}
\langle f,g\rangle_{\R/2\Z}:=\int_{-1}^1 f(t)g(t)\diff t,
\label{eq-dualaction-per1}
\end{equation}
for the dual action when $f$ and $g$ are $2$-periodic.

\begin{defn}
For a test function $\varphi\in C^\infty(\R/2\Z)$ and a distribution $u$ on
the circle $\R/2\Z$, we put
\[
\langle \varphi,\Hop_2 u\rangle_{\R/2\Z}:=-\langle\Hop_2\varphi,u\rangle_{\R/2\Z}.
\]
\end{defn}

This defines the Hilbert transform $\Hop_2 u$ for any distribution $u$ on the
circle $\R/2\Z$.

The analogue of Proposition \ref{prop-Katz} for the circle reads as follows.
Note that the formula definition of the ``valeur au point'' function makes
sense also for $u$ in the space of distributions
$L^1(\R/2\Z)+\Hop_2 L^1(\R/2\Z)$. Moreover, the independence of the cut-off
function is quite analogous to the real line case (Lemma
\ref{lem-indep-of-chi01}) and left to the interested reader.

\begin{prop}
Suppose $f\in L^1_0(\R/2\Z)$. Then the following are equivalent:

\noindent{\rm(i)}
$f\in H^1_\stars(\R/2\Z)$.

\noindent{\rm(ii)} $\Hop_2 f\in L^1(\R/2\Z)$, where $\Hop_2 f$ is
understood as a distribution on the line $\R$.

\noindent{\rm(iii)} $\pev[\Hop_2 f]\in L^1(\R/2\Z)$.
\label{prop-Katz2}
\end{prop}

\begin{proof}
This is immediate from \cite{Katzbook}, p. 87.
\end{proof}

\subsection{The real $H^\infty$-space of the circle}
The \emph{real $H^\infty$-space on the circle $\R/2\Z$} is denoted by
$H^\infty_\stars(\R/2\Z)$, and consists of all the functions in $H^\infty_\stars
(\R)$ that are $2$-periodic. The analogue of Proposition
\ref{prop-Hinftychar1.1} reads:

\begin{prop}
 We have the equivalence
\[
f\in H^\infty_\stars(\R/2\Z)\,\,\,\Longleftrightarrow\,\,\,
f,\Hop_2 f\in L^\infty(\R/2\Z).
\]
\label{prop-Hinftychar1.2}
\end{prop}

This result is well-known.

\subsection{A predual of $2$-periodic real $H^\infty$}
We put
\[
\Lspaceoper:=L^1(\R/2\Z)+\Hop_2 L^1_0(\R/2\Z),
\]
understood as a space of $2$-periodic distributions on the line $\R$.
More precisely, if $u=f+\Hop_2 g$, where $f\in L^1(\R/2\Z)$ and
$g\in L^1_0(\R/2\Z)$, then the action on a test function
$\varphi\in C^\infty(\R/2\Z)$ is given by
\begin{equation}
\langle\varphi,u\rangle_{\R/2\Z}:=\langle\varphi,f\rangle_{\R/2\Z}
-\langle \Hop_2\varphi, g\rangle_{\R/2\Z}.
\label{eq-Hilbtransf-distrib-per1.1}
\end{equation}
But a $2$-periodic distribution should be possible to think of as a
distribution on the line, which means that need to understand the action
on standard test functions in $C^\infty_c(\R)$. If $\psi\in C^\infty_c(\R)$,
we simply put
\begin{equation}
\langle\psi,u\rangle_{\R/2\Z}:=\langle\Perop_2\psi,u\rangle_{\R/2\Z},
\label{eq-extper1.01}
\end{equation}
where $\Perop_2\psi\in C^\infty(\R/2\Z)$ is given by
\begin{equation}
\Perop_2\psi(x):=\sum_{j\in\Z}\psi(x+2j).
\label{eq-Peropdef1.1}
\end{equation}
We will refer to $\Perop_2$ as the \emph{periodization operator}.

As in the case of the line, we may identify  $\Lspaceoper$ with the predual of
the real $H^\infty$ space.

\begin{prop}
Each continuous linear functional $\Lspaceoper\to\C$ corresponds to a function
$\varphi\in H^\infty_\stars(\R/2\Z)$ in accordance with
\eqref{eq-Hilbtransf-distrib-per1.1}.
In short, the dual space of $\Lspaceoper$ is isomorphic to
$H^\infty_\stars(\R/2\Z)$.
\label{prop-predualHinfty-1.2}
\end{prop}

We omit the proof, which is analogous to that of Proposition
\ref{prop-predualHinfty-1.1}.

The definition of the ``valeur au point function'' $\pev[u]$ makes sense
for $u\in\Lspaceoper$ and as in the case of the line, it does not depend on
the choice of the particular cut-off function. We have the analogue of
Proposition \ref{prop-weakL1cont}; as the result is standard, we omit the
proof.

\begin{prop}
{\rm (Kolmogorov)}
The ``valeur au point'' mapping $\pev:\Lspaceoper\to L^{1,\infty}(\R/2\Z)$,
$u\mapsto\pev[u]$, is continuous.
\label{prop-weakL1cont.per}
\end{prop}

\section{A sum of two preduals and its localization to intervals}

\subsection{The sum space $\Lspaceo\oplus\Lspaceoper$}
Suppose $u$ is distribution on the line $\R$ of the form
\begin{equation}
u=v+w,\quad\text{where}\quad v\in\Lspaceo,\,\,\, w\in\Lspaceoper.
\label{eq-sumspace001}
\end{equation}
The natural question appears as to whether the distributions $v,w$ on the
right hand side are unique. This is indeed so.

\begin{prop}
We have that $\Lspaceo\cap\Lspaceoper=\{0\}$.
\label{prop-intersec1.001}
\end{prop}

This last statement is pretty obvious in terms of the Fourier transform,
which sends $2$-periodic distributions to sums of point masses along the
integers, while the space $\Lspaceo$ is mapped to a space of bounded
continuous functions.

In view of Proposition \ref{prop-intersec1.001}, it makes sense to write
$\Lspaceo\oplus\Lspaceoper$ for the space of tempered distributions $u$
of the form \eqref{eq-sumspace001}. We endow $\Lspaceo\oplus\Lspaceoper$ with
the induced Banach space norm
\[
\|u\|_{\Lspaceo\oplus\Lspaceoper}:=\|v\|_{\Lspaceo}+\|w\|_{\Lspaceoper},
\]
provided $u,v,w$ are related via \eqref{eq-sumspace001}.

\subsection{The localization of $\Lspaceo\oplus\Lspaceoper$ to a
bounded open interval}
\label{subsec-6.2}

In the sense of Subsection \ref{subsec-4.6.01}, we may restrict a given
distribution $u\in \Lspaceo\oplus\Lspaceoper$ to a given open interval $I$.
It is natural to wonder what the space of such restrictions looks like.

\begin{prop}
% $(0<\gamma<+\infty)$
The restriction of the space $\Lspaceo\oplus\Lspaceoper$ to a bounded open
interval $I$ equals the space $\LspaceI$.
\label{prop-rest1.1}
\end{prop}

\begin{proof}
By definition, the restriction of $\Lspaceo$ to $I$ equals $\LspaceI$.
It remains to show that the restriction to $I$ of a distribution in
$\Lspaceoper$ is in $\LspaceI$ as well. Since
\[
\Lspaceoper=L^1(\R/2\Z)+\Hop_2 L^1_0(\R/2\Z),
\]
and the restriction to the bounded interval $I$ of $L^1(\R/2\Z)$ is contained
in $L^1(I)$, the only thing we need to check is that the
restriction of $\Hop_2 L^1_0(\R/2\Z)$ to $I$ is contained in $\LspaceI$.
It will be enough to show that for each $f\in L^1_0(\R/2\Z)$, there exist
$g\in L^1(\R)$, $h\in L^1_0(\R)$, and a distribution $W\in{\mathcal D}'(\R)$
with support contained in $\R\setminus I$, such that
\[
\Hop_2 f=g+\Hop h+W.
\]
We need two bounded open intervals $J_1,J_2$ such that
$I\Subset J_1\Subset J_2$.
We first let $h$ equal $f$ on
$J_1$, and put it equal to $0$ on
$\R\setminus J_2$. In the difference set
$J_2\setminus J_1$, we let $h$ be constant, where the value of the constant
is then determined by the requirement that $h\in L^1_0(\R)$.
As the cotangent kernel $\frac{1}{2}\cot \frac{\pi t}{2}$ used to define
$\Hop_2$ and the Hilbert transform kernel $\frac{1}{\pi t}$ have the same
singularity, it is easy to see that $\Hop_2 f-\Hop h$
is smooth on $J_1$, and we may declare $g$ to
equal
$\Hop_2 f-\Hop h$
on
$I$, and put it equal to $0$ on the rest
$\R\setminus I$.
The distribution $W$ is
uniquely determined by these choices, and has the required properties.
\end{proof}

%\end{document}

\section{An involution, its adjoint, and the
%weighted composition and
periodization operator}

\subsection{An involutive operator}
\label{subsec-7.1}

For each positive real number $\beta$, let $\Jop_\beta$ denote  the
involution given by
\begin{equation*}
\Jop_\beta f(x):=\frac{\beta}{x^2}\,f(-\beta/x),\qquad x\in\R^\times.
%\label{eq-Jopstar}
\end{equation*}
 With respect to the dual action $\langle\cdot,\cdot\rangle_\R$,
 this operator $\Jop_\beta$ can be understood as the preadjoint of the
 involution $\Jop_\beta^*$ defined in
 \eqref{eq-Jop1.1}.

As usual, we use the standard notation $\R^\times:=\R\setminus\{0\}$.
We now record some basic properties of this involution. For instance, by the
change-of-variables formula, $\Jop_\beta:L^1(\R)\to L^1(\R)$ is an
\emph{isometry}.

\begin{prop} Fix $0<\beta<+\infty$. The operator  $\Jop_\beta$ is an
isometric isomorphism $L^1(\R)\to L^1(\R)$. In addition,  $\Jop_\beta$ maps
$H^1_+(\R)\to H^1_+(\R)$ and $H^1_-(\R)\to H^1_-(\R)$ and consequently
$\Jop_\beta:H^1_\stars(\R)\to H^1_\stars(\R)$ as well.
\label{prop-1.001}
\end{prop}

\begin{proof}
The mapping $z\mapsto-\beta/z$ preserves the upper half-plane $\C_+$, and
so that functions holomorphic in $\C_+$ are sent to functions holomorphic
in $\C_+$ under composition by $z\mapsto-\beta/z$. The
isometric part is already settled, so it remains to check that
the space $H^1_+(\R)$ is preserved under $\Jop_\beta$, since the case of 
$H^1_-(\R)$ is identical. This follows easily by checking the property 
on a dense subspace (e.g. consisting of rational functions).
\end{proof}

If $f\in L^1(\R)$ and $\varphi\in L^\infty(\R)$, the change-of-variables formula
 yields
\begin{equation}
\langle \varphi,\Jop_\beta f\rangle_\R=\int_\R \varphi(t)\,f(-\beta/t)\,
\frac{\beta\diff t}{t^2}=\int_\R \varphi(-\beta/t)\,f(t)\,
\diff t=\langle \Jop_\beta^* \varphi,f\rangle_\R,
\label{eq-Jopdef1}
\end{equation}
where $\Jop_\beta^*$ is the involution
\[
\Jop_\beta^* \varphi(t):=\varphi(-\beta/t),\qquad t\in\R^\times.
\]
We need to extend $\Jop_\beta$ to an operator $\Lspaceo\to\Lspaceo$. To this
end, we need to understand how to define $\Jop_\beta\Hop f$ as a distribution
in $\Lspaceo$ when $f\in L^1_0(\R)$. First, following \eqref{eq-Jopdef1},
we put
\begin{equation}
\langle \varphi,\Jop_\beta\Hop f\rangle_\R:=
-\langle\Hop\Jop_\beta^* \varphi,f\rangle_\R,
\label{eq-Jopdef2}
\end{equation}
for $f\in L^1_0(\R)$ and for test functions $\varphi\in C^\infty_c(\R^\times)$,
since such test functions vanish near the origin. Note here that if $\varphi
\in C^\infty_c(\R^\times)$, then necessarily $\Jop_\beta^*\varphi
\in C^\infty_c(\R^\times)$ as well, so the right-hand side of \eqref{eq-Jopdef2}
is well-defined.

\begin{prop}
For a test function $\varphi\in C^\infty_c(\R^\times)$, we have the identity
\[
\Hop\Jop_\beta^*\varphi(x)=\Jop_\beta^*\Hop\varphi(x)-
\langle\varphi,t\mapsto\tfrac{1}{\pi t}\rangle_\R,\qquad
x\in\R^\times.
\]
\label{prop-7.1.2}
\end{prop}

\begin{proof}
By a change of variables in the corresponding integral, we have that
\[
\Jop_\beta^*\Hop\varphi(x)=-\frac{1}{\pi}\pv\int_\R\frac{x}{\beta+tx}
\,\varphi(t)\diff t,\qquad \Hop\Jop_\beta^*\varphi(x)=\frac{1}{\pi}\pv
\int_\R\frac{\beta}{t(\beta+tx)}
\,\varphi(t)\diff t,
\]
so the asserted equality is a simple consequence of the algebraic identity
\[
-\frac{x}{\beta+tx}=\frac{\beta}{t(\beta+tx)}-\frac{1}{t}.
\]
The proof is complete.
\end{proof}

As $f\in L^1_0(\R)$, its action on constants vanishes, so by a combination
of \eqref{eq-Jopdef1}, \eqref{eq-Jopdef2}, and Proposition \ref{prop-7.1.2},
we obtain
\begin{equation}
\langle \varphi,\Jop_\beta\Hop f\rangle_\R=
-\langle\Hop\Jop_\beta^*\varphi,f\rangle_\R
=-\langle\Jop_\beta^*\Hop\varphi,f\rangle_\R=
-\langle\Hop\varphi,\Jop_\beta f\rangle_\R=\langle\varphi,\Hop\Jop_\beta f
\rangle_\R,
\label{eq-Jopdef3}
\end{equation}
for $\varphi\in C^\infty_c(\R^\times)$.
As $f\in L^1_0(\R)$, we also have that $\Jop_\beta f\in L^1_0(\R)$, so that
$\Hop\Jop_\beta f\in \Hop L^1_0(\R)\subset \Lspaceo$. This means that as
distributions on the punctured line $\R^\times=\R\setminus\{0\}$,
$\Jop_\beta\Hop f$ and $\Hop\Jop_\beta f$ coincide. In particular, their ``valeur
au point'' functions, which are well-defined almost everywhere, coincide on
$\R^\times$. However, the distribution $\Hop\Jop_\beta f$ makes
sense on test functions $\varphi\in C^\infty_c(\R)$, and actually, more
generally for $\varphi\in H^\infty_\stars(\R)$. This allows us to extend the
action of $\Jop_\beta\Hop f$ from $C^\infty_c(\R^\times)$ to $H^\infty_\stars(\R)$
(compare with \eqref{eq-Hilbtrans-distr2}).

\begin{defn}
For $u\in\Lspaceo$ of the form $u=f+\Hop g\in\Lspaceo$, where
$f\in L^1(\R)$ and $g\in L^1_0(\R)$, we define the $\Jop_\beta u$ to be
the distribution on $\R$ given by the formula
\begin{equation*}
\langle \varphi,\Jop_\beta u\rangle_\R=\langle \varphi,\Jop_\beta
(f+\Hop g)\rangle_\R:=
\langle\varphi,\Jop_\beta f\rangle_\R+\langle\varphi,\Hop\Jop_\beta g\rangle_\R
=\langle\varphi,\Jop_\beta f\rangle_\R-\langle\tilde\Hop\varphi,
\Jop_\beta g\rangle_\R,
\end{equation*}
for test functions $\varphi\in H^\infty_\stars(\R)$.
\label{defn-7.1.3}
\end{defn}

As already noted, this is in complete agreement with the way we would
previously understand $\Jop_\beta u$ as a distribution on $\R^\times$, using
smooth test functions having compact support on the punctured line
$\R^\times$; see \eqref{eq-Jopdef1} and \eqref{eq-Jopdef2}.

\begin{prop}
Fix $0<\beta<+\infty$.
The involution $\Jop_\beta$ acts continuously
$\Lspaceo\to\Lspaceo$, and the involution $\Jop_\beta^*$ acts
continuously $H^\infty_\stars(\R)\to H^\infty_\stars(\R)$. Moreover, on their
respective spaces, $\Jop_\beta^2$ and $ {\Jop_\beta^*}^2$ both equal
the identity operator.
\label{prop-7.1.3}
\end{prop}

\begin{proof}
Let $u\in \Lspaceo$ be of the form $u=f+\Hop g$, where $f\in L^1(\R)$
and $g\in L^1_0(\R)$. Then, by definition, $\Jop_\beta u=\Jop_\beta f
+\Hop\Jop_\beta g\in \Lspaceo$, and it is clear that the mapping
$\Jop_\beta$ acts continuously. Moreover, by iteration
\[
\Jop_\beta^2 u=\Jop_\beta^2 f+\Hop\Jop_\beta^2 g=f+\Hop g=u
\]
since $\Jop_\beta^2F=F$ holds for all $F\in L^1(\R)$. The assertions
concerning the adjoint $\Jop_\beta^*$ follow by duality.
\end{proof}

\subsection{The periodization operator}
\label{subsec-7.2}

We recall the definition of the \emph{periodization operator} $\Perop_2$:
\[
\Perop_2f(x):=\sum_{j\in\Z}f(x+2j).
\]
In \eqref{eq-Peropdef1.1}, we defined the $\Perop_2$ on test functions.
It is however clear that it remains well-defined with much less smoothness
required of $f$. The terminology comes from the property that whenever it is
well-defined, the function $\Perop_2 f$ is $2$-periodic automatically.
A first result is the following.

\begin{prop}
The operator $\Perop_2$ acts contractively $L^1(\R)\to L^1(\R/2\Z)$.
Moreover, $\Perop_2$ maps $H^1_+(\R)$ onto $H^1_+(\R/2\Z)$ and
$H^1_-(\R)$ onto $H^1_-(\R/2\Z)$.
\label{prop-1.002}
\end{prop}

\begin{proof}
By the triangle inequality and Fubini's theorem, $\Perop_2$ is
a contraction $L^1(\R)\to L^1(\R/2\Z)$:
\[
\int_{-1}^1 |\Perop_2f(x)|\diff x\le\sum_{j\in\Z}\int_{-1}^1|f(x+2j)|
\diff x=\sum_{j\in\Z}\int_{2j-1}^{2j+1}|f(x)|
\diff x=\int_\R|f(x)|\diff x,
\]
It remains to check the mapping properties, which are immediate from
the characterizations \eqref{eq-intcharH1}, \eqref{eq-intcharH1-2} for
the line and \eqref{eq-FourChar-H1-circle} for the circle, combined with
the calculation
\begin{equation}
\int_{-1}^1\e^{\imag\pi nt}\Perop_2 f(t)\diff t=\sum_{j\in\Z}
\int_{-1}^1\e^{\imag\pi nt}f(t+2j)\diff t =\int_\R \e^{\imag\pi nt}f(t)\diff t,
\qquad n\in\Z.
\label{eq-Pi2id1.1}
\end{equation}
The proof is complete.
\end{proof}

The identity \eqref{eq-Pi2id1.1} is a special case of a more general identity,
for $f\in L^1(\R)$ and $F\in L^\infty(\R/2\Z)$
(compare with \eqref{eq-extper1.01}):
\begin{multline}
\langle F,\Perop_2 f\rangle_{\R/2\Z}=
\int_{-1}^1 F(t)\Perop_2 f(t)\diff t=\sum_{j\in\Z}
\int_{-1}^1F(t)f(t+2j)\diff t
\\
=\int_\R F(t)f(t)\diff t=\langle F,f\rangle_\R,
\qquad n\in\Z.
\label{eq-Pi2id1.1'}
\end{multline}
We need to extend $\Perop_2$ in a natural fashion to the space $\Lspaceo$.
If $\varphi\in C^\infty(\R/2\Z)$ is a test function on the circle, we glance
at \eqref{eq-Pi2id1.1'}, and
for $u\in\Lspaceo$ with $u=f+\Hop g$, where $f\in L^1(\R)$ and
$g\in L^1_0(\R)$, we set
\begin{equation}
\langle \varphi, \Perop_2 u\rangle_{\R/2\Z}:=\langle \varphi,u\rangle_\R
=\langle\varphi,f\rangle_\R-\langle\tilde\Hop\varphi,g\rangle_\R.
\label{eq-Pi2id1.2}
\end{equation}
This defines $\Perop_2 u$ as a distribution on the circle (compare with
\eqref{eq-Hilbtrans-distr2}).

\begin{prop}
For $u\in\Lspaceo$ of the form $u=f+\Hop g$, where $f\in L^1(\R)$ and
$g\in L^1_0(\R)$, we have that  $\Perop_2 u=\Perop_2 f+\Hop_2\Perop_2 g$.
In particular, $\Perop_2$ maps $\Lspaceo\to\Lspaceoper$ continuously.
\label{prop-7.2.2}
\end{prop}

\begin{proof}
For a $2$-periodic test function $\varphi\in C^\infty(\R/2\Z)$, we check that
\[
\langle\varphi,\Perop_2 f+\Hop_2\Perop_2 g\rangle_{\R/2\Z}=
\langle\varphi,\Perop_2 f\rangle_{\R/2\Z}-\langle\Hop_2\varphi,
\Perop_2 g\rangle_{\R/2\Z}=\langle\varphi,f\rangle_{\R}-\langle\Hop_2\varphi,
g\rangle_{\R},
\]
where we applied the identity \eqref{eq-Pi2id1.1'} twice. If we compare this
with \eqref{eq-Pi2id1.2}, we realize we have the same expression, because
$\tilde\Hop\varphi$ and $\Hop_2\varphi$ differ by a constant. After all,
they are two harmonic conjugates of one and the same function, and $g$
annihilates constants.
\end{proof}

\section{The spanning problem formulation of Theorem
\ref{thm-2.0}}
\label{sec-spanning}

\subsection{A reformulation of Theorem \ref{thm-2.0}
}
\label{subsec-dualform}

Let us consider the following problem.

\begin{prob}
For which values of the positive real parameter $\beta$ is the linear span
of the functions
\[
e_n(t):=\e^{\imag\pi nt},\,\,e_m^{\langle\beta\rangle}(t):=
\e^{-\imag\pi\beta m/t},\qquad m,n\in\Z_{+,0},
\]
weak-star dense in $H^\infty_+(\R)$?
\label{prob-fund1}
\end{prob}

We first remark that the functions $\e^{\imag\pi nt}$ and $\e^{-\imag\pi\beta m/t}$
for $m,n\in\Z_{+,0}$ belong to $H^\infty_+(\R)$ (they have bounded holomorphic
extensions to $\C_+$), so that the problem makes sense.
A simple scaling argument allows us to take $\alpha:=1$, so that
\emph{Theorem \ref{thm-2.0} is equivalent to Problem \ref{prob-fund1}
having an affirmative answer if and only if} $\beta\le1$.

With respect to the dual action $\langle\cdot,\cdot\rangle_\R$ on the line,
the understood predual of $H^\infty_+(\R)$ is the quotient space
$L^1(\R)/H^1_+(\R)$. So, in terms of duality, the question raised in Problem
\ref{prob-fund1} is:
\emph{When, provided that $f\in L^1(\R)$, do we have the implication}
\begin{equation}
\langle e_n,f\rangle_\R=
\langle e_m^{\langle\beta\rangle},f\rangle_\R=0 \,\,\forall m,n\in\Z_{+,0}
\quad\Longrightarrow\quad f\in H^1_+(\R)?
\label{eq-dual1.1}
\end{equation}
The argument involving point separation in $\C_+$ from \cite{HM} applies
here as well, which makes $\beta\le1$ a necessary condition for the
implication \eqref{eq-dual1.1} to hold. Actually, as mentioned in the
introduction, the methods of \cite{CHM} supply infinitely many linearly
independent counterexamples for $\beta>1$.

Also, by testing with $n=0$, we note that we might as well assume that
$f\in L^1_0(\R)$ in \eqref{eq-dual1.1}.
In view of \eqref{eq-Pi2id1.1},
\begin{equation}
\langle e_n,f\rangle_\R=\int_{-1}^1 \e^{\imag\pi nt}\Perop_2 f(t)
\diff t=\langle e_n,\Perop_2f\rangle_{\R/2\Z},
\label{eq-duality1.001}
\end{equation}
so that for $f\in L^1(\R)$ we have the equivalence
\[
\Big\{\forall n\in\Z_{+,0}:\,\langle e_n,f\rangle_\R=0\Big\}
\quad\Longleftrightarrow\quad\Perop_2 f\in H^1_+(\R/2\Z).
\]
Since $\Jop_\beta^* e_m=e_m^{\langle\beta\rangle}$, where $\Jop_\beta^*$ is the
involutive operator studied in Subsections
\ref{subsec-invol} and \ref{subsec-7.1}, we have that
\[
\langle f, e_m^{\langle\beta\rangle}\rangle_\R=\langle f,\Jop_\beta^* e_m\rangle_\R
=\langle \Jop_\beta f,e_m\rangle_\R,
\]
which leads for $f\in L^1(\R)$ to the equivalence
\[
\Big\{\forall m\in\Z_{+,0}:\,\langle e_m^{\langle\beta\rangle},f\rangle_\R=0\Big\}
\quad\Longleftrightarrow\quad
\Perop_2\Jop_\beta f\in H^1_+(\R/2\Z).
\]
We can now rephrase the question \eqref{eq-dual1.1} and hence Problem
\ref{prob-fund1}.

\begin{prob} Fix $0<\beta\le1$.
Is it true that for $f\in L^1_0(\R)$,
\begin{equation*}
\Perop_2 f,\,\Perop_2\Jop_\beta f\in H^1_+(\R/2\Z)
\quad\Longrightarrow\quad f\in H^1_+(\R)?
%\label{eq-dual1.2}
\end{equation*}
 \label{prob-fund2}
\end{prob}

It is rather obvious that the reverse implication holds (use, e.g.,
Propositions \ref{prop-1.001} and \ref{prop-1.002}).
If we think of $\Perop_2 f$ and
$\Perop_2\Jop_\beta f$ as $2$-periodic ``shadows'' of $f$ and $\Jop_\beta f$,
the issue at hand is whether knowing that the two shadows are in the
right space we may conclude the function comes from the space $H^1_+(\R)$.
We note here that the main result of \cite{HM} may be understood as the
assertion that \emph{$f$ is determined uniquely by the two ``shadows''
$\Perop_2 f$ and $\Perop_2\Jop_\beta f$ if and only if $\beta\le1$}.
This offers some rather weak support for the plausibility of the implication
of Problem \ref{prob-fund2}.
%, especially taking into account that if we
%replace $\Lspaceo$ by $L^{1,\infty}(\R)$, Theorem 1.8.2 fails to be true
%in a drastic manner, cf. Remark \ref{example}.

\subsection{An alternative reformulation in terms of the space
$\Lspaceo$}
We begin with a function $f\in L^1_0(\R)$, and form the conjugate-analytic
Szeg\H{o} projection (cf. \eqref{eq-projform1})
\[
u:=\proj_-f=\frac{1}{2}(f-\imag \Hop f)\in L^1_0(\R)+\Hop L^1_0(\R)\subset
\Lspaceo.
\]
Then, by Definition \ref{defn-7.1.3},
\[
\Jop_\beta u=\Jop_\beta \proj_-f=\frac{1}{2}(\Jop_\beta f-\imag\Jop_\beta\Hop f)
=\frac{1}{2}(\Jop_\beta f-\imag\Hop\Jop_\beta f)\in L^1_0(\R)+\Hop L^1_0(\R)\subset
\Lspaceo,
\]
and we calculate that (use Lemma \ref{prop-7.2.2})
\begin{equation}
\Perop_2 u=\frac{1}{2}(\Perop_2f-\imag \Perop_2\Hop f)
=\frac{1}{2}(\Perop_2f-\imag \Hop_2\Perop_2 f)=
\frac{1}{2}(\id-\imag \Hop_2)\Perop_2 f\in\Lspaceoper,
\label{eq-8.1.3}
\end{equation}
and that (use Proposition \ref{prop-7.2.2} again)
\begin{equation}
\Perop_2\Jop_\beta u=\frac{1}{2}(\Perop_2\Jop_\beta f-\imag\Perop_2\Hop\Jop_\beta
f)=\frac{1}{2}(\Perop_2\Jop_\beta f-\imag\Hop_2\Perop_2\Jop_\beta
f)=\frac{1}{2}(\id-\imag\Hop_2)\Perop_2\Jop_\beta f
\in \Lspaceoper.
\label{eq-8.1.4}
\end{equation}
Here, we write $\id$ for the identity operator. Modulo the constants,
the operator $\proj_{2,-}:=\frac{1}{2}(\id-\imag\Hop_2)$ projects to the
$2$-periodic conjugate-holomorphic functions in the upper half-plane $\C_+$,
and $H^1_+(\R/2\Z)$ is indeed mapped to $\{0\}$:
\begin{equation}
\proj_{2,-}H^1_+(\R/2\Z)=\{0\}.
\label{eq-proj=0}
\end{equation}
Hence we conclude from \eqref{eq-8.1.3} and \eqref{eq-8.1.4} that
\begin{equation*}
\Perop_2 f,\,\Perop_2\Jop_\beta f\in H^1_+(\R/2\Z)
\quad\Longrightarrow\quad \Perop_2u=\Perop_2\Jop_\beta u=0.
%\label{eq-dual1.2}
\end{equation*}
We are led to consider the following problem.
Let $\Lspaces$ be the one-codimensional subspace of $\Lspaceo$ given by
\[
\Lspaces:=L^1_0(\R)+\Hop L^1_0(\R)\subset\Lspaceo.
\]

\begin{prob} Fix $0<\beta\le1$.
Is it true that for $u\in\Lspaces$,
\begin{equation*}
\Perop_2 u=\Perop_2\Jop_\beta u=0
\quad\Longrightarrow\quad u=0?
%\label{eq-dual1.2}
\end{equation*}
\label{prob-fund3}
\end{prob}

\begin{prop}
If the answer to Problem \ref{prob-fund3} is affirmative, then the answers
to Problems \ref{prob-fund1} and \ref{prob-fund2} are affirmative as well,
and the assertion of Theorem \ref{thm-2.0} is valid.
\label{prop-fund3.01}
\end{prop}

\begin{proof}
We already know that Problems \ref{prob-fund1} and \ref{prob-fund2} are
equivalent. Let $f\in L^1(\R)$ be such that $\Perop_2 f\in H^1_+(\R/2\Z)$ and
$\Perop_2\Jop_\beta f\in H^1_+(\R/2\Z)$. Then, as a first step,
$f\in L^1_0(\R)$ by the identity \eqref{eq-Pi2id1.1} with $n=0$.
We recall the notation $\proj_-:=\frac12(\id-\imag\Hop)$ for the Szeg\H{o}
projection to the conjugate-holomorphic functions in $\C_+$.
Next, we consider the distribution $u:=\proj_-f=\frac{1}{2}(f-\imag\Hop f)
\in\Lspaces$, and use the identities \eqref{eq-8.1.3} and \eqref{eq-8.1.4}
together with \eqref{eq-proj=0} to see that
$\Perop_2 u=\Perop_2\Jop_\beta u=0$.
Now, given that Problem \ref{prob-fund3} has an affirmative answer, we
have that $\proj_-f=u=0$, which is only possible for $f\in L^1(\R)$
if $f\in H^1_+(\R)$. We conclude that  Problems \ref{prob-fund1} and
\ref{prob-fund2} have affirmative answers as well.
Finally, given the discussion in Subsection \ref{subsec-quad1.01},
the correctness of the assertion of Theorem \ref{thm-2.0} follows as well.
\end{proof}

\subsection{The connection with an extension of ergodic theory}

In \cite{HMerg}, the following result is obtained as an application of an
extension of ergodic theory in the setting of Gauss-type maps.

\begin{thm}
{\rm(see \cite{HMerg})}
For $0<\beta\le1$ and $u\in\Lspaces$, the following implication holds:
\begin{equation*}
\Perop_2 u=\Perop_2\Jop_\beta u=0
\quad\Longrightarrow\quad u=0.
%\label{eq-dual1.2}
\end{equation*}
\label{thm-Erg1}
\end{thm}

Modulo this result, we may now conclude the proof of Theorem \ref{thm-2.0}.

\begin{proof}[Proof of Theorem \ref{thm-2.0}]
As observed right after the formulation of Theorem \ref{thm-2.0}, a scaling
argument allows us to reduce the redundancy and \emph{fix $\alpha=1$, in which
case the condition $0<\alpha\le1$ reads $0<\beta\le1$}. Now, in view of
the above Subsection \ref{subsec-dualform} and ensuing
Proposition \ref{prop-fund3.01}, the required assertion is an immediate
consequence of Theorem \ref{thm-Erg1}.
\end{proof}

It remains to explain how Theorem \ref{thm-Erg1} connects with an extension
of ergodic theory. The connection is strongest for $\beta=1$, which is why
we restrict our attention to this value of $\beta$. For $u\in\Lspaces$, we
need to show that if $\Perop_2 u=0$ and if $\Perop_2\Jop_1 u=0$, then $u=0$
is the only possibility. We split $\Perop_2=\id+\Superop_2$, so that
\[
\Superop_2 u(t)=\sum_{j\in\Z^\times}u(t+2j),
\]
where the two sides are to be understood liberally
(compare with \eqref{eq-Pi2id1.2}). Then $\Perop_2 u=0$ is the same as
$u=-\Superop_2u$, while $\Perop_2\Jop_1 u=0$ means that
$\Jop_1u=-\Superop_2\Jop_1 u$. Since $\Jop_1$ is an involution, we could write
the latter as $u=-\Jop_1\Superop_2\Jop_1 u$. We may need to be careful with
the interpretation of the right-hand side, but let us not worry about that now.
So, the two pieces of information we have about $u\in\Lspaces$ is that
$u=-\Superop_2u$ and $u=-\Jop_1\Superop_2\Jop_1 u$. We are free to combine
them:
\begin{equation}
u=\Superop_2\Jop_1\Superop_2\Jop_1 u\quad\text{and}\quad
u=\Jop_1\Superop_2\Jop_1\Superop_2 u.
\label{eq-twoways1}
\end{equation}
If we write $\Tope_1:=\Superop_2\Jop_1$ and $\Vop_1:=\Jop_1\Superop_2$,
\eqref{eq-twoways1} maintains that $u=\Tope_1^2u$ and $u=\Vop_1^2 u$.
The operator $\Tope_1$ behaves like the transfer operator associated
with the Gauss-type transformation $\tau_1(x)=\{-1/x\}_2$ (see, e.g.
\eqref{eq-Uop.Wop}), but to get a precise fit we need to restrict our space
of distributions to the symmetric standard interval $I_1$, and consider
$\LspaceIone$. Of course $\Tope_1$ acts contractively on the space $L^1(I_1)$
(see Proposition \ref{prop-contract1}), but on the larger space
$\LspaceIone$ it is no longer a norm contraction on the space (but it does
define a bounded operator), see \cite{HMerg}.
This is a serious complication, which is overcome only by a careful analysis
of the action of the iterates of the transfer operator on the Hilbert kernel.
We remark that on the interval $I_1$, the equality $u=\Tope_1^2u$ asks for
$u$ to be an ``invariant observable'' in the space $\LspaceIone$ of ``extended
observables'' for the composition square of the Gauss-type transformation. 
In the considerably simpler $L^1(I_1)$ setting, this is the same as being a 
scalar multiple of the invariant measure (this observation uses ergodicity). 
From a functional analysis perspective, in the case of a finite mass invariant
measure, ergodicity can be understood as the property that the given invariant
measure is an extreme point in the convex body of all the invariant probability
measures.
In the case at hand, the absolutely continuous invariant measure is
$(1-t^2)^{-1}\diff t$, which is ergodic but has infinite mass, so it does not
fit in the standard functional analysis interpretation.
Then we still would know from ergodicity that the only possible solution
to $u=\Tope_1^2u$ with $u\in L^1(I_1)$ is the function $u=0$
(see, e.g. \cite{HM}).
In this sense, the assertion that $u=0$ is the only possibility in the larger
space $\LspaceIone$ of extended observables is stronger than standard 
ergodicity. The analogue for a transformation without an indifferent fixed 
point would be the statement that the given invariant observable is unique 
up to scalar multiples within the extended observables space $\LspaceIone$.
We may think of the space $\LspaceIone$ as arising from a mix of absolutely 
continuous signed densities of two types of particles, (i) point particles 
(represented by $\delta_\xi$), and (ii) fuzzy particles (represented 
by $\Hop \delta_\xi$). In the fuzzy case, we need to include source 
points $\xi$ located outside the basic interval $I_1$; if we would prefer 
to consider only $\xi\in I_1$, the Hilbert transform
needs some slight modification to give the whole space $\LspaceIone$
in this manner.

%%% References %%%


\begin{thebibliography}{1}

\bibitem{PlM} \textit{Curvature of Nielsen's spiral}.
http://planetmath.org/encyclopedia/CurvatureOfNielsensSpiral.html.

\bibitem{Aar} Aaronson, J., \textit{An introduction to infinite ergodic
theory}. Mathematical Surveys and Monographs, \textbf{50}, American
Mathematical Society, Providence, RI, 1997.

%\bibitem{ADS1}
%Avdeeva, M., Li, D., Sinai, Ya. G.
%\textit{New Erd\H{o}s-Kac type theorems for signed measures on square-free
%integers}.
%J. Stat. Phys. \textbf{154} (2014), 327-333.

%\bibitem{ADS2}
%Avdeeva, M., Li, D., Sinai, Ya. G.,
%\textit{New limiting distributions for the M\"obius function}.
%Mosc. J. Comb. Number Theory \textbf{4} (2014), 3-51.

\bibitem{Bl} Blasi Babot, D.,
\textit{Heisenberg uniqueness pairs in the plane. Three parallel lines}.
Proc. Amer. Math. Soc. \textbf{141} (2013), 3899-3904.

%\bibitem{B} Beardon, A. F., \textit{The geometry of discrete groups}.
%Graduate Texts in Mathematics, 91. Springer-Verlag, New York, 1983.

%\bibitem{beur-coll2}
%Beurling, A., \textit{The collected works of Arne Beurling}. Vols. 1 and 2.
%\textit{Harmonic analysis}. Edited by L. Carleson, P. Malliavin, J.
%Neuberger and J. Wermer. Contemporary Mathematicians. Birkhäuser Boston,
%Inc., Boston, MA, 1989.

%\bibitem{BH}
%Borichev, A., Hedenmalm, H., \textit{Completeness of translates in weighted
%spaces on the half-line}. Acta Math. {\bf174} (1995), 1--84.

%\bibitem{Buf} Bufetov, A. I.,
%\textit{Limit theorems for translation flows}.
%Ann. of Math. (2) \textbf{179} (2014), 431-499.

\bibitem{CHM} Canto-Mart\'\i{}n, F., Hedenmalm, H., Montes-Rodr\'\i{}guez,
A., \textit{Perron-Frobenius operators and the Klein-Gordon equation}.
J. Eur. Math. Soc. (JEMS) \textbf{16} (2014), 31-66.

%\bibitem{C} Cellarosi, F., \textit{Renewal-type limit theorems for continued
%fractions with even partial quotients}, Ergodic Theory Dynam. Systems {\bf29}
%(2009), no. 5, 1451--1478.

%\bibitem{ChD} Chakraborty, P. S., Dasgupta, A., \textit{Invariant measure
%and a limit theorem for some generalized Gauss maps}.
%J. Theoret. Probab. {\bf17} (2004), no. 2, 387--401.

%\bibitem{ChR} Chakraborty, S., Rao, B. V., \textit{$\theta$-expansions and
%the generalized Gauss map.} Probability, statistics and their applications:
%papers in honor of Rabi Bhattacharya, 49--64,
%IMS Lecture Notes Monogr. Ser., {\bf41}, Inst. Math. Statist., Beachwood,
%OH, 2003.

\bibitem{CFS} Cornfeld, I. P., Fomin, S. V., Sina\u\i, Ya. G.,
\textit{ Ergodic theory.} Translated from the Russian.
Grundlehren der Mathematischen Wissenschaften [Fundamental Principles of
Mathematical Sciences], \textbf{245}. Springer-Verlag, New York, 1982.

\bibitem{Dur} Duren, P. L., \textit{Theory of $H^p$  spaces}.
Pure and Applied Mathematics, \textbf{38}. Academic Press, New York, 1970.

\bibitem{Fef} Fefferman, C.,
\textit{Characterizations of bounded mean oscillation}.
Bull. Amer. Math. Soc. \textbf{77} (1971), 587-588.

\bibitem{FeSt} Fefferman, C., Stein, E. M.,
\textit{$H^p$ spaces of several variables}.
Acta Math. \textbf{129} (1972), no. 3-4, 137-193.

%\bibitem{GM} Gilman, J., Maskit, B., \textit{An algorithm for $2$-generator
%Fuchsian groups}. Michigan Math. J. 38 (1991), no. 1, 13--32.

\bibitem{GaSa} Garcia, S. R., Sarason, D., \textit{Real outer functions}.
Indiana Univ. Math. J. \textbf{52} (2003), no. 6, 1397-1412.

\bibitem{Gar}
Garnett, J. B., \textit{Bounded analytic functions}. Revised first edition.
Graduate Texts in Mathematics, \textbf{236}. Springer, New York, 2007.

\bibitem{GirSriv}
Giri, D. K., Srivastava, R. K., \textit{Heisenberg uniqueness pairs for 
some algebraic curves in the plane}. Adv. Math. \textbf{310} (2017), 
993-1016. 

\bibitem{GrRy} Gradshteyn, I. S., Ryzhik, I. M.,
\textit{Table of integrals, series, and products}.
Corrected and enlarged edition edited by Alan Jeffrey. Incorporating the
fourth edition edited by Yu. V. Geronimus and M. Yu. Tseytlin. Translated
from the Russian. Academic Press, New York-London-Toronto, Ont., 1980.

\bibitem{GroJam} Gr\"ochenig, K., Jaming, Ph., \textit{The Cram\'er-Wold
theorem on quadratic surfaces and Heisenberg uniqueness pairs}. J. Inst.
Math. Jussieu, to appear.

\bibitem{HJ} Havin, V., J\"oricke, B., \textit{The uncertainty principle in
harmonic analysis}. Ergebnisse der Mathematik und ihrer Grenzgebiete (3)
[Results in Mathematics and Related Areas (3)], \textbf{28}.
Springer-Verlag, Berlin, 1994.

\bibitem{HM}   Hedenmalm, H., Montes-Rodr\'{\i}guez, A.,
\textit{Heisenberg uniqueness pairs and the Klein-Gordon equation}. Ann. of
Math. \textbf{173} (2011), 1507-1527.

\bibitem{HMerg}   Hedenmalm, H., Montes-Rodr\'{\i}guez, A.,
\textit{The Klein-Gordon equation, the Hilbert transform, and dynamics
of Gauss-type maps: $H^\infty$ approximation}. 
J. Anal. Math., to appear.

%\bibitem{H} Heisenberg, W., \textit{\"Uber den anschaulichen Inhalt der
%quantentheoretischen Kinematik und Mechanik}, Z. Physik 43 (1927), 172--198.

\bibitem{Hormbook} H\"ormander, L., \textit{The analysis of linear partial
differential operators. I. Distribution theory and Fourier analysis}.
Grundlehren der Mathematischen Wissenschaften [Fundamental Principles of
Mathematical Sciences], \textbf{256}. Springer-Verlag, Berlin, 1983.

%\bibitem{IK} Ionescu, A. D., Klainerman, S., {\em Uniqueness results for
%ill-posed characteristic problems in curved space-times}. Comm. Math. Phys.
%{\bf285} (2009), no. 3, 873--900.

%\bibitem{Olev-Ulan} Olevskii, A., Ulanovskii, A., \textit{On Beurling's
%sampling theorem in $\R^n$}. arXiv:1106.0576.

\bibitem{JK}   Jaming, P., Kellay, K.,
\textit{A dynamical system approach to Heisenberg uniqueness pairs}.
J. Anal. Math., to appear.

\bibitem{Katzbook} Katznelson, Y.,
\textit{An introduction to harmonic analysis}.
Second corrected edition. Dover Publications, Inc., New York, 1976.

%\bibitem{KRS}
%Kenig, C. E., Ruiz, A., Sogge, C. D., \textit{Uniform Sobolev inequalities
%and unique continuation for second order constant coefficient differential
%operators}. Duke Math. J. {\bf55} (1987), 329--347.

%\bibitem{Le} Lev, N., \emph{Uniqueness theorems for Fourier transforms}.
%Bull. Sci. Math. \textbf{135} (2011), 134-140.

%\bibitem{PK} Koosis, P., \textit{The logarithmic integral. I.}
%Cambridge Studies in Advanced Mathematics, {\bf12}.
%Cambridge University Press, Cambridge, 1988.

%\bibitem{KL} Kraaikamp, C., Lopes, A. O., \textit{The theta group and
%the continued fraction expansion with even partial quotients}, Geom. Dedicata
%59 (3), 1996, 293--333.

\bibitem{Lev} Lev, N., {\em Uniqueness theorems for Fourier transforms}.
Bull. Sci. Math. \textbf{135} (2011), 134-140.

\bibitem{Lin} Lin, M., \emph{Mixing for Markov operators}. Z.
Wahrscheinlichkeitstheorie Verw. Gebiete \textbf{19} (1971), 231-242.

%\bibitem{MP} Makarov, N., Poltoratski, A., \textit{Meromorphic inner
%functions, Toeplitz kernels and the uncertainty principle}. Perspectives in
%analysis, 185--252, Math. Phys. Stud., 27, Springer, Berlin, 2005.

\bibitem{MS} Matheson, A. L., Stessin, M. I., \textit{Cauchy transforms of
characteristic functions and algebras generated by inner functions}.
Proc. Amer. Math. Soc. \textbf{133} (2005), no. 11, 3361-3370.

\bibitem{MelTer}
Melbourne, I., Terhesiu, D., \textit{Operator renewal theory and mixing
rates for dynamical systems with infinite measure}.
Invent. Math. \textbf{189} (2012), no. 1, 61-110.

%\bibitem{Montes} Montes-Rodr\'\i{}guez, A.,
%\textit{Equations involving Perron-Frobenius operators}. In preparation.

%\bibitem{Nir} Nirenberg, L., \textit{Uniqueness in Cauchy problems for
%differential equations with constant leading coefficients}. Comm. Pure Appl.
%Math. {\bf10} (1957), 89--105.

%\bibitem{Pinkus} Pinkus, A., \textit{Totally positive matrices}.
%Cambridge Tracts in Mathematics, \textbf{181}. Cambridge University Press,
%Cambridge, 2010.

\bibitem{Pol} Poltoratski, A., \textit{Asymptotic behavior of arguments of
Cauchy integrals}. Linear and complex analysis, 133-144,
Amer. Math. Soc. Transl. Ser. 2, \textbf{226}, Amer. Math. Soc.,
Providence, RI, 2009.

\bibitem{PSZ} Poltoratski, A., Simon, B., Zinchenko, M.,
\textit{The Hilbert transform of a measure}, J. Anal. Math. \textbf{112}
(2010), 247-265.

%\bibitem{S1} Schweiger, F., \textit{Continued Fractions with odd and even
%partial quotients}, Arbeitsberichte Math. Institut Universt\"{a}t Salzburg,
%4, 1982, 59--70.

%\bibitem{S2} Schweiger, F., \textit{On the Approximation by Continued
%Fractions with Odd and Even Partial Quotients}, Arbeitsberichte Math.
%Institut Universt\"{a}t Salzburg, 1-2, 1984, 105--114.

\bibitem{Sj} Sj\"olin, P., \emph{Heisenberg uniqueness pairs and a theorem of
Beurling and Malliavin}, Bull. Sci. Math. \textbf{135} (2011), 123-133.

\bibitem{Sj2}  Sj\"olin, P.,
\emph{Heisenberg uniqueness pairs for the parabola}. J. Fourier Anal. Appl.
\textbf{19} (2013),  410-416.

\bibitem{Sriv} Srivastava, R. K., \textit{Non-harmonic cones are Heisenberg
uniqueness pairs for the Fourier transform on $\R^d$}. arXiv:1507.02624

\bibitem{Steinbook1}
Stein, E. M., \emph{Singular integrals and differentiability properties
of functions}. Princeton Mathematical Series, No. \textbf{30}.
Princeton University Press, Princeton, NJ, 1970.

\bibitem{Steinbook2} Stein, E. M., \emph{Harmonic analysis: real-variable
methods, orthogonality, and oscillatory integrals}.
With the assistance of Timothy S. Murphy. Princeton Mathematical Series,
\textbf{43}. Monographs in Harmonic Analysis, III. Princeton University
Press, Princeton, NJ, 1993.

\bibitem{SteinWeissbook} Stein, E. M., Weiss, G., \emph{Introduction to
Fourier analysis on Euclidean spaces}. Princeton Mathematical Series,
No. \textbf{32}. Princeton University Press, Princeton, NJ, 1971.

\bibitem{Thaler} Thaler, M., \emph{Transformations on $[0,1]$ with infinite
invariant measures}. Israel J. Math. \textbf{46} (1983), no. 1-2, 67-96.

\end{thebibliography}
\end{document}